 \documentclass[a4paper, 10pt, intlimits, sumlimits, final]{amsart}

\usepackage[utf8]{inputenc}

\usepackage{
amsmath , mathtools %% standard AMS packages and fixes
, amssymb , amsfonts %% math fonts and symbols
, amstext %% text in math
, esint %% integral signs
, xfrac %% fraction signs
, mathrsfs %% Extra fonts
, bbm , bm %% Give characteristic function without making the letters
}
\usepackage{dutchcal}

\usepackage[bookmarks, colorlinks, breaklinks, hypertexnames=false]{hyperref}
\hypersetup{linkcolor=blue, citecolor=blue, filecolor=blue, urlcolor=blue}

\usepackage{cleveref} %% Clever reference management

\usepackage{
tikz %% Tikz 
}
\usetikzlibrary{ 
trees
}
\usepackage{tcolorbox}

\newtheorem{theorem}{Theorem} \newtheorem{theorem*}{Theorem}
\newtheorem{definition}[theorem]{Definition}
\newtheorem{definition*}{Definition}
\newtheorem{lemma}[theorem]{Lemma}
\newtheorem{lemma*}{Lemma}
\newtheorem{corollary}[theorem]{Corollary}
\newtheorem{corollary*}{Corollary}
\newtheorem{proposition}[theorem]{Proposition}
\newtheorem{proposition*}{Proposition}
\theoremstyle{remark}

\newtheorem{remark*}[theorem]{Remark}
\newtheorem{claim}[theorem]{Claim}
\newtheorem{claim*}[theorem]{Claim}

\numberwithin{theorem}{section} %% Numbering of environments
\numberwithin{equation}{section} %% Numbering of eqns

\newcommand\eqnum{\stepcounter{equation}\tag{\theequation}}

\usepackage{enumerate} %% Fancy enumeration

\usepackage[notcite, notref]{showkeys} %% To show label
\renewcommand*\showkeyslabelformat[1]{%
  \expandafter\def\expandafter\UrlBreaks\expandafter{\UrlBreaks%  save the current one
  \do\a\do\b\do\c\do\d\do\e\do\f\do\g\do\h\do\i\do\j%
  \do\k\do\l\do\m\do\n\do\o\do\p\do\q\do\r\do\s\do\t%
  \do\u\do\v\do\w\do\x\do\y\do\z\do\A\do\B\do\C\do\D%
  \do\E\do\F\do\G\do\H\do\I\do\J\do\K\do\L\do\M\do\N%
  \do\O\do\P\do\Q\do\R\do\S\do\T\do\U\do\V\do\W\do\X%
  \do\Y\do\Z}%
\parbox[t]{1.2\marginparwidth}{\raggedleft\noindent\normalfont\small\nolinkurl{#1}\par}}

\usepackage[normalem]{ulem} %% multiple types of underline types
\usepackage[obeyFinal]{todonotes} %% margin todo notes
 \usepackage{array}
 \NewDocumentEnvironment { ltae } {b }
 {\begin{array}[t]{*{50}{>{\displaystyle}l}} #1 \end{array}}{}
\newcommand{\mc}[1]{\mathcal{#1}} % calligraphic math
\newcommand{\mf}[1]{\mathfrak{#1}} % fraktur math
\newcommand{\mbb}[1]{\mathbb{#1}} % blackboard math
\newcommand{\mb}[1]{\mathbf{#1}} % boldface math
\newcommand{\ms}[1]{\mathscr{#1}} % script math
\newcommand{\mrm}[1]{\mathrm{#1}} % math roman
\newcommand{\N}{\mathbb{N}} % naturals
\newcommand{\Z}{\mathbb{Z}} % integers
\newcommand{\Q}{\mathbb{Q}} % rationals
\newcommand{\R}{\mathbb{R}} % reals
\newcommand{\C}{\mathbb{C}} % complex numbers
\newcommand{\Orth}{\mathrm{O}} % orthogonal group
\newcommand{\Id}{\mathrm{id}} % identity operator
\newcommand{\diam}{\mathrm{diam}} % diameter
\newcommand{\spt}{\mathrm{spt}} % support
\newcommand{\LHS}[1]{\ensuremath{\mathrm{LHS}{#1}}}

\newcommand{\eqd}{\coloneqq} % assignment
\newcommand{\dd}{\mathrm{d}} % d of the differential
\newcommand{\st}{\, \colon\,} % such that
\newcommand{\Fourier}{\mathfrak{F}} % Fourier transform operator
\newcommand{\FT}[1]{\widehat{#1}} % Fourier transform symbol
\newcommand{\E}{\mathbb{E}} % expectation

\renewcommand{\bar}[1]{\overline{#1}}
\newcommand{\1}{\mathbbm{1}} % characteristic function
\newcommand{\sgn}{\mathrm{sgn}} % signum function
\newcommand{\mcl}{\mathclap}
\newcommand{\mrl}{\mathrlap}

\newcommand{\nquad}{\hspace{-1em}} % negative quad
\newcommand{\nqquad}{\hspace{-2em}} % negative qquad
\usepackage[citestyle=alphabetic, bibstyle=alphabetic, maxalphanames=5, backend=bibtex]{biblatex}
\newcommand{\Lap}{\ms{L}} %% Differential operator (Laplacian)
\newcommand{\mLap}{\mc{m}} %% Differential operator (Laplacian) symbols
\newcommand{\xs}{\mf{s}} %% regularity of the space X
\newcommand{\LP}{\mathop{\kern0pt \mathrm{P}}\hspace{-0.05em}\mathopen{}} %% Littlewood-Paley sector projection
\newcommand{\QP}{\mathop{\kern0pt \mathrm{Q}}\mathopen{}} %% Unit scale projection 
\newcommand{\UP}{\mathop{\kern0pt \mathrm{U}}\hspace{-0.05em}\mathopen{}} %% Directional projection
\newcommand{\rf}{\mf{f}} %% randomized initial data f
\newcommand{\rz}{\mf{z}} %% randomized \(z\)
\newcommand{\TT}{\mbb{T}} %% tree
\newcommand{\rn}{\mf{n}} %% randomization scale
\DeclareMathOperator{\conv}{conv}
\newcommand{\vertiii}[1]{{\left\vert\kern-0.25ex\left\vert\kern-0.25ex\left\vert #1 
    \right\vert\kern-0.25ex\right\vert\kern-0.25ex\right\vert}}

\begin{document}
\title{Probabilistic well-posedness of generalized cubic nonlinear Schr{ö}dinger equations with strong dispersion using higher order expansions}
%\author{Gennady Uraltsev}
\providecommand{\todoGU}[1]{\todo[color=red!40]{#1}} % comments by Gena
%\author{Itamar Oliveira}
\providecommand{\todoI}[1]{\todo[color=green!40]{#1}} % comments by Itamar
\providecommand{\todoJB}[1]{\todo[color=blue!40]{#1}}
\author[J.-B. Casteras]{Jean-Baptiste Casteras}
\address{CMAFcIO, Faculdade de Ci{\^e}ncias da Universidade de Lisboa,
Edificio C6, Piso 1, Campo Grande 1749-016 Lisboa, Portugal}
\email{jeanbaptiste.casteras@gmail.com}
%\tmnote{J.-B.C. supported by FCT - Fundacao para a Ciencia e a
%Tecnologia, under the project: UIDB/04561/2020}

\author[J. F{ö}ldes]{Juraj F{ö}ldes}
\address{Dept. of Mathematics, University of Virginia, Kerchof Hall,
Charlottesville, VA 22904-4137}
\email{foldes@virginia.edu}
%\tmnote{J. F. was partly supported by grant NSF-DMS-1816408}

\author[I. Oliveira]{Itamar Oliveira}
\address{School of Mathematics, The Watson Building, University of Birmingham, Edgbaston,
Birmingham, B15 2TT, England.}
\email{ i.oliveira@bham.ac.uk, oliveira.itamar.w@gmail.com}

\author[G. Uraltsev]{Gennady Uraltsev}
\address{Dept.\ of Mathematical Sciences, University of Arkansas,
  Fayetteville, AR 72701}
\email{gennady.uraltsev@gmail.com}

\renewcommand{\shorttitle}{Higher order expansion}

\begin{abstract}
In this paper, we study the local well-posedness of the cubic
Schr{ö}dinger equation
$$(i\partial_t + \Lap) u = \pm |u|^2 u \qquad \textrm{on} \quad \ I\times \R^d ,$$
with initial data being a Wiener randomization at unit scale of a given function $f$ and $\Lap$ being an operator of degree $\sigma\geq 2$. In particular, we prove that a solution exists almost-surely locally in time provided \(f\in H^{S}_{x}(\mathbb{R}^{d})\) with \(S>\frac{2-\sigma}{4}\) for \(d\leq \frac{3\sigma}{2}\), i.e. even if the initial datum is taken in certain negative order Sobolev spaces. The solutions are constructed as a sum of an explicit multilinear expansion of the flow in terms of the random initial data and of an additional smoother remainder term with deterministically subcritical regularity. 
We develop the framework of directional space-time norms to control the (probabilistic) multilinear expansion and the (deterministic) remainder term and to obtain improved bilinear probabilistic-deterministic Strichartz estimates. 
\end{abstract}

\date{}
\maketitle

\tableofcontents

\section{Introduction}

The goal of this paper is to study the probabilistic local well-posedness theory of the cubic generalized Schrödinger equation (NLS):
\begin{equation}
\label{eq:intro}
\begin{cases} (i\partial_t + \Lap) u = \pm |u|^2 u \qquad \textrm{on} \quad \ I\times \R^d ,
\\ u(0)=f \in H_x^S (\R^d) ,\, \end{cases}
\end{equation}
where $\Lap$ is defined via the Fourier multiplier $\mLap$ according to the formula
 \[
\Lap f(x) \eqd \int_{\R^d}e^{2\pi ix\xi}\mc{m}(\xi)\FT{f} (\xi)\dd\xi.
\]
We say that the symbol \(\mLap: \R^{ d}\to \R\) is of \emph{order} \(\sigma\) if $\mLap$ is a real-valued smooth
function and there are $C_{\max}, C_{\Lap} \geq 1$ such that the bounds 
\begin{equation}\label{eq:symbol-conditions}
\begin{aligned}
& |\partial^{\alpha}\mLap(\xi)\big|\leq C_{\!\Lap} |\xi|^{\sigma-|\alpha|} \quad\text{for all multi-indexes } \alpha \text{ with } |\alpha|\in\{1, 2, 3\},
\\ 
&\big|\nabla\mLap(\xi)\big|\geq C_{\!\Lap}^{-1} |\xi|^{\sigma-1}, \qquad\big|\det D^{2}\mLap(\xi)\big|^{\sfrac{1}{d}}\geq C_{\!\Lap}^{-1}|\xi|^{\sigma-2},
\end{aligned}
\end{equation}
hold for all $|\xi| \geq C_{\max}$. 

The Laplace operator is the prototypical example for $\Lap$ when $\sigma=2$, in which case \eqref{eq:intro} is the classical Schr\" odinger equation with a Kerr nonlinearity. Another operator satisfying our conditions is $\Lap=\Delta +(-\Delta)^{s^{\#}}$, $s^{\#}\in (0,1)$, recently considered in \cite{dipierro2024qualitative}. 

When $\sigma=4$ important examples of operators that fall under our framework are the 4NLS 
\[
( i\partial_t - \Delta^2 )u=-|u|^{2m} u ,
\]
and fourth order Schr\"odinger-type equations with mixed dispersion. The latter were studied by Karpman \cite{MR1396248} and later with Shagalov (see \cite{MR1779828} and the references therein). They proposed adding higher-order dispersive terms to regularize solutions to the classical Schr\" odinger equation, as an alternative to stabilization by saturation of the nonlinearity (see, for instance, \cite{MR727767}). Namely, they considered the equation
\[\eqnum\label{eq:4NLS}
( i\partial_t -\gamma \Delta^2 +\Delta)u=-|u|^{2m} u ,
\]
for small $\gamma>0$ and $m<\frac{d}{(d-4)_+}$, i.e. $m < \infty$ if $d\leq 4$ and $m<\frac{d}{d-4}$ if $d\geq 5$. %Thanks to this fourth order term, using a combination of stability analysis and numerical simulations, they showed that if $0 < dm < 4$ and $\gamma$ is small enough if $2 \leq d m < 4$, waveguide solutions are stable and when $dm > 4$, they become unstable. Their results highlight the existence of a second critical value $dm = 4$, which arises from the presence of the biharmonic term. 

This level of generality allows us to better understand the interplay between the strength of dispersion phenomena compared to the effect of the nonlinearity $\pm |u|^2 u$. Our setup includes previously considered generalized Nonlinear Schrödinger equations, namely \cite{IvaKos,Tur} studies the biharmonic NLS ($\Lap=\Delta^{2}$) in relation to stability of solitons in magnetic materials. In \cite{ohtzvwan1}, the authors studied the dynamics of the biharmonic NLS on $\mathbb{T}$.

The conditions in \eqref{eq:symbol-conditions} are satisfied by perturbed powers of the Laplacian, i.e. by operators of the form $\Lap = (-\Delta)^{\sigma / 2} +
\Lap^{\sharp}$ with $\Lap^{\sharp}$ a lower order
operator. \todoGU{Originally: We remark that $\sigma \geq 2$ is a technical assumption to avoid additional difficulties.  On the other hand, the well-posedness results follow from the deterministic theory if $\sigma > d$.}\todoGU{We should actually think about what happens for $S<0$. Could we ever cover $S<(d-\sigma)/2<0$?} When $\Lap = (-\Delta)^{\sigma/2}$ is the prototypical operator of order $\sigma$, the equation \eqref{eq:intro} is invariant under the transformation
\[\eqnum\label{eq:scaling}
u\mapsto u_\lambda (t,x)\eqd\lambda^{-\sfrac{1}{2}} u(\lambda^{-1} t ,\lambda^{-\sfrac{1}{\sigma}} x),\quad \lambda>0.
\]
For any $S \in \R$ and any $\lambda>0$ one can verify that the corresponding initial condition $u_\lambda (0,  x) = \lambda^{-\sfrac{1}{2}} f(\lambda^{-\sfrac{1}{\sigma}} x)$ satisfies 
\[
\big\|u_\lambda (0)\big\|_{\dot{H}^{S}_{x}(\R^{d})}:=\Big(\int_{\R^d} \Big|(-\Delta)^{S/2}\big(\lambda^{-\sfrac{1}{2}} f(\lambda^{-\frac{1}{\sigma}} x)\big) \Big|^2 dx\Big)^{\sfrac{1}{2}} = 
\lambda^{-\sfrac{S-\xs_c}{\sigma}}  \|f\|_{\dot{H}^{S}_{x}(\R^{d})}\,, 
\]
where the deterministic critical scaling exponent is denoted by $\xs_c \eqd \frac{d-\sigma}{2}$. 
%\todo{Please check the exponent of $\lambda$,  I am getting $S$ instead of $2S$ there.}\todoI{I think it is correct.}
Therefore, the homogeneous Sobolev 
norm $\dot{H}^S_x$  of the initial condition is conserved under the scaling transformation \eqref{eq:scaling} if and only if $S=\xs_c$,  which suggests that this regularity is critical for the well-posedness of \eqref{eq:intro}.

One can show that \eqref{eq:intro} is locally well-posed if $S\geq \xs_{c}$ using a Banach fixed point argument, see for instance \cite{MR2002047,MR1055532,MR2415387,MR2288737,MR2353631,MR2746203}. On the other hand, Christ, Colliander, and Tao \cite{christ2003ill} established that \eqref{eq:intro} is locally ill-posed in the regime $S<\xs_{c}$. In other words, on any time interval, one can always find an initial condition arbitrarily small in norm for which the solution to \eqref{eq:intro} becomes arbitrarily large. 

However, ill-posedness is probabilistically exceptional: equation \eqref{eq:intro} admits with probability one a solution on some interval of positive length when the initial data $f$ is distributed according to a probability measure $\mu$ for any $\mu$  from an appropriate class of measures on $H^S_x(\R^d)$, satisfying $\mu\big(H^{S+\epsilon}_x(\R^d)\big) = 0$ for any $\epsilon>0$, . In this paper, we show that this is the case for measures given by the unit scale Wiener randomization of a profile $f\in H^S_x(\R^d)$ when \( S> S_{\mrm{min}}\) with $S_{\mrm{min}}(d,\sigma )$ given by
\[\eqnum \label{eq:condS}
S_{\mrm{min}}(d,\sigma):=
\begin{cases} \frac{2-\sigma}{4} & \text{if } d\leq\frac{3\sigma}{2},
\\  \frac{d+2}{4}-\frac{5\sigma}{8} & \text{if } \frac{3\sigma}{2} < d \leq \frac{7\sigma}{2}-2 , \\
  \frac{d+2}{2}- \frac{3\sigma}{2} & \text{if } d > \frac{7\sigma}{2}-2.
\end{cases}
\]

\subsection{History and background} The first result on probabilistic well-posedness for cubic NLS on $\mathbb{T}$ and $\mathbb{T}^{2}$ was obtained by Bourgain \cite{MR1309539,MR1374420}, who set $\mu$ to be a Gaussian noise, i.e. a random function whose Fourier coefficients are weighted i.i.d. normal Gaussian variables.  The motivation for such randomization arises from the formal construction of the Gibbs measure,  where $\mu$ originates from the principal part of the Hamiltonian.\todoGU{Details? More precision?}  One can easily check that $\mu$ is invariant with respect to the dynamics of the homogeneous linear Schr\" odinger equation.  By proving the invariance of such a measure with respect to the \emph{nonlinear} flow, Bourgain showed that almost all initial data admit global in time solutions
(see also \cite{MR939505,MR4312285,MR4236191,MR2425134} for other results in this direction).

There is no canonical choice of $\mu$ when the underlying domain is $\R^n$ rather than $\mathbb{T}^{n}$. Different ways of choosing randomized initial data (i.e. different $\mu$) have since been introduced and studied (see for instance \cite{MR2425134,spitz2021almost,shen2021almost}). In this work, we consider the randomized initial data introduced in \cite{MR3350022}, known as the \textit{Wiener randomization}, which we recall in  \Cref{sec:randomization}. It closely resembles the construction of the Gaussian noise on the torus at a unit spatial scale. 

The first local well-posedness results for \eqref{eq:intro} in $\R^n$ using this construction of randomized initial data
were obtained by B\' enyi, Oh, and Pocovnicu \cite{MR3350022} where the authors show that there exists a unique solution 
of the form 
\[
u = e^{-it\Delta}f^\omega +C([-T,T] ; H^{\frac{d-2}{2}}(\R^d)) \subset C([-T,T] ; H^{S}(\R^d)) 
\]
 to \eqref{eq:intro} with $u(0)=f^\omega$, for $\mu$ a.e. $\omega$ as long as  $d\geq 3$ and $S>\frac{d-1}{d+1}\frac{d-2}{2}$. Their argument relies on finding the perturbation around the linear evolution $e^{-it\Delta}f^\omega $ as an appropriate fixed point on variants of the atomic $X^{s,b}$ spaces adapted to the variation spaces $V^p$ and \(U^{p}\) introduced by Koch, Tataru, and collaborators \cite{MR2526409, MR2824485, MR3618884}.\todoGU{I am confused about this description of the spaces. We should check.}
This result was improved by Shen, Soffer, and Wu \cite{shen2021almost2} to cover the range $S\geq \frac{1}{6}$ when $d=3$. They also obtained  global well-posedness and scattering (see also \cite{camps2021scattering} for a related conclusion). %When $P= \Delta^2 - \mu \Delta$, $\mu \geq 0$, $d\geq 5$ and $S>\max\big\{\frac{(d-1)(d-4)}{2(d+5)},\frac{d-4}{4} \big\}$, the local existence of \eqref{intro} was obtained in \cite{MR4214037}.
As the functional framework, the articles \cite{MR3350022, shen2021almost2} use the Bourgain $X^{s,b}$ spaces that were introduced in \cite{MR2094851, MR2526409, MR2824485, MR3618884}.
 
A different framework was pioneered by  Dodson, L\" uhrmann, and Mendelson \cite{dodson},  who used 
anisotropic norms introduced by Ionescu and Kenig \cite{ionescu1,ionescu2} to prove well-posedness for the Schrödinger map equation.  Specifically,  in \cite{dodson}
the probabilistic local well-posedness in \eqref{eq:intro} was established when $d=4$, $\Lap=-\Delta$ and $S>\frac{1}{3}$, and therefore improved \cite{MR3350022}. The paper \cite{dodson} also studies global existence and scattering, but we will not discuss these results since we do not treat these questions here. Using the same functional spaces, in \cite{CFU} the first, second, and fourth authors of the present manuscript generalized \cite{dodson} to operators $\Lap$ satisfying \eqref{eq:symbol-conditions} in arbitrary dimension. 
More precisely, in \cite{CFU} it is shown that if $S>\tilde{S}_{\min}(d,\sigma)$ with
\begin{equation}\label{asosbis}
\tilde{S}_{\min}(d,\sigma) = \dfrac{d - \sigma}{2}\times
\begin{cases}
 \frac{1}{3} & \text{if}\ \sigma \geq \frac{d + 2}{3} \,, \\
%\frac{1}{6} & \frac{d + 2}{3} <  s \leq \frac{d + 1}{2} \,, \\
 \frac{d+1- 2\sigma}{d-1}  & \text{if}\ \sigma \leq  \frac{d + 2}{3} \,,
\end{cases}
\end{equation}
then \eqref{eq:intro} is probabilistically locally well posed in the following sense: for almost every $\omega \in \Omega$, there exists an open interval $0\in I^\omega$ and a unique solution of the form
\[\eqnum\label{eq:foexs}
u(t) \in e^{-it\Lap}f^\omega +C(I^\omega; \dot{H}_x^{\frac{d-\sigma}{2}} (\R^d))
\]
to \eqref{eq:intro} on $I^\omega$.  Observe that
\[
\tilde{S}_{\min}(d,2) = \dfrac{d - 2}{2}\times
\begin{cases}
 \frac{1}{3} & \text{if}\ d=3 \,, \\
%\frac{1}{6} & \frac{d + 2}{3} <  s \leq \frac{d + 1}{2} \,, \\
 \frac{d-3}{d-1}  & \text{if}\ d\geq 4 \,, 
\end{cases}
\]
hence \cite{CFU} improves all works mentioned above, except for the case $S=\frac{1}{6}$ in $d=3$, $\Lap=\Delta$ proved in \cite{shen2021almost2}).

We remark that if the solution of \eqref{eq:intro} has the form \eqref{eq:foexs}, then $u^{\#}\eqd u(t) - e^{-it\Lap}f^\omega$ satisfies \eqref{eq:intro} with the nonlinearity $\pm|u|^2u$ on the right-hand side replaced by $\zeta_3+\mc{N}(u^\#)$ where 
\[
\zeta_3 := |e^{-it\Lap} f^\omega|^2 e^{-it\Lap} f^\omega,
\]
\[
\mc{N}(u^\#) := |u^\#+e^{-it\Lap} f^\omega|^{2}(u^\#+e^{-it\Lap} f^\omega)-|e^{-it\Lap} f^\omega|^2 e^{-it\Lap} f^\omega.
\]
Observe that $\mc{N}(u^{\#})$ is a combination of seven terms containing $u^{\#}$. Since the linear propagator $e^{-it\Lap}$ does not have a smoothing effect, to obtain $u^{\#} \in C(I; \dot{H}_x^{\frac{d-\sigma}{2}} (\R^d))$, we expect 
$\zeta_3 \in \dot{H}_x^{\frac{d-\sigma}{2}}$ for all small $t$.  If $|\hat{f}^\omega (\xi)| \approx |\xi|^\alpha$ for some $\alpha \in \R$, we then expect $\zeta_3(t) \approx  |\xi|^{3\alpha}$ (again by the fact that $e^{-it\Lap}$ has no smoothing effect), hence we need $3\alpha \geq \frac{d-\sigma}{2}$. This way, heuristically, the lowest possible regularity of $f$ is $\dot{H}_x^{\frac{d-\sigma}{6}}$. Observe that this is indeed the result of \cite{CFU} for $d = 3$. %In higher dimensions, nonlinear effects influenced the lowest regularity threshold.  \todoGU{Needs to be explained in more detail and more precisely} \todoJB{This was written by Juraj. I think he just meant that we do not expect $(d-\sigma)/6$ to be the lower bound in higher dimensions.}}

\subsection{Statement of main result and contextual remarks}  Our goal in the present work is to improve the differentiability of the initial data by performing a higher order asymptotic  expansion of the solution. This idea was introduced by B\' enyi, Oh, and Pocovnicu \cite{bop2} for $\Lap=\Delta$ and $d=3$ in the context of randomized NLS in $\mathbb{R}^{n}$ with unit scale Weiner randomized initial data. The authors observed that the strictest condition on $S$ was indeed induced by the term $\zeta_3$, and to eliminate it they looked for the solution  of the form $u = \mf{z}_1 + \mf{z}_3+u^{\#}$ where
 \[
 \begin{aligned}[t]
     \mf{z}_1 (t)= e^{-it\Delta} f^\omega, \qquad 
     \mf{z}_3=-i \int_0^t e^{i\Delta (t-t^\prime )} \zeta_3(t^\prime) dt^\prime .
 \end{aligned} 
 \]
The main result of  \cite{bop2}  asserts that for any $S>\frac{1}{5}$ and for almost every $\omega$, there exists an open interval $0\in I$ and a unique solution
\[
u(t) \in \mf{z}_1 +\mf{z}_3 +C(I; \dot{H}_x^{\frac{1}{2}} (\R^3))
\]
to \eqref{eq:intro}. 
By continuing the expansion to arbitrarily large (finite) order, \cite{bop2} shows probabilistic local well-posedness of \eqref{eq:intro} with $\Lap=-\Delta$ for any $S>\frac{1}{6}$.  
Curiously, $\frac{1}{6}$ is exactly the regularity threshold of the first order expansion,  but the authors of \cite{bop2} were not able to surpass it even with higher order expansions. By performing the latter in the framework of directional spaces, we were able in \cite{CFU2} to improve the result of \cite{bop2}  to $S>0$,  which is the entire open range of regularity where the nonlinearity can be made interpreted classically.  The paper \cite{CFU2} also contains regularity thresholds for $\Lap$ being Laplacian in any dimension,  which are presently optimal.    
In this work we aim to generalize \cite{CFU2} to a larger set of operators.  Our main theorem reads as follows.%\todoI{Should we include anything about scattering and other things?}

\begin{theorem}\label{thm:main}
Assume $\Lap$ satisfies \eqref{eq:symbol-conditions}. 
If  
$f\in H_{x}^{S} (\R^{d})$ with $S > S_{\min}(d,\sigma)$, where $S_{\min}(d,\sigma)$ is given by \eqref{eq:condS}, 
then there is an explicitly computable $\kappa_0 = \kappa_0(d, \sigma, S) \in \N$ such that for any $\kappa \geq \kappa_0$ the following holds. 
% and fix an integer $\kappa \geq 0$. Assume $f\in H_{x}^{S} (\R^{d})$ with $S > S(d, \kappa)$ and
% \[\eqnum \label{eq:condS}
% S(d, \kappa):=
% \begin{cases} \frac{d-\sigma}{2(\kappa +2)} & \text{if } d<\frac{5\sigma}{2}-2,
% \\ \max \{ \frac{d-\sigma}{2(\kappa +2)} \textcolor{blue}{- \frac{(\kappa + 1)(\sigma-2)}{4(\kappa + 2)}} , \frac{d+2}{4}-\frac{5\sigma}{8} \} & \text{if } \frac{5\sigma}{2}-2\leq d \leq \frac{7\sigma}{2}-2 , \\
% \max \{ \frac{d-\sigma}{2(\kappa +2)} \textcolor{blue}{- \frac{(\kappa + 1)(\sigma-2)}{4(\kappa + 2)}}, \frac{d+2-3\sigma}{2} \} & \text{if } d\geq \frac{7\sigma}{2}-2,
% \end{cases}
% \]
% Then, 
For a.e. $\omega \in \Omega$, there exist an open interval
$0\in I$ and a unique solution
\[
u(t) \in \sum_{j=0}^{\kappa} \mf{z}_{j}(t) +C(I; \dot{H}_{x}^{\frac{d-\sigma}{2}} (\R^{d}))
\]
to
\begin{equation}\label{eq:NLS-randomized}
    \begin{cases}
(i\partial_{t} + \Lap ) u = \pm |u|^{2} u & \text{on } I\times \R^{d} ,
\\
u(0)=f^\omega  \,,
\end{cases}
\end{equation}
where the random variables $z_j$ are defined in \eqref{eq:def:zn} and can be computed explicitly.  
\end{theorem}

To put Theorem \ref{thm:main} in context, observe that we recover the results of \cite{CFU2} for $\sigma=2$, but in a more general setting. On the other hand, we obtain new phenomena for higher order operators $\Lap$. For instance,  if \(\sigma=4\),  then
\[\eqnum \label{eq:condSbi}
S_{\min}(d, 4):=
\begin{cases} \frac{-1}{2} & \text{if } d\leq 6,
\\  \frac{d-8}{4} & \text{if } 6 < d \leq 12 , \\
  \frac{d-10}{2} & \text{if } d > 12.
\end{cases}
\]  
Thus,  when $d=5,6$ and $7$,  we allow for initial data in \textit{negative} Sobolev spaces without any renormalization\footnote{We refer to \cite{brun} for some well-posedness results for the one-dimensional cubic fractional nonlinear Schrödinger equations in negative Sobolev spaces}. More generally, we allow for initial data in negative Sobolev spaces if the dispersion is very strong in the corresponding dimension, more precisely, when $d< \frac{5\sigma}{2} -2$. %Additionally,  when $d=8$, that is when the equation is critical in the Sobolev sense,  we have $S_{\min}(8, \sigma)=0$,  that is, the initial data are allowed to belong to $H^\varepsilon$, for any $\varepsilon>0$\todoI{a bit repetitive with ``that is". I'm also confused about Sobolev criticality here. The end of this paragraph has to be polished}.  In comparison,  for $\sigma=2$, the critical dimension is $d=4$ and in this case, we only have $S_{\min}(4, \sigma)=\frac{1}{4}$ as in \cite{CFU2}. Another interesting point is the existence of transitional interval in super-critical dimensions in the Sobolev sense,  for example $8<d\leq 12$ for $\sigma = 4$,  which leads to better results compared to high dimensions, e.g.  $d > 12$.  The existence of such transition is quite surprising since it does not exist in the $\sigma=2$ case.  

\subsection{Outline of the argument} We say that $u\in C^{0}([0,T);H^{S}_{x}(\R^{d}))$ is a solution to \eqref{eq:intro} if it is a fixed point of the Duhamel iteration map
\begin{equation}\label{eq:dfge}
    u\mapsto e^{it\Lap}f \mp i\int_{0}^{t}e^{i(t-s)\Lap} |u(s)|^2 u(s) \dd s \,,
\end{equation}
with the sign in front of the integral being opposite to the sign of the right-hand side of \eqref{eq:intro}.

\begin{comment}
The proof of \Cref{thm:main} is based on a fixed point argument performed in appropriate spaces. Formally rewriting \eqref{eq:intro} gives
\begin{equation}\label{eq:dfge}
u (t) := e^{it\Lap}f \mp i\int_{0}^{t}e^{i(t-s)\Lap} |u(s)|^2 u(s) \dd s \,,
\end{equation}
the solution to which is is known as a \textit{mild solution}.
\end{comment}

\subsubsection*{First step} %To rigorously justify \eqref{eq:dfge}, \todoI{Do we have to justify the Duhamel formulation?} it is necessary to define the group operator $e^{it\Lap}$ and its mapping properties in appropriate functional framework based on directional spaces. Unlike all\todoI{This is a very strong word...} results in the literature \textbf{(REFERENCES?)}, 
%we do not assume that the symbol of $\Lap$ is radially symmetric, and therefore we have to carefully study the interaction between the special direction\todoI{Are really the directions that interact?} in the definition of the space and the principal directions associated 
%to $\Lap$. 

We start by defining directional norms. These norms will capture the directional behavior of the solutions, and will be used to set up a fixed point argument for proving the well-posedness of \eqref{eq:intro}. The directions appearing in the norms are closely tied to the eigenvectors of $D^2\Lap(xi)$, which in turn depend on the frequency $xi$. This additional technical challenges compared to \cite{CFU2} is addressed by discretizing the frequency domain into \textit{sectors}. For each sector, we select an adapted collection of orthogonal directions used to define the norms in play. 

\subsubsection*{Second step}
 Next, we prove the boundedness of the Schr\" odinger propagator $f\mapsto e^{it\Lap}f$ from $L^2$ into the spaces we introduced. We prove a directional maximal estimate (\Cref{prop:max-dir}) and the smoothing estimate (\Cref{prop:dir-local-smoothing}). The latter encodes a directional gain of derivatives central to our argument., whereas the former is the main novelty of this manuscript. The structure of the directional maximal estimate allows us to isolate the action of $e^{it\Lap}$ along a special direction, and we proceed by combining an analogous (local) one-dimensional bound (\Cref{thm:Shiraki1}) with a local-to-global result (\Cref{thm:rogers1}).

\subsubsection*{Third step} The next step is to establish mapping properties of 
\begin{equation}
u \mapsto \int_{0}^{t}e^{i(t-s)\Lap} |u(s)|^2 u(s) \dd s \,.
\end{equation}
c.f.  the second term in the right-hand side of \eqref{eq:dfge}. The estimates on such operators depend only on the properties of $e^{it\Lap}$ and classical techniques such as the Christ-Kiselev lemma.

Following the lines of \cite{CFU2}, we define two families of norms $X^{\sigma}$ and $Y^{\sigma}$ indexed by a regularity parameter $\sigma\in\mathbb{R}$. These spaces are more delicate here compared to \cite{CFU2} due to the choice of a different system of coordinates for each sector. %\textcolor{blue}{We remark that due to the stronger maximal estimates, we were able to slightly simplify the definition of our fundamental $X$ and $Y$ spaces.}

%The results in the linear theory and optimal\todoI{Why optimal?} choice of our $X$ and $Y$ spaces allowed us to exploit the mapping properties of 
%\begin{equation}\label{eq:nedp}
%u \mapsto \int_{0}^{t}e^{i(t-s)\Lap} |u(s)|^2 u(s) \dd s \,.
%\end{equation}\todoI{Isn't it a bit repetitive to have these two similar integral operator definitions?}

The spaces $Y^{\sigma}$ are well-suited to control the expansions $R_{\tau}\big[f, \ldots, f\big]$ (see \eqref{eq:def:R-inductive}) and $\mf{z}_j$, while the norms $X^{\sigma}$ are used to control the remainder term $v$ in the decomposition

\[
u(t)= \sum_{j=0}^{\kappa} \mf{z}_{j}(t) +v(t).
\]

In \Cref{cor:linear-estimate-adjoint-and-continuity}, we show that the boundedness in $X^{\sigma}$ and $Y^{\sigma}$ implies boundedness in $C^{0}([0,T]; H^{s}_{x}(\mathbb{R}^{d}))$, which relates our results to the classical setting.

Our main Theorem \ref{thm:main} will then follow from the fixed-point argument in \Cref{prop:v-solution}, which follows the lines of Theorem 1.6 of \cite{CFU2}, and we refer the reader to Subsection 1.3 of that paper for a detailed description of this machinery.

\noindent
\subsection*{Organization of the paper} In Section $2$, we introduce our sector decomposition and recall the unit-scale Wiener randomization. Section $3$ is devoted to the directional maximal and smoothing estimates. In Section $4$, we introduce our working spaces and establish some estimates for a linear non-homogeneous equation in them. We prove bilinear and multilinear estimates in Section $5$. In Section 6, we state probabilistic estimates on the $z_i$ terms. In Section $7$, we introduce our iteration scheme and prove, by combining the results obtained in Section $4$ and $5$, that a cubic nonlinear Schr\" odinger equation with a forcing term admits a solution. Finally, we prove our main Theorem \ref{thm:main} in Section $8$. The appendix collects technical results needed throughout the text.

\section{Notation and setup}

\begin{comment}
For a fixed number \(\sigma\geq2\), we say that a symbol
\(\Lap: \R^{ d}\to \R\) is of order \(\sigma\geq 2\) if it is a real-valued smooth
function and there is $N_{\min} > 0$ such that 
\[\eqnum\label{eq:symbol-conditions}
\begin{aligned}[t]
& |\partial^{\alpha}\Lap(\xi)\big|\lesssim |\xi|^{\sigma-|\alpha|} \quad\text{for all multi-indexes } \alpha \text{ with } |\alpha|\in\{1, 2, 3\},
\\ 
&\big|\nabla\Lap(\xi)\big|\gtrsim |\xi|^{\sigma-1}, \qquad\big|\det D^{2}\Lap(\xi)\big|\gtrsim|\xi|^{d(\sigma-2)},
\end{aligned}
\]
for all $|\xi| \geq N_{\min}$.  In this paper, the bounds are
allowed to depend on the implicit constants in
\eqref{eq:symbol-conditions} and on $N_{\min}$. An example of a symbol
that satisfies \eqref{eq:symbol-conditions} is $\Lap(\xi) = |\xi|^\sigma$.

Given a symbol $\Lap$, we use $\Lap$ to also denote the associated
operator given by 
\[
\Lap f(x) \eqd \int_{\R^d}e^{2\pi ix\xi}\Lap(\xi)\FT{f} (\xi)\dd\xi.
\]
The operator associated with $|2\pi\xi|^{\sigma}$ is denoted \(\Lap=|-\Delta|^{\sigma/2}\). 

\end{comment}

In this section, we introduce the notation and develop the technical framework in which the main estimates of the paper are established. Our analysis will be guided by certain spectral features of the symbol $\Lap$, and this will be reflected in the directions present in Propositions \ref{prop:max-dir} and \ref{prop:dir-local-smoothing}, which are the main novelties of this paper. We start by recalling the directional norms and by presenting the \textit{sector decomposition} of the Fourier space adapted to $\Lap$. The last part of this section recalls the unit-scale randomization of the initial data used in \cite{CFU2}.

\subsection{Directional norms}

Let $\Orth(d)$ be the set of $d\times d$ orthogonal matrices. We denote by \(\hat{e}_{j}\), \(j\in\{1,\ldots,d\}\), the \(j\)-th standard basis vector of \(\R^{d}\), and we define \(S_{j}\) to be the (orthogonal)
permutation matrix so that \(S_{j}\hat{e}_{1}=\hat{e}_{j}\) by setting
\[\eqnum\label{eq:dfsj}
S_{j}\hat{e}_k\eqd
\begin{dcases}
\hat{e}_{1} & \text{if } k=j , \\
\hat{e}_{j} & \text{if } k=1 , \\
\hat{e}_{k} & \text{if } k\notin\{1, j\}.
\end{dcases}    
\]

For any time interval $I \subseteq \R$, any orthogonal matrix
\(O\in\Orth(d)\), and any \(j\in\{1, \ldots, d\}\), the directional space-time
norm with exponents $ (\mf{a}, \mf{b}, \mf{c}) \in[1, \infty)^{3}$ is given
by
\[\eqnum\label{eq:def:directional-norm}
\begin{aligned}
&\|h\|_{L_{O, j}^{( \mf{a}, \mf{b}, \mf{c} )} (I)}=\Big\|h\circ(\Id_{I}\otimes O S_{j})\Big\|_{{L^{( \mf{a}, \mf{b}, \mf{c} )} (I)}}
\\
& =\Big( \int_{\R} \Big( \int_{I} \Big( \int_{\R^{d - 1}} \big|h (t, O S_{j}(x_1, x') \big|^{\mf{c}}  \dd {x'} \Big)^{\frac{\mf{b}}{\mf{c}}} \dd t \Big)^{\frac{\mf{a}}{\mf{b}}}  \dd x_1  \Big)^{\frac{1}{\mf{a}}} \,,
\end{aligned}
\]
where $x = (x_1, x')$ with $x_1 \in \R$ and $x' \in \R^d$.
If $\{ \mf{a}, \mf{b}, \mf{c} \} \ni \infty$ we use the standard modifications by the
essential supremum norm.

We refer to $L_{O, j}^{(\mf{a}, \infty, \mf{c} )} (I)$ and to
$L_{O, j}^{(\mf{a}, \mf{b}, \mf{c} )} (I)$ with \(\mf{b}\gg1\) as ``directional
maximal norms'' because of the supremum in \(t\in I\).  We refer to the
norms $L_{O, j}^{(\mf{a}, 2, \mf{c} )} (I)$ and
$L_{O, j}^{(\mf{a}, \mf{b}, \mf{c} )} (I)$ with \(\mf{b} \approx 2\) as ``directional
smoothing norms'' due to the negative power of \(N\) appearing below in
\eqref{eq:dir-local-smoothing-genbasis}, which manifests damping of
high oscillations.

Below, we use the directional norms to control the solution to NLS,
where the estimates are related to the principle curvature directions
of the symbol \(\Lap(\xi)\) and to the magnitude \(|\xi|\). To capture this
behaviour we define the collection of sectors and the associated
projectors.

\subsection{Sectors and associated projectors}
\label{sec:sectors}

In this section we introduce a covering of the Fourier space into subsets called \textit{sectors}, and on each sector we define a set of coordinates for which $\Lap$ has non-degenerate spectral properties. The sectors are carefully chosen to respect the scaling properties of $\Lap$. 

% For any sufficiently small parameter \(\epsilon_{\mrm{sec}}>0\) depending
% on \(d\) and \(\Lap\),  and chosen  and on the symbol , we introduce the following objects.

\begin{definition}[Sectors]\label{defn:sectors}
Fix a resolution \(\epsilon_{\Theta}\in(0, 1]\) 
and let
\(\ms{S}_{\Theta}\subset \mbb{S}^{d-1}\eqd\{\hat{e}\in\R^{d}\st \|\hat{e}\|=1\}\) be a
\emph{finite} collection of unit vectors, such that
\(\mbb{S}^{d-1}\subset\bigcup_{\hat{e}\in\ms{S}_{\Theta}}B_{\epsilon_{\Theta}}(\hat{e})\) and
\(\|\hat{e}-\hat{e}'\|\geq \epsilon_{\Theta}\) for any \(\hat{e}\neq\hat{e}'\in\ms{S}_{\Theta}\).
Then for every $N\in(1+\epsilon_{\Theta})^{\N}$ and any $\hat{e}\in\ms{S}_{\Theta}$ we define the sector by 
\[\eqnum
\label{eq:defn:sector}
\begin{aligned}[t]
&\theta_{N, \hat{e}}\eqd\Big\{\xi\in\R^{d}\st 1\leq\frac{|\xi|}{N}\leq (1+\epsilon_{\Theta})^{2}, \, \big| \sfrac{\xi}{|\xi|}-\hat{e} \big|\leq 2 \epsilon_{\Theta}\Big\}\quad N>1\,,
\\
&
\theta_{1, \hat{e}}\eqd\Big\{\xi\in\R^{d}\st |\xi|\leq (1+\epsilon_{\Theta})\Big\}\,,
\end{aligned}
\]
and we let $\Theta := \{\theta_{N, \hat{e}}: N\in(1+\epsilon_{\Theta})^{\N}, \hat{e}\in\ms{S}_{\Theta}\}$
be the collection of all sectors. 

To any sector \(\theta\in\Theta\) we associate
\(\hat{e}_{\theta}\in\ms{S}\) and \(N_{\theta}\in(1+\epsilon_{\Theta})^{\N}\) 
 such that
\(\theta=\theta_{N_{\theta}, \hat{e}_{\theta}}\). When $N_\theta=1$ we choose
$\hat{e}_{\theta}\in \ms{S}_{\Theta}$ arbitrarily. We define the \emph{center} of the sector
$\theta \in \Theta$ as 
\[
c_{\theta}\eqd
\begin{dcases} N_\theta\hat{e}_\theta
& \text{if } N_\theta>1,
\\
0 & \text{if } N_\theta=1.
\end{dcases}
\]
\end{definition}

Note that the union of sectors  in \(\Theta\) covers \(\R^{d}\) with finite overlap, that is, 
\(\sum_{\theta\in\Theta}\1_{\theta}\approx_{\epsilon_{\Theta}}1\), where $\1_{\theta}$ is the characteristic function of the set $\theta$. When
\(N=1\) the same sector is associated to all directions \(\hat{e}\in\ms{S}_{\Theta}\).
By construction, for any $N, N'\in(1+\epsilon_{\Theta})^{\N\setminus\{0\}}$ it holds that
\[
\theta_{N, \hat{u}}=\frac{N}{N'}\theta_{N', \hat{u}} := \Big\{\frac{N}{N'} \xi: \xi \in \theta_{N', \hat{u}}\Big\}.
\]
We also have 
\[\eqnum\label{eq:sector-diameter}
\diam(\theta)\lesssim_{d}
\begin{dcases}
N_\theta \epsilon_{\Theta} & \text{if } N_{\theta}>1,
\\
1 & \text{if } N_{\theta}=1,
\end{dcases}
\]
with an implicit constant depending only on the dimension.

Next, we introduce a smooth partition of unity
\(\Big\{\chi_{\theta}(\xi)\in C^{\infty}_{c}(\R^{d};[0, 1])\Big\}_{\theta\in\Theta}\) associated to the
sector in \(\Theta\) and  we use it to define sector
(approximate) projections by setting
\[\eqnum\label{eq:defn:sector-projections}
\FT{\LP_{\theta}f}(\xi)\eqd \chi_{\theta}(\xi)\FT{f}(\xi)\,.
\]

\begin{definition}[Sector projections]\label{defn:sector-projections}
Fix $\epsilon_{\Theta}$ and $\ms{S}_{\Theta}$
and let \(\Theta\) be a collection of sectors as in \Cref{defn:sectors}.
Let \(\chi\in C^{\infty}([0, \infty);\R)\) be a non-increasing function 
 with \(\chi\geq0\), \(\chi=1\) on
\([0, 1]\), \(\chi\geq1/2\) on \([0, \sqrt{1+\epsilon_{\Theta}}]\), and
\(\chi=0\) on \([1+\epsilon_{\Theta}, \infty)\). Then, define radial multipliers
\[
\chi_{N}^{\mrm{rad}}(\xi)
\eqd
\begin{dcases}
\chi\Big(\frac{|\xi|}{(1+\epsilon_{\Theta})N}\Big)-\chi\Big(\frac{|\xi|}{N}\Big) & \text{if } N>1,
\\
\chi\big(|\xi|\big) &\text{if } N=1
\end{dcases}
\]
and observe that the multipliers
\((\LP_{N}^{\mrm{rad}})_{N\in(1+\epsilon_{\Theta})^{\N}}\) form a partition of unity of
\(\R^{d}\).

Next, for any \(\hat{e}\in\ms{S}_{\Theta}\) define angular multipliers by setting 
\[
\begin{aligned}
& \chi_{\hat{e}}^{\mrm{ang}}(\xi)\eqd \frac{\chi\Big(\frac{\big|\sfrac{\xi}{|\xi|}-\hat{e}\big|}{2\epsilon_{\Theta}}\Big)}{\sum\limits_{\hat{e}'\in\ms{S}_{\Theta}}
\chi\Big(\frac{\big|\sfrac{\xi}{|\xi|}-\hat{e}'\big|}{2\epsilon_{\Theta}}\Big)} \,.
\end{aligned}
\]
 Note that the multipliers
\((\LP_{\hat{e}}^{\mrm{ang}})_{\hat{e}\in\ms{S}_{\Theta}}\) form
\(0\)-homogeneous smooth partition of unity of \(\R^{d}\setminus\{0\}\).

Finally, for each
sector \(\theta\) we define the multiplier
\[
\chi_{\theta}(\xi)
=
\begin{dcases}
\chi^{\mrm{rad}}_{N_{\theta}}(\xi) \chi^{\mrm{ang}}_{\hat{e}_{\theta}}(\xi)& \text{ if } N_{\theta}>1,
\\
\frac{1}{|\ms{S}_{\Theta}|}\chi^{\mrm{rad}}_{1}(\xi)& \text{ if } N_{\theta}=1.
\end{dcases}
\]
The multipliers \(\{\chi_{\theta}\}_{\theta\in\Theta}\) form a smooth partition of unity of \(\R^{d}\), with \(\chi_{\theta}\) supported in \(\theta\).

We abuse the notation and denote $P_\theta$ the operators with the Fourier multiplier $\chi_\theta$.
\end{definition}

\begin{lemma}\label{lem:sector-symbol-properties}
Let \(\conv(\theta)\) be the convex hull of \(\theta\). For any \(\xi\in\conv(\theta)\), it holds that 
\[\eqnum \label{eq:sector-points-bounds}
\begin{aligned}[t]
&
(1+\epsilon_{\theta})^{-1}N_{\theta}=(1+\epsilon_{\theta})^{-1}|c_{\theta}|\leq|\xi|\leq(1+\epsilon_{\Theta})^{2}N_{\theta},
\end{aligned}
\]
\[\eqnum \label{eq:sector-points-difference-bounds}
\begin{aligned}[t]
&
\big|\sfrac{\xi}{|\xi|}-\sfrac{c_{\theta}}{|c_{\theta}|}\big|\lesssim\epsilon_{\Theta},
\qquad |\xi-c_{\theta}|\lesssim\epsilon_{\Theta}N_{\theta},
\end{aligned}
\]
\[\eqnum \label{eq:sector-Lap-bounds}
\begin{aligned}[t]
& C_{\!\Lap}^{-1}N_{\theta}^{\sigma-1}\lesssim \big|\nabla\mLap(\xi)\big|\lesssim C_{\!\Lap} N_{\theta}^{\sigma-1},
\quad
C_{\!\Lap}^{-1}N_{\theta}^{\sigma-2}\lesssim \big\|D^{2}\mLap(\xi)\big\|\lesssim C_{\!\Lap} N_{\theta}^{\sigma-2},
\end{aligned}
\]
\[\eqnum \label{eq:sector-Lap-bounds-compare}
\begin{aligned}[t]
&\Big|\nabla\mLap(\xi)-\nabla\mLap(c_{\theta})\Big|\lesssim \epsilon_{\Theta}C_{\!\Lap} N_{\theta}^{\sigma-1},
\quad
\Big\|D^{2}\mLap(\xi)-D^{2}\mLap(c_{\theta})\Big\|\lesssim\epsilon_{\Theta}C_{\!\mLap}N_{\theta}^{\sigma-2}.
\end{aligned}
\]
and 
\[\eqnum \label{eq:sector-spectrum-lower-bounds}
\Big|\big\langle \hat{e}_{j};O^{T}D^{2}\mLap(\xi)O\hat{e}_{j}\big\rangle\Big|\gtrsim C_{\Lap}^{-2d+1}N_{\theta}^{\sigma-2}.
\]
if \(O^{T}D^{2}\Lap(\xi)O\) is diagonal.
\end{lemma}
\begin{proof}
Let \(\theta=\theta_{N,\hat{e}}\). The first bound of \eqref{eq:sector-points-difference-bounds} follows immediately from the defintion \eqref{eq:defn:sector} since \(\sfrac{c_{\theta}}{|c_{\theta}|}=\hat{e}\). The second bound follows since \(|c_{\theta}|=N_{\theta}\)  and \eqref{eq:defn:sector} gives that \(0<\sfrac{|\xi|}{|c_{\theta}|}-1<3\epsilon_{\Theta}\). The bounds \eqref{eq:sector-Lap-bounds} follow immediately from \eqref{eq:symbol-conditions} since \(\big|\nabla\mLap(\xi)\big|\approx \sum_{j=1}^{d}|\partial_{j}\mLap(\xi)|\) and \(\big\|D^{2}\mLap(\xi)\big\|\approx \sum_{i,j=1}^{d}|\partial_{i,j}\mLap(\xi)|\). The bounds \eqref{eq:sector-Lap-bounds-compare} follow by applying the fundamental theorem of calculus:
\[
|\partial_{j}\mLap(\xi)-\partial_{j}\mLap(c_{\theta})|\lesssim\sup_{\xi'\in\theta}\sum_{i\in\mrl{\{1,\ldots,d\}}}\big|\partial_{i,j}\mLap(\xi')\big||\xi-c_{\theta}|
\]
and
\[
|\partial_{i,j}\mLap(\xi)-\partial_{i,j}\mLap(c_{\theta})|\lesssim\sup_{\xi'\in\theta}\sum_{k\in\mrl{\{1,\ldots,d\}}}\big|\partial_{k,i,j}\mLap(\xi')\big||\xi-c_{\theta}|.
\]
Finally,  we prove bound \eqref{eq:sector-spectrum-lower-bounds}. Let \(\lambda_1, \ldots, \lambda_d\) be the eigenvalues of \(D^{2}\mLap(\xi)\), listed with multiplicity.
Since $ D^2\mLap(\xi)$ is symmetric, it holds that \(\bigl|\det D^2\mLap(\xi)\bigr| =\allowbreak |\lambda_1| \cdots |\lambda_d| \). Using the bound from below on \(\det D^2\mLap(\xi)\) in condition \eqref{eq:symbol-conditions} we obtain
\[
N_{\theta}^{d(\sigma-2)}C_{\Lap}^{-d}\lesssim \big|\det D^2\mLap(\xi)\big|=\min_{j\in\{1,\ldots,d\}}|\lambda_{j}|(\max_{j\in\{1,\ldots,d\}}|\lambda_{j}|)^{d-1}.
\]
Since \(\big\|D^{2}\mLap(\xi)\big\|=\max_{j\in\{1,\ldots,d\}}|\lambda_{j}|\), \eqref{eq:sector-Lap-bounds} allows us to deduce that that
\[
\min_{j\in\{1,\ldots,d\}}|\lambda_{j}|\gtrsim N_{\theta}^{(\sigma-2)} C_{\Lap}^{-2d+1}\,,
\]
which is the required bound since \(\LHS{\eqref{eq:sector-spectrum-lower-bounds}}\geq \min_{j\in\{1,\ldots,d\}}|\lambda_{j}|\).
\end{proof}

Next, we associate to every sector an orthonormal basis that \textit{almost} diagonalizes \(D^2\Lap(\xi)\), but that has specified non-degeneracy properties.  Recall that $\Orth(d)$ is the set of orthogonal $d\times d$ matrices.

\begin{definition}[Sector basis collection]\label{defn:sector-basis-collection}
Let \(\Theta\) be the collection of sectors and \(\mLap\) a symbol. Fix
\(C_{\ms{O}}\geq1\). The mapping \(\ms{O}: \theta\in\Theta\mapsto \ms{O}_\theta \in \Orth(d)\) is a sector basis choice adapted to \(\Theta\) and \(\mLap\) if the range \(\ms{O}(\Theta)\subset O(d)\) is finite and the following conditions hold for all \(\theta\in\Theta\) with $N_\theta \geq 2N_{\min}$ (see \eqref{eq:symbol-conditions}), for any \(j\in\{1, \ldots, d\}\) and any $\xi \in \theta$: \todoGU{Check that \(N_{\theta}<N_{\min}\) is not necessary}
\[\eqnum \label{eq:sector-basis-hessian-condition}
C_{\ms{O}}   N_{\theta}^{\sigma-2} \leq\Big|\Big\langle \ms{O}_{\theta}\hat{e}_{j};D^2\mLap(\xi)\ms{O}_{\theta}\hat{e}_{j} \Big\rangle\Big|\leq C_{\ms{O}} N_{\theta}^{\sigma-2}
\]
\[\eqnum\label{eq:sector-basis-gradient-condition}
\begin{aligned}
  C_{\ms{O}} ^{-1} N_{\theta}^{\sigma-1}\leq\big|\big\langle\nabla \mLap(\xi);\ms{O}_{\theta} \hat{e}_{j}\big\rangle\big|\leq C_{\ms{O}}  N_{\theta}^{\sigma-1},
\end{aligned}\]
\[\eqnum\label{eq:sector-basis-projection-condition}
\big|\big\langle\xi ;\ms{O}_{\theta} \hat{e}_{j}\big\rangle\big|\gtrsim C_{\ms{O}} ^{-1} N_{\theta}\,. 
\]
\end{definition}

The next proposition is proved in  \Cref{sec:prop-sector-basis-proof} and it captures key technical aspects of the directional maximal estimates \eqref{eq:directional-maximal} and \eqref{eq:directional-maximal-unit-FT-support}.\todoGU{Generic sentence}

\begin{proposition}[Sector basis]\label{prop:sector-basis-exists}
Let \(\mLap\) be a symbol satisfying \eqref{eq:symbol-conditions} and let $C_{\ms{O}} \gtrsim_{d} C_{\!\Lap}^{2d+2}$. For any sector collectection \(\Theta\) with resolution $\epsilon_{\Theta} \lesssim_{d} C_{\!\Lap }^{-(2d+3)}$ there exists a sector basis choice \(\ms{O}\) as in \Cref{defn:sector-basis-collection}.
\end{proposition}

\begin{proof} See the following subsection.
\end{proof}

\subsection{Proof of \Cref{prop:sector-basis-exists}\label{sec:prop-sector-basis-proof}}

The proof of \Cref{prop:sector-basis-exists} relies mainly on applying the following linear algebra fact.

\begin{proposition} \label{prop:finite-basis}
Fix \(M\in\N\) and \(r\in(0,1/2)\). There exists a finite collection \(\ms{A}''\subset O(d)\) such that
for any \(O\in O(d)\) and any pair of vectors \((v_{1},v_{2})\in\R^{d}\times\R^{d}\) there exists \(A\in\ms{A}''\) such that \(\|A-O\|\leq r\) and
\[\eqnum\label{eq:non-vanishing-projections-3}
|\langle   v_{k}; A \hat{e}_{j}\rangle|\gtrsim_{M}r|v_k| \qquad \forall j\in\{1,\ldots,d\} 
\]
for \( k\in\{1,2\}\).
\end{proposition}
We postpone the proof to the end of the section and concentrate on deducing \Cref{prop:sector-basis-exists} from  \Cref{prop:finite-basis}. For each sector \(\theta\in\Theta\) with \(N_{\theta}>N_{\min}\) we select \(\ms{O}_{\theta}\) by applying \Cref{prop:finite-basis} to the orthogonal change of basis matrix that diagonlizes \(D^{2}\Lap(c_{\theta})\) and requiring that \eqref{eq:non-vanishing-projections-3} hold for the two vectors \(\ms{V}=\big(c_{\theta},\nabla\Lap(c_{\theta})\big)\). The rest of the argument focuses on showing that the conditions of \Cref{prop:sector-basis-exists} are stable when passing from \(c_{\theta}\) to some other \(\xi\in\theta\), as long as the geometric parameters \(\epsilon_{\Theta}\) and \(r\) are chosen to be suitably small.

Let \(\ms{A}''\) be as in \Cref{prop:finite-basis}, and fix \(r\in(0,1/2)\) to be established later. For every sector \(\theta\in\Theta\) such that \(N_{\theta}\geq N_{\min}\) let \(O_{\theta}\in O(d)\) be such that \(O_{\theta}^{T}D^{2}\Lap (c_{\theta})O_{\theta}\) is diagonal and such that \(\langle\nabla\Lap(c_{\theta}) ;O_{\theta}\hat{e}_{1}\rangle\gtrsim|\nabla\Lap(c_{\theta})|\). The latter condition can be guaranteed by composition with a permutation matrix.
To define \(\ms{O}_{\theta}\in\ms{A}''\) required by \Cref{prop:sector-basis-exists} we set \((v_{1},v_{2})=(c_{\theta},\nabla\Lap(c_{\theta}))\) and apply \Cref{prop:finite-basis} to allow us to select \(\ms{O}_{\theta}\in\ms{A}''\) such that 
\(\|\ms{O}_{\theta}-O_{\theta}\|<r\) and such that \eqref{eq:non-vanishing-projections-3} holds. This definition guarantees that \(\ms{O}(\Theta)\) is finite. For the remainder of the proof of \Cref{prop:sector-basis-exists} we fix \(\theta\in\Theta\) with \(N_{\theta}\geq2N_{\min}\), we always assume that \(\xi\in\conv(\theta)\) and we check that all the other conditions hold.

In the proof below we use \(C_{1,d}\) and \(C_{2,d}\) to indicate constants depending only on the dimension, arising from the bounds established above. We will show that it is sufficient that the parameter \(r\), used in defining \(\ms{A}''\) and thus \(\ms{O}\), satisfy \(r=C_{d}C_{\!\Lap}^{-2d-1}\) for some sufficiently small dimensional constant \(C_{d}\in(0,1)\). The constant in the definition of the sector needs to satisfy \(\epsilon_{\Theta}=C_{d}^{2}C_{\!\Lap}^{-2d-3}\). 

\textbf{Proof of \eqref{eq:sector-basis-hessian-condition}.}
The upper bound in \eqref{eq:sector-basis-hessian-condition}, 
follows using \eqref{eq:sector-Lap-bounds} by the definition of the operator norm:
\[
\begin{aligned}[t]
\Big|\big\langle \ms{O}_{\theta} \hat{e}_{j};D^2\mLap(\xi)\ms{O}_{\theta}\hat{e}_{j} \big\rangle\Big|
&\leq \big\| D^2\mLap(\xi)\big\||\ms{O}_{\theta} \hat{e}_{j}|^{2}\lesssim C_{\!\Lap}N_{\theta}^{\sigma-2}\,.
\end{aligned}
\]
Next we prove the lower bound in \eqref{eq:sector-basis-hessian-condition}.  Setting \(Q_{\theta}\eqd\ms{O}_{\theta}-O_{\theta}\), we obtain that
\[
\begin{aligned}[t]
\Big|\Big\langle \ms{O}_{\theta} \hat{e}_{j};D^2\mLap(\xi)\ms{O}_{\theta}\hat{e}_{j} \Big\rangle\Big|
 &\geq  \Big|\Big\langle \ms{O}_{\theta} \hat{e}_{j};D^2\mLap(c_{\theta}) \ms{O}_{\theta}\hat{e}_{j} \Big\rangle \Big|
 - \big\|D^2\mLap(\xi)-D^2\mLap(c_{\theta})\big\|
 \\
 &
\nquad \geq
\begin{aligned}[t]
\qquad &\nqquad \Big|\big\langle O_{\theta} \hat{e}_{j};D^2\mLap(c_{\theta}) O_{\theta} \hat{e}_{j} \big\rangle \Big|
- \Big|\big\langle Q_{\theta}\hat{e}_{j};D^2\mLap(c_{\theta}) O_{\theta} \hat{e}_{j} \big\rangle \Big|
\\ & 
-\Big|\big\langle O_{\theta} \hat{e}_{j};D^2\mLap(c_{\theta}) Q_{\theta} \hat{e}_{j} \big\rangle \Big|
 - \big\|D^2\mLap(\xi)-D^2\mLap(c_{\theta})\big\|
\end{aligned}
\end{aligned}
\]
Since \(O_{\theta}^{T}D^2\mLap(c_{\theta}) O_{\theta} \) is diagonal \eqref{eq:sector-spectrum-lower-bounds} gives that 
\[
\begin{aligned}[t]
\Big|\big\langle O_{\theta} \hat{e}_{j};D^2\mLap(c_{\theta}) O_{\theta}\hat{e}_{j} \big\rangle \Big| \gtrsim C_{\!\Lap}^{-2d-1} N_{\theta}^{\sigma-2}.
\end{aligned}
\]
By construction \(\|Q_{\theta}\|<r\), so \eqref{eq:sector-Lap-bounds} gives that 
\[
\begin{aligned}[t]
\Big|\big\langle Q_{\theta}\hat{e}_{j};D^2\mLap(c_{\theta}) O_{\theta} \hat{e}_{j} \big\rangle \Big|+
\Big|\big\langle O_{\theta} \hat{e}_{j};D^2\mLap(c_{\theta}) Q_{\theta} \hat{e}_{j} \big\rangle \Big|
\lesssim2r C_{\!\Lap} N_{\theta}^{\sigma-2}.
\end{aligned}
\]
Finally, \eqref{eq:sector-Lap-bounds-compare} gives that \(\big\|D^2\mLap(\xi)-D^2\mLap(c_{\theta})\big\|
\lesssim\epsilon_{\Theta}C_{\!\Lap} N_{\theta}^{\sigma-2}\). Combining these estimates we obtain that
\[
\begin{aligned}[t]
\Big|\Big\langle \ms{O}_{\theta} \hat{e}_{j};D^2\mLap(\xi)\ms{O}_{\theta}\hat{e}_{j} \Big\rangle\Big|
& \geq C_{1,d}  C_{\!\Lap}^{-2d+1} N_{\theta}^{\sigma-2} \big(1-(2r+\epsilon_{\Theta})C_{2,d}C_{1,d}^{-1}C_{\!\Lap}^{2d}\big)
\\
& \gtrsim C_{\!\Lap}^{-2d+1} N_{\theta}^{\sigma-2}\,,
\end{aligned}
\]
as required. The two constants \(C_{1,d}\) and \(C_{2,d}\) depend only on the dimension while the last bound holds as long as \(2r+\epsilon_{\Theta}\leq \frac{C_{1,d}}{2C_{2,d} }C_{\!\Lap}^{-2d}\).

\textbf{Proof of \eqref{eq:sector-basis-gradient-condition}.}
If \(j=1\), setting \(Q_{\theta}\eqd\ms{O}_{\theta}-O_{\theta}\), we obtain that we have
\[
\begin{aligned}[t]
    \big|\big\langle\nabla \mLap(\xi);\ms{O}_{\theta} \hat{e}_{1}\big\rangle\big|
    &\geq
    \big|\big\langle\nabla \mLap(c_\theta);O \hat{e}_{1}\big\rangle\big|
    -\big|\big\langle\nabla \mLap(c_\theta);Q_{\theta} \hat{e}_{1}\big\rangle\big|
    \big|\big\langle\nabla \mLap(\xi)- \nabla\mLap(c_\theta);\ms{O}_{\theta} \hat{e}_{1}\big\rangle\big|.
\\ &\geq rC_{1,d}|\nabla \mLap(c_\theta)| - \big|\nabla \mLap(\xi)- \nabla\mLap(c_\theta)\big|.
    \\ &
    \geq  C_{1,d}C_{\!\Lap}^{-1}N_{\theta}^{\sigma-1} - r C_{2,d} C_{\!\Lap}N_{\theta}^{\sigma-1}
    -\epsilon_{\Theta} C_{2,d}C_{\!\Lap}N_{\theta}^{\sigma-1}
    \gtrsim  C_{\!\Lap}^{-1}N_{\theta}^{\sigma-1}\,.
\end{aligned}
\]
The two constants \(C_{1,d}\) and \(C_{2,d}\) depend only on the dimension while the last bound holds as long as \(r+ \epsilon_{\Theta}\leq \frac{C_{1,d}}{2C_{2,d}}C_{\!\Lap}^{-2}\).

If \(j\neq1\) we use \eqref{eq:sector-Lap-bounds-compare} and \eqref{eq:non-vanishing-projections-3}, coming from the construction of \(\ms{O}\), we obtain that
\[
\begin{aligned}[t]
    \big|\big\langle\nabla \mLap(\xi);\ms{O}_{\theta} \hat{e}_{j}\big\rangle\big|
    &\geq
    \big|\big\langle\nabla \mLap(c_\theta);\ms{O}_{\theta} \hat{e}_{j}\big\rangle\big|
    - 
    \big|\big\langle\nabla \mLap(\xi)- \nabla\mLap(c_\theta);\ms{O}_{\theta} \hat{e}_{j}\big\rangle\big|.
\\ &\geq rC_{1,d}|\nabla \mLap(c_\theta)| - \big|\nabla \mLap(\xi)- \nabla\mLap(c_\theta)\big|.
    \\ &
    \geq  r C_{1,d}C_{\!\Lap}^{-1}N_{\theta}^{\sigma-1} - \epsilon_{\Theta} C_{2,d}C_{\!\Lap}N_{\theta}^{\sigma-1}
    \gtrsim r C_{\!\Lap}^{-1}N_{\theta}^{\sigma-1}\,.
\end{aligned}
\]
The two constants \(C_{1,d}\) and \(C_{2,d}\) depend only on the dimension while the last bound holds as long as \( \epsilon_{\Theta}\leq \frac{C_{1,d}r}{2C_{2,d}}C_{\!\Lap}^{-2}\).

\textbf{Proof of \eqref{eq:sector-basis-projection-condition}.}
Using \eqref{eq:non-vanishing-projections-3}, coming from the construction of \(\ms{O}\), we obtain that
\[
\begin{aligned}[t]
    \big|\big\langle\xi;\ms{O}_{\theta} \hat{e}_{j}\big\rangle\big|
    &\geq
    \big|\big\langle c_\theta;\ms{O}_{\theta} \hat{e}_{j}\big\rangle\big|
    - 
    \big|\big\langle\xi-c_\theta;\ms{O}_{\theta} \hat{e}_{j}\big\rangle\big|
    \\
    &\geq rC_{1,d}N_{\theta}- \big|\xi- c_\theta\big|.
    \\ &
    \geq  r C_{1,d}N_{\theta} - \epsilon_{\Theta} N_{\theta}
    \gtrsim r N_{\theta}\,.
\end{aligned}
\]
The two constants \(C_{1,d}\) and \(C_{2,d}\) depend only on the dimension while the last bound holds as long as \( \epsilon_{\Theta}\leq \frac{C_{1,d}r}{2C_{2,d}}C_{\!\Lap}^{-2}\).

Next, we prove \Cref{prop:finite-basis}. 

\begin{lemma}\label{lem:linear-algebra-basis-kicking}
Fix \(r\in(0,1/2)\). There exists a finite collection 
\[
\ms{A}\subset \bigl\{A\in O(d) \st \|A-\Id\|<r \bigr\}\,
\]
such that for any vector \(v\in\R^{d}\) there exists \(A\in\ms{A}\) for which
\[\eqnum\label{eq:non-vanishing-projections-1}
|\langle   v; A \hat{e}_{j}\rangle|\gtrsim r|v| \qquad \forall j\in\{1,\ldots,d\}.
\]
\end{lemma}

\begin{proof}
For any \(j_{0}\in\{1,\ldots,d\}\) and any subset \(\mc{J}\subset\{1,\ldots,d\}\setminus\{j_{0}\}\) let \(e_{\mc{J}}=|\mc{J}|^{-1/2}\sum_{j\in\mc{J}}\hat{e}_{j}\), and let \(A_{j_{0},\mc{J}}\eqd e^{r (\hat{e}_{j_{0}}\cdot e_{\mc{J}}^{T}-e_{\mc{J}}\cdot e_{j_{0}}^{T})}\) be the rotation matrix of angle \(r\) in the plane spanned by \(\hat{e}_{j_{0}}\) and \(e_{\mc{J}}\).  If \(\mc{J}=\emptyset\) then \(A_{j_{0},\mc{J}}=\Id\).
Defines a finite collection \(\ms{A}\) to be \(\{A_{j_{0},\mc{J}}\}\). By properties of rotation matrixes we have \( \big\|O_{j_{0},\mc{J}}-\Id\big\|\leq r\).

Let us show that \(\ms{A}\) satisfies conditions \eqref{eq:non-vanishing-projections-1}. By homogeneity, we assume that \(|v|=1\). Let \(j_{0}\in\{1,\ldots,d\}\) be an index such that \(|v_{j_{0}}|>(2d)^{-1/2}\) and let \(\mc{J}=\{j\st |v_{j}|\leq\epsilon r\}\) for some \(\epsilon>0\) to be determined later. Let us decompose
\[
\begin{aligned}[t]
& v=v_{j_{0}} \hat{e}_{j_{0}} + v_{\mc{J}} e_{\mc{J}}+v_{\mc{J}}^{\perp} + v_{\#}, \text{ where }
\\
& \begin{aligned}[t]
& v_{j_{0}}=\langle   v; \hat{e}_{j_{0}}\rangle, & &v_{\mc{J}} =\langle  v;e_{\mc{J}}\rangle
\\
&v_{\mc{J}}^{\perp}=\sum_{j\in\mc{J}}v_{j}\hat{e}_{j}-v_{\mc{J}} e_{\mc{J}}, &&v_{\#}=\sum_{j\notin\mc{J}\cup\{j_{0}\}}v_{j}\hat{e}_{j}.
\end{aligned}
\end{aligned}
\]

\begin{itemize}
\item By construction, from the definition of \(\mc{J}\), we have that
\[
|v_{\mc{J}}|\leq\epsilon r d^{1/2}, \qquad |v_{\mc{J}}^{\perp}|\leq 2\epsilon r d 
\]
\item Since \(v_{\#}\) and \(v_{\mc{J}}^{\perp}\) are orthogonal to the plane of rotation we have
\[A_{j_{0},\mc{J}}^{T} v_{\#}=v_{\#}, \qquad A_{j_{0},\mc{J}}^{T} v_{\mc{J}}^{\perp}=v_{\#}.
\]
\item In the plane of rotation we have
\[
\begin{aligned}[t]
& A_{j_{0},\mc{J}}^{T} (v_{j_{0}} \hat{e}_{j_{0}}) =v_{j_{0}} \big(\cos(r)\hat{e}_{j_{0}}+\sin(r) e_{\mc{J}}\big),
\\ & A_{j_{0},\mc{J}}^{T} (v_{\mc{J}} e_{\mc{J}}) =v_{\mc{J}} \big(-\sin(r)\hat{e}_{j_{0}}+\cos(r) e_{\mc{J}}\big).
\end{aligned}
\]
\end{itemize}
Setting \(\epsilon=(8d)^{-2}\) the claim then follows. Indeed
\begin{itemize}
\item if \(j\notin \{j_{0}\}\cup\mc{J}\) we have \(|\langle v; A \hat{e}_{j}\rangle|=|\langle  v ; \hat{e}_{j}\rangle|\geq\epsilon r\gtrsim r\);
\item if \(j=j_{0}\) we have 
\[
\begin{aligned}[t]
|\langle v; A \hat{e}_{j}\rangle|&=|\cos(r)v_{j_{0}}-\sin(r)v_{\mc{J}}|
\\ & \geq |\cos(r)|(2d)^{-1/2}-|\sin(r)| r \epsilon d^{1/2}\gtrsim 1\,;
\end{aligned}
\]
\item if \(j\in\mc{J}\) we have 
\[
\begin{aligned}[t]
&|\langle v; A \hat{e}_{j}\rangle|=\big|(v_{j_{0}}\sin(r)-\cos(r)v_{\mc{J}}) \langle e_{\mc{J}};\hat{e}_{j} \rangle +\langle v_{\mc{J}}^{\perp},\hat{e}_{j}\rangle\big|
\\ & \quad>|\sin(r)|(2d)^{-1}-\epsilon rd^{1/2}|\cos(r)|-2\epsilon r d
\geq r ((4d)^{-1}-\epsilon-2 \epsilon d^{1/2})\gtrsim r.
\end{aligned}
\]
\end{itemize}
\end{proof}

\begin{corollary}\label{cor:2vector-basis-kicking}
Fix  \(r\in(0,1/2)\). There exists a finite collection 
\[
\ms{A}'\subset \bigl\{A'\in O(d) \st \|A'-\Id\|<r \bigr\}\,
\]
such that for any two vectors \(v_{1},v_{2}\subset\R^{d}\) there exists one \(A'\in\ms{A}'\) for which it holds that 
\[\eqnum\label{eq:non-vanishing-projections-2}
|\langle   v_{k}; A' \hat{e}_{j}\rangle|\gtrsim r|v_k| \qquad \forall j\in\{1,\ldots,d\} 
\]
for \( k\in\{1,2\}\).
\end{corollary}
\begin{proof}
Fix \(\epsilon>0\) to be determined later and let \(\ms{A}_{1}, \ms{A}_{2}\subset O(d)\) the two collections given by \Cref{lem:linear-algebra-basis-kicking} with \(r\) replaced by \(r/2\) and \(\epsilon r/2\), respectively. Set \(\ms{A}'\eqd \bigl\{A_{1}A_{2} \st A_{k}\in\ms{A}_{k}\bigr\} \). For any \(A'\in\ms{A}'\) we have that  \(\|A'-\Id\|\leq r\) since \(\|A_{1}A_{2}-\Id\|\leq \|A_{1}-\Id\|\,\|A_{2}\|+\|A_{2}-\Id\|\).

It remains to prove \eqref{eq:non-vanishing-projections-2}. First, we use \Cref{lem:linear-algebra-basis-kicking} to select \(A_{1}\) so that 
\(|\langle v_{1}; A_{1} \hat{e}_{j}\rangle|\gtrsim r|v_{1}|, \quad \forall j\in\{1,\ldots,d\}\). Next, we select \(A_{2}\) using \Cref{lem:linear-algebra-basis-kicking} so that
\[
\Big|\big\langle A_{1}^{T}v_{2}; A_{2} \hat{e}_{j}\big\rangle\Big|\gtrsim\epsilon r|v_{2}| \qquad \forall j\in\{1,\ldots,d\},
\]
having used that \(\big|A_{1}^{T} v_{2}\big|=|v_{2}|\). We set \(A=A_{1}A_{2}\). We already have that  for all $j\in\{1,\ldots,d\}$ it holds that      
\(\Big|\big\langle v_{2};A_{1} A_{2} \hat{e}_{j}\big\rangle\Big| \gtrsim\epsilon r|v_{2}|\). On the other hand, for \(v_{1}\) we have that
\[
\begin{aligned}[t]
\qquad&\nqquad \Big|\big\langle v_{1}; A_{1}A_{2} \hat{e}_{j}\big\rangle\Big|
\geq
\Big|\big\langle v_{1}; A_{1}\hat{e}_{j}\big\rangle\Big|
-
\Big|\big\langle v_{1}; A_{1}(\Id-A_{2})\hat{e}_{j}\big\rangle\Big|
\\
&
\geq C_{1,d}r|v_{1}|- \epsilon r|v_{1}|\gtrsim r|v_{n}|\,,
\end{aligned}
\]
where \(C_{1,d}\) is a constant that depends only on the dimension. The last bound holds by setting \(\epsilon=\sfrac{C_{1,d}}{2}\).
\end{proof}

Finally,  \Cref{prop:finite-basis}  follows from the statement above.

\begin{proof}[ Proof of \Cref{prop:finite-basis}]
    Let $\ms{B}\subset O(d)$ be a finite collection of matrixes such that for each $O\in O(d)$ there exists $B\in \ms{B}$ such that $\|B-O\|\leq r$. Such a collection exists by compactness of $O(d)$. Let $\ms{A}'$ be constructed as in \Cref{cor:2vector-basis-kicking}. Define 
    \[
\ms{A}''=\bigcup{B^TA'\in O(d) \st \ms{B}\in \ms{B}\text{ and } A'\in \ms{A}'}    
    \]
    For any $O\in O(d)$ we can select $A''\in \ms{A}''$ by setting $A''=B^TA'$ where we choose $B\in \ms{B}$ such that $\|B-O\|\leq r$ and we choose $A'\in \ms{A}'$ such that
    \eqref{eq:non-vanishing-projections-3} holds for the vectors $Bv1$ and $Bv_2$ in place of  $v_1$ and $v_2$. Then we have that 
    \[
    |\langle   v_{k}; B^TA' \hat{e}_{j}\rangle|
    =
    |\langle   Bv_{k}; TA' \hat{e}_{j}\rangle|
    \gtrsim r|v_k| \qquad \forall j\in\{1,\ldots,d\} 
    \]
    for \( k\in\{1,2\}\). Finally, we have that 
    \[
    \|B^TA'-O\|\leq\|B^T(A'-\Id)\|+\|B^T-O\|=\|A'-\Id\|+\|B-O\|\leq 2r.
    \]
    The claim of \Cref{prop:finite-basis} follows by replacing $r$ with $r/2$
\end{proof}

\subsection{Randomization}\label{sec:randomization} In this subsection, we recall the unit-scale Wiener randomization \(\rf\) from \cite{CFU2} of the
initial condition \(f\in H^{S}_{x}(\R^{d})\). Fix a sequence
$(g_k)_{k\in \Z^d}$ of i.i.d. complex valued random variables on a
probability space $(\Omega , \mathcal{A}, \mathbb{P})$ and assume that all their moments
are bounded, that is, \(\E\big[ |g_{k}|^{p}\big]<\infty\) for all
\(p\in\N\).\footnote{These assumptions are for example satisfied when
\(g_{k}\) are independent, standard (unit variance and \(0\) mean) complex
Gaussian random variables.}

Let $\psi \in C_c^{\infty} (\R^d)$ be an
even, non-negative cut-off function supported in the unit-ball in
$\R^d$ centered at $0$ such that
\[
\sum_{k \in \Z^d} \psi (\xi - k) = 1
\]
for all $\xi \in \R^d$.  To each $k \in \Z^d$ we associate the projection $Q_k$ with the Fourier multiplier $ \psi (\cdot - k)$:
\begin{equation}\label{eq:proj-unit-scale}
\Fourier (\QP_{k} h) (\xi) = \psi (\xi - k)\Fourier(h) (\xi),\qquad \text{for} \quad \xi \in \R^d,
\end{equation}
where $\Fourier (h)$ stands for the Fourier transform of a function $h:\R^n \to \R$.
 If $(g_k)_{k \in \Z^d}$ is a sequence of i.i.d zero-mean complex random
variables with finite moments of all orders on a probability space
$(\Omega, \mathcal{A}, \mathbb{P})$, then  
 the
randomization of $f$ is 
\begin{equation}\label{eq:randomization}
\rf=f^{\omega} = \sum_{k \in \Z^d} g_k (\omega) \QP_{k} f\, .
\end{equation}

The randomization \eqref{eq:randomization} does not improve the differentiability of $f$ in the sense
that if $f \in H^S (\R^d) \setminus H^{S + \varepsilon} (\R^d)$ for some
$\varepsilon > 0$, then $f^{\omega} \in H^S (\R^d) \setminus H^{S +
\varepsilon} (\R^d)$ almost surely (see {\cite{MR2425133}}). 
On the other hand, it \textit{improves} space-time integrability. In particular, for any $p\in [2,\infty)$ it holds that 
\[
\Big(\E \|f^\omega\|_{L^p_x(\R^d)}^p\Big)^{\sfrac{1}{p}}
\approx 
\Big(\E \|f^\omega\|_{L^2_x(\R^d)}^2\Big)^{\sfrac{1}{2}} \approx \|f\|_{L^2_x(\R^d)}.
\]
A key part of our argument establishes a similar improvement in the spatial integrability of the evolution $e^{-it\Lap}f$ (\Cref{prop:max-dir}).
Recall that we interpret the infinite sum in \eqref{eq:randomization}
as the limit in \(\|\cdot\|_{L^{p}_{\Omega}H^{S}_{x}}\) for any
\(p<\infty\) of finite partial sums. Equivalently, we define
\eqref{eq:randomization} for functions
\(f\in L^{2}(\R^{d})\) with compact Fourier support and then extend the
definition by continuity. 

\begin{comment}

\subsection{Randomization - new modifications}
~

Fix \(\rn\in\Z\) and define
\[
\QP_{k}f(x)\eqd\int_{\R^{d}}\FT{f}(\xi)\, \1_{k+[0, 2^{-\rn}]^{d}}(\xi) e^{2\pi i\xi x}\dd\xi.
\]
and let \({g_{k}}_{k \in 2^{-\rn}\Z^{d}}\) be a collection of i.i.d random variables with \(\E g_{k}=0\) and \(\E|g_{k}|^{p}<\infty\) for all \(p\in<\infty\). We define the randomization of $f$ as
\[\eqnum\label{eq:randomization}
\rf\eqd \sum_{k\in2^{-\rn}\Z^{d}}g_{k}\QP_{k}f
\]
with 

\end{comment}

\section{Linear estimates}

In this section we establish the key linear estimates that are used in the rest of the paper. In particular, we show estimates on the Schödinger group in the directional spaces with a different basis in each section, which are counterparts of the Strichatz estimates. Specifically, we establish the directional maximal estimate which requires a careful choice of the basis and directional smoothing estimate which is independent of coordinates. These estimates are the main technical novelty of the manuscript. 

For better reference, we first formulate the classical Bernstein inequality. 
% \begin{tcolorbox}
%     \textbf{The estimates are proven, but this section needs some introductory text.}
% \end{tcolorbox}

\begin{lemma}[Bernstein's inequality]\label{lem:bernstein}
For any $r_{1}, r_{2}\in[1, \infty]$ with \(r_{2}\geq r_{1}\) it holds that
\[
\|f\|_{L^{r_{2}}_{x} (\R^{d})}
\lesssim
\diam\big(\spt(\FT{f})\big)^{ \frac{d}{r_{1}} - \frac{d}{r_{2}} }
\|f\|_{L^{r_{1}}_{x} (\R^{d})} .
\]

In particular,
since $\diam \big( \spt (\FT{\QP_{k} f}) \big) \lesssim 1$ for any $k \in \Z^{d}$, we obtain
\[\eqnum\label{eq:unit-scale-bernstein}
\big\|\QP_{k} f\big\|_{L^{r_{2}} _{x}(\R^{d})} \lesssim \big\|\QP_{k} f\big\|_{L^{r_{1}}_{x}(\R^{d})}.
\]
Also for any section $\theta \in \Theta$ (see \Cref{defn:sectors}) 
we have 
$\diam\big( \spt( \FT{\LP_{\theta} f})\big) \lesssim N_{\theta}$, and therefore
\[\eqnum\label{eq:LP-bernstein}
\big\|\LP_{\theta} f\big\|_{L^{r_{2}}_{x} (\R^{d})} \leq N_{\theta}^{d( \frac{1}{r_{1}} - \frac{1}{r_{2}} )} \big\|\LP_{\theta} f\big\|_{L^{r_{1}}_{x} (\R^{d})}\,.
\]
The implicit constants are allowed to depend only on the dimension $d$ and on \(r_{1},r_{2}\).
\end{lemma}

Recall that the Schrödinger propagator is defined as%\todoGU{I am not sure it is called this way, and also, did we already write this? If not, why ``recall''?}
\[
    e^{it \Lap}f(x)
    \eqd 
    \int_{\R^d}e^{2\pi ix\xi}e^{it \Lap(\xi)}\FT{f} (\xi)\dd\xi.
\]

The next two propositions are the main results of this section.

\begin{proposition}[Directional maximal estimate]\label{prop:max-dir}
Let $d>\sigma\geq2$, let $\Lap$ be an operator whose symbol
satisfies \eqref{eq:symbol-conditions}, let \(\Theta\) be a collection of sectors and let \(\ms{O}:\Theta\to O(d)\) be a collection of bases adapted to the sectors (see \Cref{defn:sector-basis-collection}).

For any $\epsilon>0$, $\mf{c}\in[2, \infty]$, and any \(\theta\in\Theta\) and \(j\in\{1,\ldots,d\}\) we have
\[\eqnum\label{eq:directional-maximal}
\Big\|e^{it\Lap} \LP_{\theta} f\Big\|_{L^{(2, \infty, \mf{c})}_{\ms{O}_{\theta}, j}(I\times\R^{d})}
\lesssim_{\epsilon}
N_{\theta}^{
\frac{\sigma}{4} + \big( d - 1 - \frac{\sigma}{2} \big) \big( \frac{1}{2} - \frac{1}{\mf{c}} \big)+\epsilon}\|\LP_{\theta}f\|_{L^{2}(\R^{d})} \,.
\]
Furthermore, if $\diam(\spt{\FT{f}})\leq R$, then
\[\eqnum\label{eq:directional-maximal-unit-FT-support}
\Big\|e^{it\Lap} \LP_{\theta} f\Big\|_{L^{(2, \infty, \mf{c})}_{\ms{O}_{\theta}, j}(I\times\R^{d})} \lesssim N_{\theta}^{\frac{\sigma}{4}+\epsilon}R^{( d - 1 ) \big( \frac{1}{2} - \frac{1}{\mf{c}} \big)}\|\LP_{\theta}f\|_{L^{2}(\R^{d})}.
\]
The implicit constants are allowed to depend on \(C_{\ms{O}}\),\(C_{\!\Lap}\), and \(\Theta\).
\end{proposition}

\begin{proposition}[Directional smoothing estimate]\label{prop:dir-local-smoothing}
Fix $d>\sigma\geq2$ and let $\mLap$ be the symbol
(corresponding to $\Lap$) satisfying
\eqref{eq:symbol-conditions}. Let \(\Theta\) be a collection of sectors with resolution \(\epsilon_{\Theta}\lesssim_{d}C_{\!\Lap}^{2}\) as in \Cref{defn:sectors}.

Given a sector \(\theta\in\Theta\), for any \(O\in \Orth(d)\) let \(\ms{J}_{\theta}(O)\) the collection of indexes \(j\in\{1, \ldots, d\}\) given by
\[\eqnum\label{eq:def:dominant-sector-directions}
\ms{J}_{\theta}(O)\eqd\Big\{j\in\{1, \ldots, d\}\st |\langle\nabla\mLap(c_{\theta});O \hat{e}_{j}\rangle|\geq d^{-\sfrac{1}{2}}|\nabla\mLap(c_{\theta})|\Big\}
\]

There exists $T_0 > 0$ such that for any open interval $I \subset \R$ with
$|I| \leq T_0$, any $\theta\in\Theta$, \(O\in\Orth(d)\),  $\mf{c}\in[2, \infty]$, and any \(j\in\ms{J}_{\theta}(O)\) it holds that 
\[\eqnum\label{eq:dir-local-smoothing-genbasis}
 \big\| e^{- i t \Lap} \LP_{\theta} f \big\|_{L_{O, j}^{(\infty, 2, \mf{c})} (I \times \R^d)} \lesssim
 N_{\theta}^{-\frac{\sigma - 1}{2} +(d-1)\big(\frac{1}{2}-\frac{1}{\mf{c}}\big)}
 \|\LP_{\theta} f\|_{L^2 (\R^d)} \,.
\]

Furthermore, if $\diam\big(\spt \FT{f} \big)\leq R$ then 
\[\eqnum\label{eq:dir-local-smoothing-genbasis-unit-spt}
 \big\| e^{- i t \Lap} \LP_{\theta}f \big\|_{L_{O, j}^{(\infty, 2, \mf{c})} (I \times \R^d)} 
 \lesssim
 N_{\theta}^{-\frac{\sigma - 1}{2}} R^{(d-1)\big(\frac{1}{2}-\frac{1}{\mf{c}}\big)}
 \|\LP_{\theta} f\|_{L^2 (\R^d)} \,.  
\]
%\todoGU{I removed the directional projections. If the sectors are small enough, then on the support of \(\LP_{\theta}\) the gradient points in the same direction anyway.} 
\end{proposition}

%\begin{tcolorbox}The directional estimates of this Proposition are only for the directions of this set \(\ms{J}_{\theta}(O)\) (it can have only $1$ element), whereas in your previous paper they hold for all coordinate directions of $\mathbb{R}^{d}$. We should add a remark here about that, saying why that's enough for us, and contrast this setting to the one of the other paper.
%\end{tcolorbox}

We prove Propositions \ref{prop:max-dir} and \ref{prop:dir-local-smoothing} in the next two subsections.

\subsection{Proof of \Cref{prop:max-dir}} Fix the sector \(\theta\in\Theta\) and assume, without loss of generality, that \(j=1\) and thus that \(S_{j}=\Id\). We allow all the implicit constants in the proof to depend on \(C_{\ms{O}}\),\(C_{\!\Lap}\), and \(\Theta\).

If  $N_\theta <  2N_{\min}$ (see \eqref{eq:symbol-conditions}) we prove the claim using a rather crude argument. The fundamental theorem of calculus gives that 
\[
\begin{aligned}[t]
\sup_{t\in I} \|e^{it\Lap} \LP_{\theta} f(x_1,x')\|_{L^{\mf{c}}_{x'}(\R^{d-1})}
\leq&  
\|\LP_{\theta} f(x_1,x')\|_{L^{\mf{c}}_{x'}(\R^{d-1})}
\\ &+
\int_{t \in I} \|\Lap e^{it\Lap} \LP_{\theta} f(x_1,x')\|_{L^{\mf{c}}_{x'}(\R^{d-1})}.
\end{aligned}
\]
Using the triangle inequality, Minkowski, and Bernstein's inequality allow us to obtain that 
\[
\begin{aligned}[t]
\qquad &\nqquad \Big\|e^{it\Lap} \LP_{\theta} f\Big\|_{L^{(2, \infty,\mf{c})}_{\ms{O}_{\theta}, j}(I\times\R^{d})} 
\\ 
& \leq
\Big\|\big\|\LP_{\theta} f(x_1,x')\big\|_{L_{x'}^{\mf{c}}(\R^{d-1})}\Big\|_{L^2_{x_1}(\R)} 
+
\int_{t \in I} \Big\|\big\|\Lap e^{it\Lap}\LP_{\theta} f(x_1,x')\big\|_{L_{x'}^{\mf{c}}(\R^d)}\Big\|_{L^2_{x_1}(\R)} 
\\ 
& 
\lesssim N_{\min}^{(d-1)\big(\frac{1}{2}-\frac{1}{c}\big)}
\Big( 
\big\|\LP_{\theta} f(x_1,x')\big\|_{L^{2}(\R^d)}
+
\int_{t \in I} \big\|\Lap e^{it\Lap}\LP_{\theta} f(x_1,x')\big\|_{L^{2}(\R^d)}
\Big)
\\
\end{aligned}
\]
By Plancherel, and using that  $\mLap(\xi) e^{it\mLap(\xi)}$ is bounded on the bounded domain $B_{4N_{\min}}(0)$ we obtain that 
\[
\begin{aligned}[t]
& \Big\|e^{it\Lap} \LP_{\theta} f\Big\|_{L^{(2, \infty,\ mf{c})}_{\ms{O}_{\theta}, j}(I\times\R^{d})} 
\lesssim N_{\min}^{(d-1)\big(\frac{1}{2}-\frac{1}{c}\big)} |I|\sup_{\xi\in \theta} |\mLap(\xi)| \, \big \|\LP_{\theta} f\big\|_{L^2(\R^d)},
\end{aligned}
\]
as desired, since the implicit constant in our claim is allowed to depend on $N_{\min}$ and $|I|$.

Henceforth, let us suppose that $N_\theta \geq  2N_{\min}$, and thus
that \eqref{eq:symbol-conditions} hold for all \(\xi\in\conv(\theta)\).
A series of reductions allows us to deduce the full claim of \Cref{prop:max-dir} from the bounds \eqref{eq:directional-maximal} with \(\mf{c}=2\). Indeed, 
When \(\mf{c}=\infty\), bounds \eqref{eq:directional-maximal} for any $O\in \Orth(d)$ in place of $\ms{O}_{\theta}$
 follow from \cite[Lemma
2.4]{CFU} that 
\[\eqnum\label{eq:dir-maximal-13feb24}
\big\| e^{- i t \Lap} \LP_{\theta} f \big\|_{L_{O, j}^{(2, \infty, \infty)} (I \times \R^d)} 
\lesssim N_\theta^{\frac{d - 1}{2}}  \|\LP_{\theta} f\|_{L_x^2 (\R^d)} .
\]
 Using Riesz-Thorin interpolation in the inner \(L_{x'}^{\mf{c}}(\R^{d})\) norm the bound one \eqref{eq:directional-maximal} follows for arbitrary $\mf{c} \in (2, \infty)$ as long as one shows them for \(\mf{c}=2\). Furthermore, bound \eqref{eq:directional-maximal-unit-FT-support}
can be obtained by applying the Bernstein inequality (\Cref{lem:bernstein}). Indeed, if $\diam (\spt \FT{\LP_{\theta} f})\leq R$ then for any fixed $(x_{1},t)$ it holds that 
\[
\Big\|e^{it\Lap} \LP_{\theta} f\ms{O}_{\theta}(x_{1},\cdot)\Big\|_{L^{\mf{c}}_{x'}(\R^{d-1})}  \lesssim R^{(d-1)\big(\frac{1}{2}-\frac{1}{\mf{c}}\big)}\Big\|e^{it\Lap} \LP_{\theta} \ms{O}_{\theta}f(x_{1},\cdot)\Big\|_{L^{2}_{x'}(\R^{d-1})}\,,
\]
since the Fourier support of $e^{it\Lap} \LP_{\theta} f(x)$ is also contained in a ball of radius $R$.\todoGU{Actually, we may want to write a more precise proof of the \(\mf{c}=\infty\) case and track how it depends on \(R\). Also the current proof in our previous paper blackboxes some oscillatory estimates. We should provide a clean proof for the final version of this manuscript}

It remains to prove \eqref{eq:directional-maximal} for $\mf{c}=2$. We represent
$\eta\in\theta$ in the orthonormal basis of the columns of
$\ms{O}_{\theta}$ by setting \(\eta= \ms{O}_{\theta}S_{j}(\eta_{1}, \eta')^T\) with
\(\eta_{1}\in\R\) and \(\eta'\in\R^{d-1}\). 

For any fixed $\eta'\in\R^{d-1}$, define
the one-dimensional symbol
\[\eqnum\label{eq:odld}
\mLap_{\eta'}(\eta_{1}) := \mLap \big(\ms{O}_{\theta}(\eta_{1}, \eta')\big) \,.
\]
Our proof relies on the fact that \(\mLap_{\eta'}(\eta')\) is a symbol of
order \(\sigma\). More precisely, we assume the following claim to be shown at the end of the proof.

\begin{claim}\label{claim:claim-eta}
For every \(\eta'\in\R^{d-1}\), if the set  \(\big\{\eta_{1}\st \ms{O}_{\theta}(\eta_{1}, \eta')\in\conv(\theta) \big\}\) is non-empty then it is an interval of the form \((R,C_{0}R)\) with \(R>0\) and \(|R|\approx N_{\theta}\) or of the form \((C_{0}R,R)\) with \(R<0\) and \(|R|\approx N_{\theta}\) for some \(C_{0}\approx1\). Furthermore, on that interval, the multiplier $\eta_{1}\mapsto\mLap_{\eta'}(\eta_{1})$ satisfies the conditions \eqref{eq:local-shiraki-multiplier-conditions} of \Cref{prop:local-shiraki}. 
\end{claim}

By the definition \eqref{eq:def:directional-norm} of the directional norms we have 
\[
\Big\|e^{it\Lap} \LP_{\theta} f\Big\|_{L^{(2, \infty, 2)}_{\ms{O}_{\theta}, 1}(I\times\R^{d})}
 = \Big\|\sup_{t\in I}\Big\|\big(e^{it\Lap}\LP_{\theta}f\big)\big(\ms{O}_{\theta}(x_{1}, x^{\prime})\big)\Big\|_{L^{2}_{x'(\R^{d-1})}}\Big\|_{L^{2}_{x_{1}}(\R)}.
\]
Using the Plancherel identity in the \(x'\) variable we obtain that 
\[
\begin{aligned}[t]
\hspace{2em}&\hspace{-2em}
\Big\|\big(e^{it\Lap}\LP_{\theta}f\big)\big(\ms{O}_{\theta}(x_{1}, x^{\prime})\big)\Big\|_{L^{2}_{x'}} 
\\
&=
\Big\|\int_{\R}
\int_{\R^{d-1}}e^{it \Lap\big(\ms{O}_{\theta}(\eta_{1}, \eta')\big)}\FT{\LP_{\theta}f}\big(\ms{O}_{\theta}(\eta_{1}, \eta')\big)
e^{2\pi i\eta_{1}x_{1}+2\pi i\langle\eta'; x'\rangle}\dd\eta'
\dd \eta_{1}\Big\|_{L^{2}_{x'}}
\\
&=
\Big\|\int_{\R} e^{it\Lap_{\eta'}(\eta_{1})}\FT{\LP_{\theta}f}\big(\ms{O}_{\theta}(\eta_{1}, \eta')\big)e^{2\pi i\eta_{1}x_{1}}\dd \eta_{1}\Big\|_{L^{2}_{\eta'}}.
\end{aligned}
\]
Using the Minkowski and Fubini theorems we obtain that 
\[
\Big\|e^{it\Lap} \LP_{\theta} f\Big\|_{L^{(2, \infty, 2)}_{\ms{O}_{\theta}, 1}(I\times\R^{d})}
\leq
\bigg\|\bigg\|
\sup_{t}\Big|\int_{\R} e^{it\Lap_{\eta'}(\eta_{1})}\FT{\LP_{\theta}f}\big(\ms{O}_{\theta}(\eta_{1}, \eta')\big)e^{2\pi i\eta_{1}x_{1}}\dd \eta_{1}\Big|
\bigg\|_{L^{2}_{x_{1}}} \bigg\|_{L^{2}_{\eta'}}.
\]
Having assumed \Cref{claim:claim-eta}, we apply \Cref{prop:local-shiraki} for each fixed \(\eta'\in\R^{d-1}\) to obtain that 
\[
\Big\|e^{it\Lap} \LP_{\theta} f\Big\|_{L^{(2, \infty, 2)}_{\ms{O}_{\theta}, 1}(I\times\R^{d})}
\lesssim
\bigg\|
\bigg\|\big(1+|\eta_{1}|^{2}\big)^{s/2}\FT{\LP_{\theta}f}\big(\ms{O}_{\theta}(\eta_{1}, \eta')\big) \bigg\|_{L^{2}_{\eta_{1}}}
\bigg\|_{L^{2}_{\eta'}}.
\]
for all $s>\frac{\sigma}{4}$, with an implicit constant independent of $\eta'$ and $\theta$. The desired bound
\eqref{eq:directional-maximal} with $\mf{c} = 2$ is equivalent to
\[
\Big\|e^{it\Lap} \LP_{\theta} f\Big\|_{L^{(2, \infty, 2)}_{\ms{O}_{\theta}, 1}(I\times\R^{d})}\lesssim \big\| \LP_{\theta}f\big\|_{H^{s}(\R^{d})}
\]
which holds since 
\[
\bigg\|
\Big\|\big(1+|\eta_{1}|^{2}\big)^{s/2}\FT{\LP_{\theta}f}\big(\ms{O}_{\theta}(\eta_{1}, \eta')\big) \Big\|_{L^{2}_{\eta_{1}}}
\bigg\|_{L^{2}_{\eta'}}
\lesssim\|\LP_{\theta}f\|_{H^{s}}.
\]

We conclude the proof by showing that \Cref{claim:claim-eta} holds. Since \(\conv(\theta)\) is convex then \(\big\{\eta_{1}\st \ms{O}_{\theta}(\eta_{1}, \eta')\in\conv(\theta) \big\}\) is an interval. If it is non-empty, then, using
\eqref{eq:sector-basis-projection-condition} and
\eqref{eq:sector-points-bounds} it follows that
\[
|\eta_{1}|= |\langle \eta, O_\theta \hat{e}_j \rangle| \approx |\eta|\approx N_{\theta},
\]
and so \(\big\{\eta_{1}\st \ms{O}_{\theta}(\eta_{1}, \eta')\in\conv(\theta) \big\}\) is indeed an interval of the required form.  
Next we check \eqref{eq:local-shiraki-multiplier-conditions}. We have that
\[
\partial_{\eta_{1}} \mLap_{\eta'} (\eta_{1})=\partial_{\eta_{1}} \mLap\ms{O}_{\theta}(\eta_{1},\eta') =\big\langle \nabla \mLap\ms{O}_{\theta}(\eta_{1},\eta');\ms{O}_{\theta}\hat{e}_{1}\big\rangle
\]
Thus the bound \(|\partial_{\eta_{1}} \mLap_{\eta'} (\eta_{1})|
\approx N_{\theta }^{\sigma-1}\) follows from the condition  \eqref{eq:sector-basis-gradient-condition} on the sector basis. This shows the bound on the first derivative in \eqref{eq:local-shiraki-multiplier-conditions}.

Finally, we have that
\[
\partial_{\eta_{1}}^{2} \mLap_{\eta'} (\eta_{1})=\big\langle \partial_{\eta_{1}} \nabla \mLap\ms{O}_{\theta}(\eta_{1},\eta');\ms{O}_{\theta}\hat{e}_{1}\big\rangle
=
\big\langle  D^{2} \mLap\ms{O}_{\theta}(\eta_{1},\eta') \ms{O}_{\theta}\hat{e}_{1};\ms{O}_{\theta}\hat{e}_{1}\big\rangle.
\]
Thus the bound\(|\partial_{\eta_{1}}^{2} \mLap_{\eta'} (\eta_{1})|
\approx N_{\theta }^{\sigma-2}\),  follows from the condition  \eqref{eq:sector-basis-hessian-condition} on the sector basis. This shows the bound on the second derivative in \eqref{eq:local-shiraki-multiplier-conditions} and concludes the proof of \ref{claim:claim-eta}.

\subsection{Proof of \Cref{prop:dir-local-smoothing}}  Bounds \eqref{eq:dir-local-smoothing-genbasis} and \eqref{eq:dir-local-smoothing-genbasis-unit-spt} for \(\mf{c} \geq 2\) follow from \eqref{eq:dir-local-smoothing-genbasis} for \(\mf{c}=2\). Indeed, by Bernstein's inequality (\Cref{lem:bernstein}) if $\diam (\spt \FT{\LP_{\theta} f})\leq R$   
\[
\Big\|\big(e^{it\Lap}\LP_{\theta}f\big)\big(O_{\theta}(x_{1}, x^{\prime})\big)\Big\|_{L^{\mf{c}}_{x'}(\mathbb{R}^{d-1})}\lesssim R^{(d-1)(\frac{1}{2}-\frac{1}{\mf{c}})}\Big\|\big(e^{it\Lap}\LP_{\theta}f\big)\big(O_{\theta}(x_{1}, x^{\prime})\big)\Big\|_{L^{2}_{x'}(\mathbb{R}^{d-1})}
\]
and \(\diam (\spt \FT{\LP_{\theta} f}))\lesssim\min\Big(N_{\theta}\diam (\spt \FT{f})\Big) \). So, we concentrate on proving \eqref{eq:dir-local-smoothing-genbasis} with $\mf{c}=2$.

Fix $j\in \ms{J}_{\theta}(O)$. By replacing \(\Lap(\xi)\) with \(\Lap(OS_{j} \xi)\) we may, without loss of generality (by \eqref{eq:def:directional-norm}), assume that \(O=\Id\) and \(j=1\). Let %\todoGU{Is this enough clarity?}\todoI{Yes. I added \eqref{eq:def:directional-norm} as reference here.} \todoGU{I think just the reference does not clarify much. The point is that the symbol, the is well behaved with respect to composing with an orthogonal transformation, in the sense that the constants of \eqref{eq:symbol-conditions} remain essentially the same. Technically, the sector also gets moved around but it is not very influential. Of course you can rewrite this procedure using \(V\) and \(V^{\perp}\)} 
\[
T f (t, x_{1},x') \eqd \int_{\mrl{\nquad \R\times\R^{d-1}}} e^{2 \pi i (\xi_{1}x_{1},\xi'x') - i t \Lap (\xi_{1},\xi')} \chi_{{\theta}}(\xi_{1},\xi') \FT{f} (\xi_{1},\xi') \dd \xi_{1}d \xi' \,,
\]
where $\chi_{{\theta}} = \chi^{\mrm{rad}}_{N_\theta}\chi^{\mrm{ang}}_{\hat{u}_\theta}$ is the multiplier associated to $P_\theta$. For each fixed $x_1\in\R$, Plancherel's identity in $x'$ gives 
\[
\begin{aligned}[t]
\big\| T f (t, x_1, x') \big\|_{L^2_t L^2_{x'} ( I \times \R^{d-1})}^{2}
= \int_{\mcl{I\times \R^{d-1}}} \Big| \int_{\R} e^{2 \pi i \xi_1 x_1 - i t \Lap (\xi_{1},\xi')} \chi_{\theta} (\xi_{1},\xi') \FT{f} (\xi_1, \xi') \dd \xi_1 \Big|^2 \dd t \dd\xi',
\end{aligned}
\]
where
$\xi=\xi_{1}e_{1}+\xi'$ with $\xi_{1}\in \R$ and  $\xi' \in \R^{d-1}$.

If $N_{\theta} \leq 2N_{\min}$, by Cauchy-Schwarz inequality, using that $| I | \lesssim T_0$, and that \(\diam(\spt\chi_{\theta})\lesssim N_{\min}\) we obtain 
\[
\big\| T f (\cdot, x_1, x') \big\|_{L^2_t L^2_{x'} ( I \times \R^{d-1})}^{2} \lesssim
\int_{\mrl{\nquad I\times \R^{d-1}}} N_{\min}\|\FT{f}(\xi_{1},\xi')\|_{L^{2}_{\xi_{1}}(\R)}^{2} \dd\xi'
\lesssim T_{0}N_{\min}\|\FT{f}\|_{L^{2}(\R^{d})}^{2}\,.
\]
In this case, our claim follows since the implicit constant is allowed to depend on $N_{\min}$ and on $T_{0}$.

Now let us focus on the case \(N_{\theta}\geq 2N_{\min}\). It is sufficient to show that for fixed $x_1 \in \R$, and $\xi' \in \R^{d-1}$ one has
\[\eqnum\label{eq:ibfls-genbasis}
\int_I \Big| \int_{\R} e^{i 2 \pi \xi_1 \cdot x_1  - i t  \Lap (\xi_{1},\xi')} \chi_{\theta} (\xi_{1},\xi') \FT{f} (\xi_{1},\xi') d\xi_1 \Big|^2 \dd t
\\
\lesssim N_{\theta}^{-(\sigma -1)}\hspace{-0.7em}\int_{\R} | \FT{f} (\xi_{1},\xi') |^2 \dd   \xi_1 \,.
\]
Define
$\theta_{\xi'}:=\big\{\xi_{1}\in\mathbb{R};(\xi_{1},\xi')\in\conv(\theta)\big\}$; we can restrict the
integral in \(\xi_{1}\) to \(\theta_{\xi'}\) since otherwise \(\chi_{\theta}(\xi_{1},\xi')\) vanishes.

For all \(\xi_{1}\in\theta_{\xi'}\) it holds that \(|\partial_{\xi_{1}}\mLap (\xi_1,  \xi')|\approx N_{\theta}^{\sigma-1}
\). Indeed, by taking \(j\in\ms{J}_{\theta}(O)\) and using \eqref{eq:sector-Lap-bounds} we have assumed that \(|\partial_{\xi_{1}}\mLap (c_{\theta})|\geq d^{-\sfrac{1}{2}}C_{\!\Lap}N_{\theta}^{\sigma-1}\). On the other hand, using \eqref{eq:sector-Lap-bounds-compare}, we obtain that
\[
|\partial_{\xi_{1}}\mLap (\xi_{1},\xi')|\geq |\partial_{\xi_{1}}\mLap (c_{\theta})|-|\partial_{\xi_{1}}\mLap (c_{\theta})-\partial_{\xi_{1}}\mLap (\xi_{1},\xi')|
\geq (d^{-\sfrac{1}{2}}-\epsilon_{\Theta})C_{\!\Lap}N_{\theta}^{\sigma-1},
\]
as required. %\todoGU{Explain why.}\todoI{Isn't it untimately because of \eqref{eq:sector-Lap-bounds} and \eqref{eq:sector-basis-gradient-condition}?}\todoGU{No! The matrix here is NOT the one of the sector, it is a random matrix. This is important}
In particular, \(\partial_{\xi_{1}}\mLap (\xi_1,  \xi')\) does not change sign so we can rewrite the integral on the left-hand side of \eqref{eq:ibfls-genbasis} in terms of the variable \(\eta=\mLap(\xi_{1},\xi')\):
\[
\begin{aligned}[t]
&
\int_I \Big| \int_{\theta_{\xi'}} e^{i 2 \pi \xi_1 \cdot x_1  - i t  \Lap (\xi_{1},\xi')} \chi_{\theta} (\xi_{1},\xi') \FT{f} (\xi_{1},\xi') d\xi_1 \Big|^2 \dd t
\\
& =
\int_I \Big| \int_{\tilde{\theta}_{\xi'}} e^{i 2 \pi \xi_1(\eta) \cdot x_1  - i t\eta} \chi_{\theta} (\xi_{1}(\eta),\xi') \FT{f} (\xi_{1}(\eta),\xi') \frac{d\eta}{|\partial_{\xi_{1}}\mLap(\xi_{1}(\eta),\xi')|} \Big|^2 \dd t,
\end{aligned}
\]
with \(\tilde{\theta}_{\xi'}=\{\mLap(\xi_{1},\xi')\st \xi_{1}\in\theta_{\xi'}\}\). By Plancherel we obtain that
\[
\begin{aligned}[t]
&
\int_I \Big| \int_{\tilde{\theta}_{\xi'}} e^{i 2 \pi \xi_1(\eta) \cdot x_1  - i t\eta} \chi_{\theta} (\xi_{1}(\eta),\xi') \FT{f} (\xi_{1}(\eta),\xi') \frac{d\eta}{|\partial_{\xi_{1}}\mLap(\xi_{1}(\eta),\xi')|} \Big|^2 \dd t,
\\
& \qquad \leq
\int_{\tilde{\theta}_{\xi'}} \Big|  e^{i 2 \pi \xi_1(\eta) \cdot x_1 } \chi_{\theta} (\xi_{1}(\eta),\xi') \FT{f} (\xi_{1}(\eta),\xi') \Big|^2 \frac{ d\eta}{|\partial_{\xi_{1}}\mLap(\xi_{1}(\eta),\xi')|^{2}}.
\end{aligned}
\]
Combining the above two step and using the inverse change of variables, we obtain that 
\[
\begin{aligned}[t]
\quad &\nquad
\int_I \Big| \int_{\theta_{\xi'}} e^{i 2 \pi \xi_1 \cdot x_1  - i t  \Lap (\xi_{1},\xi')} \chi_{\theta} (\xi_{1},\xi') \FT{f} (\xi_{1},\xi') d\xi_1 \Big|^2 \dd t
\\
& \leq
\int_{\theta_{\xi'}} \Big|  e^{i 2 \pi \xi_1 \cdot x_1 } \chi_{\theta} (\xi_{1},\xi') \FT{f} (\xi_{1},\xi') \Big|^2 \frac{ d\xi_{1}}{|\partial_{\xi_{1}}\mLap(\xi_{1},\xi')|}
\\
& \leq
C_{\!\Lap}N_{\theta}^{-(\sigma-1)}\int_{\theta_{\xi'}} \Big|  e^{i 2 \pi \xi_1 \cdot x_1 } \chi_{\theta} (\xi_{1},\xi') \FT{f} (\xi_{1},\xi') \Big|^2  d\xi_{1}\,,
\end{aligned}
\]
as desired.

\section{Evolution norms}
For this section we suppose that $\Lap$ is a fixed symbol (and also
corresponding operator) satisfying \eqref{eq:symbol-conditions},
\(\Theta\) is a collection of sectors, and let \(\ms{O}\) be a basis collection
adapted to the sectors and \(\Lap\).

Let $\epsilon_{0}\in\big(0, 2^{-100}\big)$ be sufficiently small (to be chosen
appropriately below depending on other parameters) and we allow all
subsequent constants to depend implicitly on \(\epsilon_{0}\).

For any interval
\(I\subseteq\R\) and for any $\sigma \in \R$ we set
\[\eqnum\label{eq:def:X-norm}
\begin{aligned}[c] &
\|v\|_{X^{\sigma} (I)} \eqd \Big( \sum_{\theta\in\Theta} N_{\theta}^{2 \sigma} \|\LP_{\theta} v\|^{2}_{X_{\theta} (I)} \Big)^{\sfrac{1}{2}} \,,
\\
& 
\begin{aligned}[t]
\|v\|_{X_{\theta} (I)}  \eqd \nqquad & \qquad
\|v\|_{L_{t}^{\frac{2}{\epsilon_0}} L_{x}^{\frac{2}{1 - \epsilon_{0}}} (I)}
+
\|v\|_{L_{t}^{\frac{2}{1 - \epsilon_{0}}} L_{x}^{\frac{2d}{d-2}\frac{1}{1 - \epsilon_{0}}} (I)}
\\
& 
+ \sum_{j= 1}^{d}\Big(
N_{\theta}^{-\frac{\sigma}{4}} \|v\|_{L_{\ms{O}_{\theta}, j}^{( \frac{2}{1 - \epsilon_{0}}, \frac{2}{\epsilon_{0}},  \frac{2}{1-\epsilon_{0}})} (I)}
+
N_{\theta}^{-\frac{d-1}{2}} \|v\|_{L_{\ms{O}_{\theta}, j}^{( \frac{2}{1 - \epsilon_{0}}, \frac{2}{\epsilon_{0}}, \frac{2}{\epsilon_{0}})} (I)}
\Big)
\\
&
+ \sum_{\mcl{O\in\ms{O}(\Theta)}}\;\;\;\sum_{\!\nquad\mrl{j \in\mc{J}_{\theta}(O)}}
\Big( N_{\theta}^{\frac{\sigma-1}{2}} \| v \|_{L_{O, j}^{(\frac{2}{\epsilon_{0}}, \frac{2}{1 - \epsilon_{0}}, \frac{2}{1 - \epsilon_{0}} )} (I)}
+   N_{\theta}^{-\frac{d-\sigma}{2}}\|  v\|_{L_{O, j}^{(\frac{2}{\epsilon_{0}}, \frac{2}{1 - \epsilon_{0}}, \frac{2}{\epsilon_{0}} )} (I)}\Big)\,,
\end{aligned}
\end{aligned}
\]
and 
\[\eqnum\label{eq:def:Y-norm}
\begin{aligned}[c] 
& \|v\|_{Y^{\sigma} (I)} \eqd \Big( \sum_{\omega \in \Theta} N_{\theta}^{2 \sigma} \|\LP_{\theta} v\|^{2}_{Y_{N} (I)} \Big)^{\frac{1}{2}},
\\
& \begin{aligned}[t]
\|v\|_{Y_{\theta} (I)} \eqd \nquad &\quad   \|v\|_{L_{t}^{\frac{2}{\epsilon_0}} L_{x}^{\frac{2}{1 - \epsilon_{0}}} (I)} + \|v\|_{L_{t}^{\frac{2}{1 - \epsilon_{0}}} L_{x}^{\frac{2d}{d-2}\frac{1}{1 - \epsilon_{0}}} (I)}
\\
&
+ \sum_{j= 1}^{d}\Big(
N_{\theta}^{-\frac{\sigma}{4}} \|v\|_{L_{O_{\theta}, j}^{( \frac{2}{1 - \epsilon_{0}}, \frac{2}{\epsilon_{0}},  \frac{2}{1-\epsilon_{0}})} (I)}
+
N_{\theta}^{-\frac{\sigma}{4}} \|v\|_{L_{O_{\theta}, j}^{( \frac{2}{1 - \epsilon_{0}}, \frac{2}{\epsilon_{0}}, \frac{2}{\epsilon_{0}})} (I)}
\Big)
\\
&
+ \sum_{O\in\ms{O}}\sum_{\nquad\mrl{j \in\mc{J}_{\theta}(O)}}
\Big( N_{\theta}^{\frac{\sigma-1}{2}} \|v\|_{L_{O, j}^{(\frac{2}{\epsilon_{0}}, \frac{2}{1 - \epsilon_{0}}, \frac{2}{1 - \epsilon_{0}} )} (I)}
+   N_{\theta}^{\frac{\sigma-1}{2}}\|v\|_{L_{O, j}^{(\frac{2}{\epsilon_{0}}, \frac{2}{1 - \epsilon_{0}}, \frac{2}{\epsilon_{0}} )} (I)}\Big)\,,
\end{aligned}
\end{aligned}
\]

The norms $X_{\theta}$ and $Y_{\theta}$ consist of a combination of space-time
norms with appropriate scaling. Note that the set $\ms{O}$ is finite,
and therefore all the sums in the definitions of
$X_{\theta}$ and $Y_{\theta}$ are well defined.

Next, we formulate estimates on  solutions to the non-homogeneous Schrödinger type equation
\[\eqnum\label{eq:non-homogeneous-schroedinger}
\begin{cases}
(i \partial_{t} + \Lap) v = h(t, x) & \text{on } \R \times \R^{d},
\\
v (0, x) = v_{0}(x) \,. &
\end{cases}
\]
By the Duhamel's formula, the solution $v$ of \eqref{eq:non-homogeneous-schroedinger} is given by 
\[\eqnum\label{eq:duhamel}
v (t, x) = e^{ i t \Lap} v_{0} -i\int_{0}^{t} e^{i (t - s) \Lap} h (s, x) \dd s \,.
\]

To control the non-homogeneous term $h$ we use the norms
\(X^{*, \sigma}(I) \), related through duality to the space
$X^{\sigma} (I)$. Specifically, we define
\[\eqnum\label{eq:def:X-norm-dual}
\begin{aligned}[c]
& \|h\|_{X^{*, \sigma} (I)} \eqd \big( \sum_{\theta\in\Theta} N^{2 \sigma} \|\LP_{\theta} h\|^{2}_{X_{\theta}^{*}(I)} \big)^{\frac{1}{2}},
\\
& \|h\|_{X_{\theta}^{*} (I)} \eqd \sup \Big\{ \Big| \int_{I \times \R^{d}} h_{*} (t, x ) h (t, x )  \dd t \dd x \Big| \st \|h_{*} \|_{X_{\theta} (I)} \leq 1 \Big\} \,.
\end{aligned}
\]    
Also, for any $r \in \R$ we denote
\begin{equation}\label{eq:jbnb}
 \langle r \rangle = (1 + r^2)^{\sfrac{1}{2}} 
\end{equation}
the Japanese bracket. 

\begin{proposition}\label{prop:main-linear-estimate}
Fix \(I\subseteq\R\), \(I\ni0\), and $s \in \R$. Fix any 
$0 < \epsilon\lesssim 1$ and any 
$0 < \epsilon_0 \lesssim_{\epsilon} 1$.   
Then, for any $v :I \times \R^{d} \mapsto \mathbb{C}$ given by \eqref{eq:duhamel},  there exists a constant \(c>0\) such that
\[\eqnum\label{eq:non-homogeneous-bound-2}
\|v\|_{X^{s} (I)} \lesssim \big\langle|I|^{-1}\big\rangle^{-c} \Big( \|v_{0} \|_{H^{s+\epsilon} (\R^{d})} +\|h\|_{X^{*, s+\epsilon} (I)} \Big)\,.
\]
If \(\diam\big(\spt(\FT{v_{0}})\big) \leq2R \) and \(\diam\big(\spt(\FT{h})\big)\leq2R\), then 
\[\eqnum\label{eq:non-homogeneous-bound-unit-scale-2}
\|v\|_{Y^{s} (I)}
\lesssim_{R}
\big\langle|I|^{-1}\big\rangle^{-c}  \Big(\|v_{0} \|_{H^{s+\epsilon} (\R^{d})} +\|h\|_{X^{*, s+\epsilon} (I)} \Big)\,.
\]
More specifically, for any $\theta \in \Theta$, we have 
\[\eqnum\label{eq:non-homogeneous-bound}
\|\LP_{\theta} v\|_{X_{\theta} (I)}
\lesssim N_\theta^{\epsilon}  \big\langle|I|^{-1}\big\rangle^{-c} \big( \| \LP_{\theta} v_{0} \|_{L^{2}_x (\R^{d})}+ \|\LP_{\theta} h\|_{X_{\theta} ^{*}(I)}  \big) 
\]
and, if \(\diam\big(\spt(\FT{v_{0}})\big) \leq2R\), and \(\diam\big(\spt(\FT{h})\big)\leq2R\), then
\[\eqnum\label{eq:non-homogeneous-bound-unit-scale}
\|\LP_{\theta} v\|_{Y_{\theta} (I)}
\lesssim_{R} N^{\epsilon}  \big\langle|I|^{-1}\big\rangle^{-c}  \big( \| \LP_{\theta} v_{0} \|_{L^{2}_x (\R^{d})}+ \|\LP_{\theta} h\|_{X_{\theta}^{*} (I)}  \big)\,.
\]
The constant \(c\) and all implicit constants are allowed to depend on \(\epsilon, \epsilon_{0}\) but are independent of \(s\), $I$, and \(\theta\). 
\end{proposition}
\begin{proof}
The proof is an immediate adaptation of the proof of  \cite[Proposition 3.1]{CFU}.
\end{proof}
Proceeding as in \cite[Corollary $3.3$]{CFU}, we obtain an estimate in more classical norms. 

\begin{corollary}\label{cor:linear-estimate-adjoint-and-continuity}
Under the assumptions of \Cref{prop:main-linear-estimate} 
we have
\[\eqnum\label{eq:continuity-bound}
\|v\|_{C^{0}(I, H^{s}(\R^{d}))}\lesssim \|v_{0}\|_{H^{s}(\R^{d})}+ \|h\|_{X^{*, s+\epsilon}(I)}.
\]
\end{corollary}

\section{Multilinear estimates}
We start this section by proving a bilinear estimates in terms of
directional space-time norms. 

\begin{lemma}\label{lem:bilinear-2-2-bound}
Fix any \(\epsilon\in(0, 1]\) and let \(0 < \epsilon_{0} \lesssim_{\epsilon} 1\) be as in the definition of $X^\sigma$ and $Y^\sigma$ norms. For any two functions
$h_{+}, h_{-} : I \times \R^{d} \to \C$ and any $\theta_{+}, \theta_{-} \in \Theta$ with $N_{\theta_{+}} \geq N_{\theta_{-}}$
it holds that 
\[\eqnum\label{eq:bilinear-bound}
\begin{aligned}
\|h_{+} h_{-} \|_{L^{2}_{t} L^{2}_{x}  (I \times \R^{d})} \lesssim_{\epsilon_{0}}N_{\theta_{+}}^{\epsilon}
N_{\theta_{+}}^{-\frac{\sigma-1}{2}}N_{\theta_{-}}^{\frac{\sigma}{4}}
\times
\begin{dcases}
N_{\theta_{-}}^{\frac{2d-\sigma-1}{4}} \|h_{+} \|_{X_{\theta_{+}} (\R) }  \|h_{-} \|_{X_{\theta_{-}} (\R) },
\\
\|h_{+} \|_{Y_{\theta_{+}} (\R) }  \| h_{-} \|_{Y_{\theta_{-}} (\R) },
\\
\| h_{+} \|_{X_{\theta_{+}} (\R) }  \| h_{-} \|_{Y_{\theta_{-}} (\R) },
\\
\| h_{+} \|_{Y_{\theta_{+}} (\R) }\| h_{-} \|_{X_{\theta_{-}} (\R) } \,.
\end{dcases}
\end{aligned}
\]
\end{lemma}

\begin{proof}
Fix \(j\in\ms{J}_{\theta_{+}}(O_{\theta_{-}})\). By Fubini's Theorem, and Hölder's inequality we obtain that 
\[
\big\|h_{+}  \, h_{-} \big\|_{L^{2}_{t} L^{2}_{x}  (\R \times \R^{d})}
\begin{aligned}[t]
& = \big\|h_{+}  h_{-} \big\|_{L_{O_{\theta_{-}}, j}^{(2, 2, 2)} (\R )} 
\\
& \lesssim \|h_{+}\|_{L_{O_{\theta_{-}}, j}^{(\frac{2}{\epsilon_{0}}, \frac{2}{1 - \epsilon_{0}}, \frac{2}{1-\epsilon_{0}} )} (\R) }  \| h_{-} \|_{L_{O_{\theta_{-}}, j}^{( \frac{2}{1 - \epsilon_{0}}, \frac{2}{\epsilon_{0}}, \frac{2}{\epsilon_{0}})} (\R) } \,.
\end{aligned}
\]
The first three bounds in \eqref{eq:bilinear-bound} directly follow
from the definitions \eqref{eq:def:X-norm} and \eqref{eq:def:Y-norm} of the norms \(X_{\theta_{+}}\) and \(Y_{\theta_{-}}\). To prove the last estimate in \eqref{eq:bilinear-bound}, we use Hölder's inequality to obtain the bound
\[
\big\|h_{+} h_{-} \big\|_{_{L_{O_{\theta_{-}}, j}^{(2, 2, 2)} (\R )}}
\lesssim
\|h_{+}\|_{L_{O_{\theta_{-}}, j}^{(\frac{2}{\epsilon_{0}}, \frac{2}{1 - \epsilon_{0}}, \frac{2}{\epsilon_{0}} )} (\R) }
\|h_{-} \|_{L_{O_{\theta_{-}}, j}^{( \frac{2}{1 - \epsilon_{0}}, \frac{2}{\epsilon_{0}}, \frac{2}{1-\epsilon_{0}})} (\R) }\,,
\]
and the assertion follows from the definitions \eqref{eq:def:X-norm} and \eqref{eq:def:Y-norm} of \(X_{\theta_{+}}\) and \(Y_{\theta_{-}}\).
\end{proof}

Next, using the bilinear estimates in Lemma \ref{lem:bilinear-2-2-bound}, we obtain quadrilinear bounds.

\begin{lemma}\label{lem:fourlinear} Let
\[
\frac{1}{3}-\frac{\sigma}{4}\leq S_1 \leq S_2 \leq S_3<\xs_c=\frac{d-\sigma}{2}<\xs<\xs_c+\frac{\sigma-1}{2}\,,
\]
and fix any $\sigma^\ast \in \R$, any $0 < \epsilon \lesssim 1$, and any $0 < \epsilon_0 \lesssim_{\epsilon, S_j, \xs} 1$.

For any function $h \in \{v_*, z_j, v_j\}$, $j \in\{ 1, 2, 3\}$, and $\theta \in \Theta $ we define 
quantities
\[\eqnum\label{eq:srnta}
Z_{\theta}( h) \eqd
\begin{cases}
\|v_*\|_{X_{\theta}(\R)} & \text{if } h = v_* \,,
\\
\|v_j\|_{X_{\theta}(\R)} & \text{if } h = v_j \,,
\\
\|z_j\|_{Y_{\theta}(\R)} & \text{if } h = z_j \,,
\end{cases}
\qquad 
\alpha(h) \eqd
\begin{cases}
-\sigma_{*} & \text{if } h = v_* \,,
\\
\xs & \text{if } h = v_j \,,
\\
S_j & \text{if } h = z_j \,.
\end{cases}    
\]
We denote $\sigma(q_1, \cdots, q_4)$ any permutation of a quadruple $(q_1, \cdots, q_4)$. 

For any 4 sectors \((\theta_{j}\in\Theta)_{j\in\{1, 2, 3, 4\}}\) we have that
\[\eqnum\label{eq:4-linear-orthogonality}
\int_{\R\times \R^{d}}\LP_{\theta_{1}}h_{1}\LP_{\theta_{2}}h_{2}\LP_{\theta_{3}}h_{3}\LP_{\theta_{4}}h_{4}\dd t\dd x
= 0
\]
unless \(N_{\theta_{j}}\lesssim\sum_{j'\neq j} N_{\theta_{j'}}\) for all \(j\in\{1, 2, 3, 4\}\). Furthermore, it holds that 
\[\eqnum\label{eq:4linear-form}
\begin{aligned}
\Big|
\int_{\R\times \R^{d}}\LP_{\theta_{1}}h_{1}\LP_{\theta_{2}}h_{2}\LP_{\theta_{3}}h_{3}\LP_{\theta_{4}}h_{4}\dd t\dd x
\Big| \hspace{-7em} & 
\\ & \lesssim (N_{\theta_1} N_{\theta_2}N_{\theta_3}N_{\theta_4})^{\epsilon}
\prod_{j = 1}^4 N_{\theta_j}^{\alpha(h_j)} \|\LP_{N_{\theta_j}} h_j\|_{Z_{\theta_j}(h_j)} \,,
\end{aligned}
\]
assuming that the functions $(h_j)_{j\in\{1,\ldots,4\}}$ and $\sigma_*$ are as in one of the following cases.

\noindent 
Case $zzz$:   
$(h_1, \cdots, h_4) = 
\sigma (v_*, z_1, z_2, z_3)$ and
\[\eqnum\label{eq:YYY}
    \sigma_{*} \leq S_{1}+ \min\big(S_{2}, \sfrac{\sigma}{4}\big)+\min\big(S_{3}, \sfrac{\sigma}{4}\big) + \frac{\sigma-2}{2}   \,.
\]
Case $zzv$: 
$(h_1, \cdots, h_4) = 
\sigma (v_*, z_1, z_2, v_3)$ and
\[\eqnum\label{eq:YYX}
    \sigma_{*} \leq S_{1}+\min\big(  S_2,  \sfrac{\sigma}{4} \big)+\min\big(\xs,\sfrac{\sigma}{4} \big)  + \frac{\sigma-2}{2}
\]
Case $zvv$:
 $(h_1, \cdots, h_4) = 
\sigma (v_*, z_1, v_2, v_3)$ and
\[\eqnum\label{eq:YXX}
\sigma_{*} \leq \min\big(2\xs -\xs_{c} + \frac{\sigma - 2}{4} + \min\{\sfrac{\sigma}{4}, S_1 \},
S_1+\sigma-1\big) \,,
\]
Case $vvv$:  $(h_1, \cdots, h_4) = 
\sigma (v_*, v_1, v_2, v_3)$  and
\[\eqnum\label{eq:XXX}
    \sigma_{*} \leq  \xs + 2(\xs - \xs_c) \,.
\]
\end{lemma}

\begin{proof} 
To simplify the notation we denote $N_j = N_{\theta_j}$ for each $j = 1, \cdots, 4$. 
Note that the expressions in \eqref{eq:4-linear-orthogonality} and \eqref{eq:4linear-form}, as well as the subsequent conditions, are the same if we exchange \(P_{\theta_{j}}h_{j}\) with \(P_{\theta_{j'}}h_{j'}\) \(j, j'\in\{1, 2, 3, 4\}\), and therefore we assume without loss of generality, that \(N_{1}\geq N_{2}\geq N_{3}\geq N_{4}\).

The assertion \eqref{eq:4-linear-orthogonality}
follows from Plancherel identity exactly as in Lemma 4.3 of \cite{CFU2}. 

We henceforth suppose that \(N_{1}\approx N_{2}\) and we proceed to prove \eqref{eq:4linear-form}.
To estimate the left-hand side of \eqref{eq:4linear-form} we apply the Cauchy-Schwarz inequality by pairing one function with the high frequency ($N_1$ or $N_2$) with a function with a lower frequency ($N_3$ or $N_4$). Hence, 
\[
\begin{aligned}
&  \Big|\int_{\R\times \R^{d}}\LP_{\theta_{1}}h_{1}\LP_{\theta_{2}}h_{2}\LP_{\theta_{3}}h_{3}\LP_{\theta_{4}}h_{4}\dd t\dd x\Big| \\
 & \hspace{5em} \leq  \min\Big(
\begin{aligned}[t]
& \|\LP_{\theta_{1}} h_{1}\LP_{\theta_{4}}h_{4} \|_{L^{2}_{t} L^{2}_{x} ( \R \times \R^{d} )}
\|\LP_{\theta_{2}} h_{2}\LP_{\theta_{3}}h_{3}\|_{L^{2}_{t} L^{2}_{x} (\R  \times \R^{d}  )} ,
\\
& \qquad\|\LP_{\theta_{1}} h_{1}\LP_{\theta_{3}}h_{3} \|_{L^{2}_{t} L^{2}_{x} ( \R \times \R^{d} )}
\|\LP_{\theta_{2}} h_{2}\LP_{\theta_{3}}h_{4}\|_{L^{2}_{t} L^{2}_{x} (\R  \times \R^{d}  )} \Big).
\end{aligned}
\end{aligned}
\]
By definitions in \eqref{eq:srnta} and \Cref{lem:bilinear-2-2-bound} for any \(l<l'\in\{1, 2, 3, 4\}\) it holds that 
\[
\begin{aligned}
\|\LP_{\theta_{l}} h_{l}\LP_{\theta_{l'}}h_{l'} \|_{L^{2}_{t} L^{2}_{x} ( \R \times \R^{d} )} &  \lesssim  N_l^{\epsilon -\sfrac{(\sigma - 1)}{2}-\alpha(h_{l})} N_{l'}^{\sfrac{\sigma}{4}  - \alpha(h_{l'})+\beta(h_{l}, h_{l'})} 
\\
& \hspace{5em}\times \prod_{j \in \{l, l'\}} N_j^{\alpha(h_j)}Z_{N_{j}}(\LP_{\theta_{j}} h_{j}) \,, \end{aligned}
\]
where
\[\eqnum
\beta(h_{l}, h_{l'}) \eqd
\begin{cases}
\frac{d-1}{2} - \frac{\sigma}{4} & \text{if } h_l, h_{l'} \in \{v_*, v_1, v_{2}, v_3\}\,, \\
0 & \text{otherwise}.
\end{cases}
\]
By homogeneity of the required bound \eqref{eq:4linear-form} we can assume, without loss of generality, that  \(N_j^{\alpha(h_j)}Z_{N_{j}}(\LP_{N_{j}} h_{j}) = 1\), \(j\in\{1, 2, 3, 4\}\). Thus, recalling that \(N_{1}\approx N_{2}\),  \eqref{eq:4linear-form} reduces to proving
\begin{multline}\eqnum\label{eq:sieww}
N_{3}^{\sfrac{\sigma}{4}-\alpha(h_{3})}N_{4}^{\sfrac{\sigma}{4}-\alpha(h_{4})}
\min\Big( N_{3}^{\beta(h_{2}, h_{3})}N_{4}^{\beta(h_{1}, h_{4})}, N_{3}^{\beta(h_{1}, h_{3})}N_{4}^{\beta(h_{2}, h_{4})}\Big)  
\\
\lesssim N_{1}^{\sigma - 1+\alpha(h_{1})+\alpha(h_{2})} \,.
\end{multline}  
We first claim that \(\alpha(h_{1})+\alpha(h_{2})+ \sigma - 1 \geq0\). Indeed, if $\{\alpha(h_{1}),\alpha(h_{2})\}\subset\{S_{1},S_{2},S_{3},\xs\}$, then \(\alpha(h_{1})+\alpha(h_{2})+ \sigma - 1 \geq 2S_{1}+ \sigma - 1 \geq0\), because $S_{1}\geq\frac{1}{3}-\frac{\sigma}{4}\geq \frac{1-\sigma}{2}$ by hypothesis. If $\alpha(h_{1})=-\sigma_*$, then $\alpha(h_2)\geq S_{1}$, hence \(\alpha(h_{1})+\alpha(h_{2})+ \sigma - 1 \geq -\sigma_* +S_1 + \sigma - 1 \geq 0\), which holds in the zzz, zzv and zvv cases by \eqref{eq:YYY}, \eqref{eq:YYX} and \eqref{eq:YXX}. In the vvv case we have the stronger condition $\alpha(h_2)\geq \xs$, hence \(\alpha(h_{1})+\alpha(h_{2})+ \sigma - 1 \geq -\sigma_* +\xs + \sigma - 1 \geq 0\), which holds by \eqref{eq:XXX}. The analysis of the case $\alpha(h_2)=-\sigma_*$ is analogous.

Thus, since \(N_{1}\geq N_{3}\) it is sufficient to show \eqref{eq:sieww} for \(N_{1}=N_{3}\), that is,  to show
\[\eqnum\label{eq:ctble}
\begin{aligned}  
\hspace{10em}& \hspace{-10em}\min\Big( N_{3}^{\beta(h_{2}, h_{3})}N_{4}^{\beta(h_{1}, h_{4})}, N_{3}^{\beta(h_{1}, h_{3})}N_{4}^{\beta(h_{2}, h_{4})}\Big)
\\
& \lesssim  N_{3}^{\sfrac{3\sigma}{4} -1 +\alpha(h_{1})+\alpha(h_{2})+\alpha(h_{3})}N_{4}^{-\sfrac{\sigma}{4}+\alpha(h_{4})}
\end{aligned}
\]
After standard algebraic manipulations, \eqref{eq:ctble} follows if we show that for any $N_3 \geq N_4$ either 
\[
N_4^{\beta(h_{1}, h_{4}) +\sfrac{\sigma}{4}-\alpha(h_{4}) } \lesssim N_3^{\sfrac{3\sigma}{4}-1+\alpha(h_{1})+\alpha(h_{2})+\alpha(h_{3}) - \beta(h_{2}, h_{3})}
\] 
or 
\[
N_4^{\beta(h_{2}, h_{4}) +\sfrac{\sigma}{4}-\alpha(h_{4}) } \lesssim N_3^{\sfrac{3\sigma}{4}-1+\alpha(h_{1})+\alpha(h_{2})+\alpha(h_{3}) - \beta(h_{1}, h_{3})} \,.
\]
Since $N_3 \geq N_4$ are arbitrary, it suffices to 
show that either 
\[
   \max\{0, \beta(h_{1}, h_{4}) +\sfrac{\sigma}{4}-\alpha(h_{4}) \}  \leq \sfrac{3\sigma}{4} - 1+\alpha(h_{1})+\alpha(h_{2})+\alpha(h_{3}) - \beta(h_{2}, h_{3})
\] 
or 
\[
  \max\{0, \beta(h_{2}, h_{4}) +\sfrac{\sigma}{4}-\alpha(h_{4}) \}  \leq \sfrac{3\sigma}{4} - 1 +\alpha(h_{1})+\alpha(h_{2})+\alpha(h_{3}) - \beta(h_{1}, h_{3})
\]
holds. In more symmetric form, it suffices to show 
\[\eqnum\label{eq:nlcn1}
\begin{multlined}
\beta(h_{2}, h_{3}) + \beta(h_{1}, h_{4}) +  \max\{\alpha(h_{4}) - \sfrac{\sigma}{4}   - \beta(h_{1}, h_{4}), 0 \}  \\
\leq \alpha(h_{1})+\alpha(h_{2})+\alpha(h_{3}) 
+\alpha(h_{4}) + \sfrac{\sigma}{2} - 1
\end{multlined}
\]
or 
\[\eqnum\label{eq:nlcn2}
\begin{multlined}
  \beta(h_{1}, h_{3}) + \beta(h_{2}, h_{4}) +  \max\{\alpha(h_{4})  -\sfrac{\sigma}{4} - \beta(h_{2}, h_{4}), 0\}  \\
  \leq \alpha(h_{1})+\alpha(h_{2})+\alpha(h_{3}) 
    +\alpha(h_{4}) + \sfrac{\sigma}{2} - 1
\end{multlined}
\]
Next, we discuss each case separately.

\noindent
\textbf{Case $zzz$}. Then, 
\[ 
\alpha(h_{1})+\alpha(h_{2})+\alpha(h_{3}) 
    +\alpha(h_{4}) \geq -\sigma_* + S_1 + S_2 + S_3    
\]
and $\beta(h_j, h_k) = 0$ for any $j \neq k$, and therefore \eqref{eq:nlcn1} is equivalent to 
\begin{equation}\label{eqn:zzzspec}
   \max\{\alpha(h_{4}) -  \sfrac{\sigma}{4}, 0 \}  \leq 
   -\sigma_* + S_1 + S_2 + S_3 + \sfrac{\sigma}{2} - 1 \,. 
\end{equation}

Notice that $\alpha(h_{4})\leq\max\{S_3,-\sigma^{\ast}\}$. If $\max\{S_3,-\sigma^{\ast}\}=S_3$, \eqref{eqn:zzzspec} holds if 
\[ 
 \sigma_* \leq  S_1 + S_2 + S_3 - 
 \max\{S_3 -  \sfrac{\sigma}{4}, 0 \} + \sfrac{\sigma}{2} - 1 = 
 S_1 + S_2  + \min\{S_3, \sfrac{\sigma}{4}\} + \sfrac{\sigma}{2} - 1 \,,
\]
which follows from \eqref{eq:YYY}. If $\max\{S_3,-\sigma^{\ast}\}=-\sigma^{\ast}$, then \eqref{eqn:zzzspec} holds provided
   \[ 
   \max\{-\sigma_{\ast} -  \sfrac{\sigma}{4}, 0 \}  \leq 
   -\sigma_* + S_1 + S_2 + S_3 + \sfrac{\sigma}{2} - 1 \,.
\]
If $\max\{-\sigma_{\ast} -  \sfrac{\sigma}{4}, 0 \}=0$, then
   \[ 
   \sigma_* \leq  S_1 + S_2 + S_3 + \sfrac{\sigma}{2} - 1 \,,
\]
which follows from \eqref{eq:YYY}. If $\max\{-\sigma_{\ast} -  \sfrac{\sigma}{4}, 0 \}=-\sigma_{\ast} -  \sfrac{\sigma}{4}$, then
   \[ 
   1 \leq  S_1 + S_2 + S_3 + \sfrac{3\sigma}{4} \,,
\]
which follows from $S_1\geq \frac{1}{3}-\frac{\sigma}{4}$.

\noindent
\textbf{Case $zzv$}. Then,
\[ 
\alpha(h_{1})+\alpha(h_{2})+\alpha(h_{3}) 
    +\alpha(h_{4}) \geq -\sigma_* + S_1 + S_2 + \xs    
\]
and either $\beta(h_1, h_4) = \beta(h_2, h_3) =  0$ or 
$\beta(h_1, h_3) = \beta(h_2, h_4) = 0$. Suppose the former holds true, the latter case 
follows analogously by proving \eqref{eq:nlcn2} instead of \eqref{eq:nlcn1}. Then, \eqref{eq:nlcn1} 
is equivalent to 
\begin{equation}\label{eqn:zzvspec}
  \max\{\alpha(h_{4}) -  \sfrac{\sigma}{4}, 0\}  \leq -\sigma_* + S_1 + S_2 + \xs + \sfrac{\sigma}{2} - 1 \,.
\end{equation}

Observe that $\alpha(h_{4})\leq\max\{\mathfrak{s},-\sigma^{\ast}\}$. If $\max\{\mathfrak{s},-\sigma^{\ast}\}=\mathfrak{s}$, then \eqref{eqn:zzvspec} follows if we show that 
\begin{equation}
  \max\{\xs -  \sfrac{\sigma}{4}, 0 \}  \leq -\sigma_* + S_1 + S_2 + \xs + \sfrac{\sigma}{2} - 1 \,,
\end{equation}
which is equivalent to 
\[ 
  \sigma_*   \leq  S_1 + S_2 + 
  \min\{\sfrac{\sigma}{4}, \xs\} + \sfrac{\sigma}{2} - 1\,,
\] 
and the assertion follows from \eqref{eq:YYX}. If $\max\{\mathfrak{s},-\sigma_{\ast}\}=-\sigma_{\ast}$, it's enough to check that
\[ 
  \max\{-\sigma_{\ast} -  \sfrac{\sigma}{4}, 0\}  \leq -\sigma_* + S_1 + S_2 + \xs + \sfrac{\sigma}{2} - 1 \,.
\]
If $\max\{-\sigma_{\ast} -  \sfrac{\sigma}{4}, 0 \}=0$, then
\[ 
  \sigma_* \leq S_1 + S_2 + \xs + \sfrac{\sigma}{2} - 1 \,,
\]
which follows from \eqref{eq:YYX}. If $\max\{-\sigma_{\ast} -  \sfrac{\sigma}{4}, 0 \}=-\sigma_{\ast} -  \sfrac{\sigma}{4}$, then
 \[ 
   1 \leq  S_1 + S_2 + \xs + \sfrac{3\sigma}{4} \,,
\]
which follows from $S_1\geq \frac{1}{3}-\frac{\sigma}{4}$.

\noindent
\textbf{Case $zvv$}. Assuming without loss of generality that the $z$ term is $z_1$, 
\[ 
\alpha(h_{1})+\alpha(h_{2})+\alpha(h_{3}) 
    +\alpha(h_{4}) = -\sigma_* + S_1 + \xs + \xs 
\]
and 
\[\eqnum
 \beta(h_{1}, h_{4}) + \beta(h_{2}, h_{3}) =  \beta(h_{1}, h_{3}) + \beta(h_{2}, h_{4}) = \frac{d - 1}{2} - \frac{\sigma}{4} \,.
\]
If $\beta(h_{1}, h_{4}) = \beta(h_{2}, h_{4}) = 0$, then $h_4 = z_1$, and \eqref{eq:nlcn1} is equivalent to 
\begin{equation}
\label{eqn:zvvspec}
  \frac{d-1}{2} - \frac{\sigma}{4} + \max\{S_1 -  \sfrac{\sigma}{4}, 0 \}  \leq -\sigma_* + S_1 + 2\xs  + \sfrac{\sigma}{2} - 1 \,.  
\end{equation}
Standard algebraic manipulations and $\xs_{c} = \frac{d-\sigma}{2}$
yield 
\[\eqnum
 \sigma_* \leq  2\xs -\xs_{c} + \frac{\sigma - 2}{4} + \min\{\sfrac{\sigma}{4}, S_1 \}    
\]
and \eqref{eqn:zvvspec} follows from \eqref{eq:YXX}. 

Next, assume without loss of generality that 
$\beta(h_{1}, h_{4}) = \frac{d-1}{2}- \frac{\sigma}{4}$, otherwise  $\beta(h_{2}, h_{4}) = \frac{d-1}{2}- \frac{\sigma}{4}$ and we prove \eqref{eq:nlcn2} instead of \eqref{eq:nlcn1}. Again, $\alpha(h_{4})\leq\max\{\mathfrak{s},-\sigma^{\ast}\}$. If $\max\{\mathfrak{s},-\sigma^{\ast}\}=\mathfrak{s}$, \eqref{eq:nlcn1} holds if we prove 
\[\eqnum\label{eq:agzvvs}
  \frac{d-1}{2} - \frac{\sigma}{4} + \max\Big\{\xs - \frac{d-1}{2}, 0 \Big\}  \leq -\sigma_* + S_1 + 2\xs + \sfrac{\sigma}{2} - 1\,.
\]
Since $\xs \leq \xs_{c} + \frac{\sigma - 1}{2} = \frac{d - 1}{2}$, then \eqref{eq:agzvvs} and $\xs_{c} = \frac{d-\sigma}{2}$ is equivalent to
\[\eqnum\label{eq:agzvvs2}
  \sigma_*  \leq   S_1 + 2\xs - \xs_{c} + \frac{\sigma - 2}{4} \,,
\]
which
follows from  \eqref{eq:YXX}.  If $\max\{\mathfrak{s},-\sigma_{\ast}\}=-\sigma_{\ast}$, it's enough to verify that
\[
  \frac{d-1}{2} - \frac{\sigma}{4} + \max\Big\{-\sigma_{\ast} - \frac{d-1}{2}, 0 \Big\}  \leq -\sigma_* + S_1 + 2\xs + \frac{\sigma}{2} - 1\,.
\]
If $\max\Big\{-\sigma_{\ast} - \frac{d-1}{2}, 0 \Big\}=0$, then 
\[
  \sigma_*   \leq 2\xs - \xs_{c} + \frac{\sigma}{4} - \frac{1}{2} + S_{1}\,,
\]
which follows from \eqref{eq:YXX}. If $\max\Big\{-\sigma_{\ast} - \frac{d-1}{2}, 0 \Big\}=-\sigma_{\ast} - \frac{d-1}{2}$, then
$$1\leq S_{1} + 2\xs +\frac{3\sigma}{4},$$
which follows from $S_1\geq \frac{1}{3}-\frac{\sigma}{4}$.

\noindent
\textbf{Case $vvv$}. Then, 
\[ 
\alpha(h_{1})+\alpha(h_{2})+\alpha(h_{3}) 
    +\alpha(h_{4}) = -\sigma_* + 3\xs    
\]
and $\beta(h_j, h_k) = \frac{d-1}{2} - \frac{\sigma}{4}$ for each $j \neq k$. Then, \eqref{eq:nlcn1} follows if we show
\begin{equation}\label{eqn:vvvspec}
d-1 - \frac{\sigma}{2} + \max\{\alpha(h_{4}) -  \frac{d-1}{2}, 0 \}  \leq -\sigma_* + 3\xs + \frac{\sigma}{2} - 1 \,.
\end{equation}
Again, $\alpha(h_{4})\leq\max\{\mathfrak{s},-\sigma^{\ast}\}$. If $\max\{\mathfrak{s},-\sigma^{\ast}\}=\mathfrak{s}$, as in the case $zvv$ we have $\xs \leq  \frac{d-1}{2}$, and we obtain an equivalent expression
\[ 
\sigma_*  \leq  3\xs - (d-1) + \sigma - 1 = 3\xs - 2\xs_{c}
  \,,
\] 
and \eqref{eqn:vvvspec} follows from \eqref{eq:XXX}. If $\max\{\mathfrak{s},-\sigma_{\ast}\}=-\sigma_{\ast}$, it's enough to check that
\[ 
d-1 - \frac{\sigma}{2} + \max\{-\sigma_{\ast} -  \frac{d-1}{2}, 0 \}  \leq -\sigma_* + 3\xs + \frac{\sigma}{2} - 1 \,.
\] 
If $\max\{-\sigma_{\ast} -  \frac{d-1}{2}, 0 \}=0$, this becomes
\[ 
\sigma_* \leq 3\xs -2\xs_{c} \,,
\] 
which follows from \eqref{eq:XXX}. If $\max\{-\sigma_{\ast} -  \frac{d-1}{2}, 0 \}=-\sigma_{\ast} -  \frac{d-1}{2}$, we must have
\[ 
\frac{d-\sigma}{2} + \frac{1}{2} -\frac{\sigma}{2}  \leq 3\xs \,.
\] 
Given that $\frac{d-\sigma}{2}\leq\xs$, this holds if $\xs\geq \frac{1}{4}-\frac{\sigma}{4}$, which holds if $\xs_{c}\geq \frac{1}{4}-\frac{\sigma}{4}$, which is equivalent to $d\geq\frac{1}{2}+\frac{\sigma}{2}$. This is true because $\sigma\leq d$. 

\end{proof}

\begin{proposition}\label{prop:trilinear-estimate} Fix any 
\[\eqnum\label{eq:trilinear-exponent-assumptions}
    \frac{1}{3}-\frac{\sigma}{4} < S_1 \leq S_2 \leq S_3 < \xs_c < \xs 
    \,,
\]
and choose any \( 0 < \epsilon \lesssim_{S_{j}, \xs}\!1\) and
\(0 < \epsilon_{0} \lesssim_{\epsilon, S_j, \xs} \!1\).  Then for any
$z_j$ and \(v_j\), $j\in\{1, 2, 3\}$, the following estimates hold:
\[\eqnum\label{eq:YYYg}
\big\|z_{1}z_{2}z_{3}\big\|_{X^{*, \sigma_{*}}(\R)}\lesssim\|z_{1}\|_{Y^{S_{1}+\epsilon  }(\R)}\|z_{2}\|_{Y^{S_{2}+\epsilon}(\R)}\|z_{3}\|_{Y^{S_{3}+\epsilon}(\R)}
\quad \sigma_{*} \textrm{ as in } \eqref{eq:YYY} \,.
\]
\[\eqnum\label{eq:YYXg}
 \big\|z_{1} z_{2}v_{3}\big\|_{X^{*, \sigma_{*}}(\R)} \lesssim\|z_{1}\|_{Y^{S_{1}+\epsilon  }(\R)}\|z_{2}\|_{Y^{S_{2}+\epsilon}(\R)}\|v_{3}\|_{X^{\xs+\epsilon}(\R)}
 \quad  \sigma_{*} \text{ as in } \eqref{eq:YYX} \,.
\]
\[\eqnum\label{eq:YXXg}
 \big\|z_{1}v_{2}v_{3}\big\|_{X^{*, \sigma_{*}}(\R)} \lesssim
\|z_{1}\|_{Y^{S_1+\epsilon}(\R)}\|v_{2}\|_{X^{\xs+\epsilon}(\R)}\|v_{3}\|_{X^{\xs+\epsilon}(\R)}
 \quad  \sigma_{*}\text{ as in } \eqref{eq:YXX} \,.
\]
\[\eqnum\label{eq:XXXg}
 \big\|v_{1}v_{2}v_{3}\big\|_{X^{*, \sigma_{*}}(\R)} \lesssim \|v_{1}\|_{X^{\xs+\epsilon}(\R)}\|v_{2}\|_{X^{\xs+\epsilon}(\R)}\|v_{3}\|_{X^{\xs+\epsilon}(\R)}
 \quad  \sigma_{*}
 \text{ as in } \eqref{eq:XXX} \,.
\]
The implicit constants are allowed to depend on \(\epsilon\), \(\epsilon_{0}\), $\xs$, and $S_{j}$. 
\end{proposition}

The proof of \Cref{prop:trilinear-estimate} is a minor
modification of \cite[Lemma 4.2]{CFU2}.

We end this subsection, with a nonlinear estimate which is  relevant for probabilistic estimates on $z_i$.
For any $S \in \R$, $n \geq 1$ define
\[\eqnum\label{eq:def:mu-v2}
\mu(n, S) := \min\Big(nS+\frac{(n-1)(\sigma-2)}{4}, \, 2S + \frac{3\sigma}{4}-1, \, S +\sigma - 1 \Big) .
\]
Observe that the definition of $\mu (n, S)$ is related to the zzz case of Lemma \ref{lem:fourlinear}.

\begin{claim} If $S>\frac{1}{2}-\frac{\sigma}{4}$, then $\mu(k_{1},S)\leq\mu(k_{2},S)$ if $k_{1}<k_{2}$.
\end{claim}
\begin{proof} The proof is elementary.
\end{proof}

\begin{lemma}[Inductive property of $\mu(k, S)$]\label{lem:mu-inductive-property} For any  any \(k_{1}, k_{2}, k_{3}\in\N\setminus\{0\}\), with
\(k_{1}\leq k_{2}\leq k_{3}\), and any \(S>\frac{1}{2}-\frac{\sigma}{4}\), it holds that
\[\eqnum\label{eq:mu-inductive-property}
\mu(k_{1}+k_{2}+k_{3}, S) \leq\mu(k_{1}, S)+ \min\big(\mu(k_{2}, S), \sfrac{\sigma}{4}\big)+ \min\big(\mu(k_{3}, S), \sfrac{\sigma}{4}\big) + \frac{\sigma}{2} - 1.
\]

Furthermore, set \(k\eqd k_{1}+k_{2}+k_{3}\) and fix
\(0 < \epsilon \lesssim_{k} 1\) and 
\(0 < \epsilon_{0} \lesssim_{k, S, \epsilon} 1\). Then for any functions
$z_j$, $j \in \{1, 2, 3\}$, it holds that
\[\eqnum\label{eq:trilinear-Y-mu-bound}
\|z_{1}z_{2}z_{3}\|_{X^{*, \mu(k, S)}}\lesssim \prod_{j = 1}^{3} \|z_{j}\|_{Y^{\mu(k_{j}, S)+\epsilon}} \,,
\]
where the implicit constant depends on $k$, $\epsilon$, $\epsilon_0$, and $S$. 

\end{lemma}
\begin{proof}[Proof of \Cref{lem:mu-inductive-property}]
Without loss of generality, suppose that \(k_{1}\leq k_{2}\leq k_{3}\). Let
\[
J\eqd\mu(k_{1}, S)+ \min\big(\mu(k_{2}, S), \sfrac{\sigma}{4}\big)+ \min\big(\mu(k_{3}, S), \sfrac{\sigma}{4}\big) + \frac{\sigma}{2} - 1.
\]
We first observe that $\mu(k_{j}, S) \geq S$, \(j\in\{1, 2, 3\}\), which follows from the restriction $S\geq\frac{1}{2}-\frac{\sigma}{4}$. If $\mu(k_{3}, S)\geq \mu(k_{2}, S) \geq \frac{\sigma}{4}$ then 
\[
 J \geq S + \frac{\sigma}{4} + \frac{\sigma}{4} + \frac{\sigma}{2} - 1 = S + \sigma - 1 \geq  \mu(k_{1}+k_{2}+k_{3}, S)\,, 
\]
and \eqref{eq:mu-inductive-property} follows. If $\mu(k_{3}, S) \geq \frac{\sigma}{4} \geq \mu(k_{2}, S)$, then
\[
J \geq S + S + \frac{\sigma}{4} + \frac{\sigma}{2} - 1  = 2S + \frac{3\sigma}{4} - 1 \geq  \mu(k_{1}+k_{2}+k_{3}, S) \,,
\]
and \eqref{eq:mu-inductive-property} follows.  Finally, if $ \frac{\sigma}{4} \geq \mu(k_{3}, S)\geq \mu(k_{2}, S)$ then
\[
 J = \mu(k_{1}, S) + \mu(k_{2}, S)  + \mu(k_{3}, S) + \frac{\sigma}{2} - 1
\]
and in addition, \(\min(S+\sigma - 1, 2S+\frac{3\sigma}{4} - 1)>\frac{\sigma}{4}\) implies 
\[
\mu(k_{j}, S) = k_{j}S+\frac{(k_{j}-1)(\sigma-2)}{4},\quad j\in\{1, 2, 3\} \,.
\]
 Then we deduce that 
\[
 J =  \big(k_{1}+k_{2}+k_{3}\big)S + \frac{(k_{1}+k_{2}+k_{3}-3)(\sigma-2)}{4} +\frac{\sigma-2}{2} \geq\mu(k_{1}+k_{2}+k_{3}, S)
 ,
\]
and \eqref{eq:mu-inductive-property} follows. Finally, \eqref{eq:trilinear-Y-mu-bound} follows from  \eqref{eq:YYY} with \(S\) replaced by \(S+\epsilon\). This completes the proof.
\end{proof}

\section{Probabilistic estimates}

We define the set \(\TT\) of ternary trees as
$\TT := \bigcup_{n \geq 1} \TT_{n}$, where for each
\(n\in \N\setminus{0}\) the set of trees \(\TT_{n}\) is given by induction as
follows:
\begin{itemize}
\item We say $\tau \in \TT_{1}$ if \(\tau = [ \bullet ]\).
\item We say $\tau \in \TT_{n}$ for $n \geq 1$ if \(\tau = [\tau_{1}, \tau_{2}, \tau_{3}]\) for some $\tau_{j} \in \TT_{n_{j}}$ and $n = n_{1} + n_{2} + n_{3}$ with $1 \leq n_{j} < n$.
\end{itemize}

To each tree \(\tau\in\TT_{n}\) with \(n\in2\N+1\) we inductively assign an
\(n\)-(real) linear tree operator \(R_{\tau}\) mapping \(n\)-tuples of
functions \((f_{1}, \ldots, f_{n})\in L^{2}(\R^{d})\) to functions on \(\R\times\R^{d}\).
\begin{itemize}
\item We set \(R_{[\bullet]}[f](t, x) \eqd e^{it\Lap}f(x).\)
\item Inductively, we set
\[\eqnum\label{eq:def:R-inductive}
\begin{aligned}[c]
\hspace{5em}&\hspace{-5em} R_{[\tau_{1}, \tau_{2}, \tau_{3}]}[\mb{f}_{1}\oplus \mb{f}_{2}\oplus\mb{f}_{3}](t, \cdot)
\\
& \eqd \mp i \int_{0}^{t} e^{i (t - s) \Lap} \Big(R_{\tau_{1}} [\mb{f}_{1}] (s, \cdot)\, \overline{R_{\tau_{2}} [\mb{f}_{2}] (s, \cdot)}\,  R_{\tau_{3}} [\mb{f}_{3}] (s, \cdot)\Big) \dd s \,,
\end{aligned}
 \]
 where the choice of the sign is exactly the opposite to the sign on the right hand side of \eqref{eq:intro}. Here \(\mb{f}_{j}=(f_{j, 1}, \ldots, f_{j,|\tau_{j}|})\in \big(L^{2}(\R^{d})\big)^{|\tau_{j}|}\) are \(|\tau_{j}|\)-tuples of functions in \(L^{2}(\R^{d})\) and 
\[
\mb{f}_{1}\oplus \mb{f}_{2}\oplus\mb{f}_{3}=\big(f_{1, 1}, \ldots, f_{1,|\tau_{1}|}, f_{2, 1}, \ldots, f_{2,|\tau_{2}|}, f_{3, 1}, \ldots, f_{3,|\tau_{3}|}^{3}\big)\,.
\]
\end{itemize}

To simplify notation, given \(\tau\in\TT\) and a scalar function \(f \in L^{2}(\R^{d})\), we write
\[
R_{\tau}[f] := R_{\tau}\big[f, \ldots, f\big] \,,
\]
where the right-hand side contains $|\tau|$ copies of $f$. Thanks to our definition of $\mu (n, S)$, we can proceed as in \cite[Proposition $5.1$]{CFU} to obtain the following result.

\begin{proposition}\label{prop:tree-bound-probabilistic} 
Fix \(\tau\in\TT\), \(S\in\mathbb{R}\) and let \(0<\epsilon \lesssim_{|\tau|} 1\) be small. Then for
every \(0 < \epsilon_{0}\lesssim_{\epsilon, |\tau|, S} 1\) there exist constants
$C > 0$ and $c > 0$ depending $\epsilon_{0}$, \(S\), \(\epsilon\), and on
\(|\tau|\), but independent of $f\in L^{2}(\R^{d})$, such that for any
$\lambda > 0$ it holds that
\[\eqnum\label{eq:tree-bound-probabilistic}
\mathbb{P} \Big(\big\{\omega \in \Omega : \big\| R_{\tau}[\rf] \big\|_{Y^{\mu(|\tau|, S)} (\R)} > \lambda\big\}\Big)
\leq C \exp\Bigg(-c \frac{ \lambda^{\frac{2}{|\tau|}}}{\|f\|_{H_{x}^{S +\epsilon} (\R^{d})}^{2}}\Bigg)\,,
\]
where \(\rf\) is the randomization of \(f\) given by
\eqref{eq:randomization}. In particular,
\[
\big\|R_{\tau}[\rf] \big\|_{Y^{\mu(|\tau|, S)} (\R)} < \infty \qquad \text{almost surely}
\]
for any \(f\in H^{S+\epsilon}(\R^{d})\).
\end{proposition}

\section{Iteration scheme}
 
For any \(t_{0}\geq0\), $z \in Y^S$, and $v \in X^{\xs}$
we replace $h$ by $\Phi_{z}[v] +h$ in 
\eqref{eq:duhamel} and define
\[\eqnum\label{eq:def:delayed-iteration-map}
\mc{K}(v) := e^{i(t-t_{0})\Lap}v_{0}\mp i\int_{t_{0}}^{t}e^{i(t-s)\Lap}(\Phi_{z}[v] +h)\dd s \,,
\]
where
\[\eqnum\label{eq:remainder-cubic-nonlinearity}
\begin{aligned}[t]
\Phi_{z}[v] & \eqd\big|z+v\big|^{2}\big(z+v\big)-\big|z\big|^{2}z
 = \sum_{\hspace{-3em}\mrl{\substack{
    \mrl{h_{1},h_{2},h_{3}\in\{z,v\}}
    \\
\mrl{ (h_{1},h_{2},h_{3})\neq (z,z,z)}
  }}} h_{1}\bar{h_{2}}h_{3} \,.
\end{aligned}
\]
In the definition of $\mc{K}(v)$ we consider arbitrary $t_0 \geq 0$, because it is used in the proof of the uniqueness.  
Combining the multilinear estimates from Lemma \ref{lem:fourlinear} and linear estimates from  \Cref{prop:main-linear-estimate}, we obtain the following result.

\begin{lemma}\label{lem:iteration-map-bound} Fix an interval \(I\subseteq\R\) and assume \(0<S<\xs_{c}<\xs\) satisfy
\[\eqnum\label{eq:asoxs}
\xs < S+\sigma- 1 \quad \text{and if $\xs>\frac{\sigma}{4}$, then in addition }
\xs < 2S+\frac{3\sigma-4}{4} \,.
\]
Choose \(0<\epsilon\lesssim_{S, \xs}1\) and a small \(0<\epsilon_{0}\lesssim_{\epsilon, S, \xs}1\).  Then
there exist constants \(C, c>0\) depending on \(\epsilon\), \(\epsilon_0\), \(S\), \(\xs\)
such that for any $v_0 \in H^{\xs+\epsilon}(\R^{d})$, \(z\in Y^{S}(I)\), \(h\in X^{\xs
+ \epsilon}(I)\), and \(v_{1}, v_{2}\in X^{\xs}(I)\) the following bounds hold:
\[\eqnum\label{eq:iteration-map-bound}
\begin{aligned}
\Big\|\mc{K}(v_1)\Big\|_{X^{\xs}(I)}
 \leq C \big\langle |I|^{-1}\big\rangle^{-c} 
\begin{aligned}[t]
\Big(& \|v_{0}\|_{H_x^{\xs+\epsilon} (\R^{d})} + \|h\|_{X^{*, \xs+\epsilon}(I)}
\\
& \quad+ \big(\|v_{1}\|_{X^{\xs}(I)}+\|z\|_{Y^{S}(I)}\big)^{2}\|v_{1}\|_{X^{\xs}(I)} \Big)
\end{aligned}
\end{aligned}
\]
and 
\[\eqnum\label{eq:iteration-map-lipschitz}
\begin{aligned}
& \Big\|\mc{K}(v_1) - \mc{K}(v_2)\Big\|_{X^{\xs}(I)}
\\
&\qquad \leq C \big\langle |I|^{-1}\big\rangle^{-c} \big(\|v_{1}\|_{X^{\xs}(I)}+\|v_{2}\|_{X^{\xs}(I)}+\|z\|_{Y^{S}(I)}\big)^{2}\big\|v_{1}-v_{2}\big\|_{X^{\xs}(I)}\,.
\end{aligned}
\]
Note that $\langle \cdot \rangle$ was defined in \eqref{eq:jbnb}. 
As a consequence:
\begin{itemize}
\item For any \(v_{0}\in H^{\xs+\epsilon}(\R^{d})\), \(z\in Y^{S}(I)\), \(h\in X^{*, \xs+\epsilon}(I)\), and $\delta_0 > 0$ we define
\[\eqnum\label{eq:contraction-time-bound}
\begin{aligned}
& T_{\delta_{0}}\big(\|v_{0}\|_{H^{\xs+\epsilon}(\R^{d})}, \|z\|_{Y^{S}(I)}, \|h\|_{X^{*, \xs+\epsilon}(I)}\big)
\\
& \quad\eqd\frac{1}{2^{100}(C\delta_{0})^{\sfrac{1}{c}}}\big\langle\delta_0+\|v_{0}\|_{H_x^{\xs+\epsilon}(\R^{d})} + \|h\|_{X^{*, \xs+\epsilon}(\R)} +\|z\|_{Y^{S}(\R)}\big\rangle^{-\sfrac{3}{c}} \,.
\end{aligned}
\]
Then for any interval \(J=[t_{0}, t_{1}]\subseteq I\) with
\[
|J|<T_{\delta_{0}}\big(\|v_{0}\|_{H^{\xs+\epsilon}(\R^{d})}, \|z\|_{Y^{S}(I)}, \|h\|_{X^{*, \xs+\epsilon}(I)}\big)
\]
the map $v \mapsto \mc{K}(v)$ is a contraction on the set
\[
\bar{\mathcal{B}}_{\delta_{0}}^{J}\eqd\Big\{v\in X^{\xs}(J)\st \|v\|_{ X^{\xs}(J)}\leq\delta_{0}\Big\}\,,
\]
and thus it admits a unique fixed point in \(\bar{\mathcal{B}}_{\delta_{0}}^{J}\).
\end{itemize}
\end{lemma}

\begin{proof}[Proof of Lemma \ref{lem:iteration-map-bound}]
Fix \(z\in Y^{S}(I)\), \(h\in X^{\xs + \epsilon}(I)\), and \(v_{1}, v_{2}\in X^{\xs}(I)\). Using \eqref{eq:non-homogeneous-bound-2}, we have
\[
\begin{aligned}[t]
\Big\|\mathcal{K}(v_1)\Big\|_{X^{\xs}(I)}
\leq C \big\langle |I|^{-1}\big\rangle^{-c} \Big(\|v_{0}\|_{H^{\xs+\epsilon} (\R^{d})} + \|h\|_{X^{*, \xs+\epsilon}(I)}+ \big\|\Phi_{z}[v_{1}]\big\|_{X^{*, \xs+\epsilon}(I)}\Big)\, ,
\end{aligned}
\]
and
\[
\Big\|\mc{K}(v_1) - \mc{K}(v_2)\Big\|_{X^{\xs}(I)}
\leq C  \big\langle |I|^{-1}\big\rangle^{-c} \big\|\Phi_{z}[v_{1}]-\Phi_{z}[v_{2}]\big\|_{X^{*, \xs+\epsilon}(I)}.
\]
So to prove \eqref{eq:iteration-map-bound} and \eqref{eq:iteration-map-lipschitz}, we only have to control the terms involving $\Phi_z$. We can rewrite \(\Phi_{z}[v]\) (defined in \eqref{eq:remainder-cubic-nonlinearity} as 
\[
\Phi_{z}[v_{1}]-\Phi_{z}[v_{2}]=G_{1}\big[z, \bar{z}, v_{1}, \overline{v_{1}}, v_{2}, \overline{v_{2}}\big] (v_{1}-v_{2})+G_{2}\big[z, \bar{z}, v_{1}, \overline{v_{1}}, v_{2}, \overline{v_{2}}\big](\overline{v_{1}-v_{2}}) \,,
\]
where \(G_{j}[z, \bar{z}, v_{1}, \overline{v_{1}}, v_{2}, \overline{v_{2}}] \), \(j\in\{1, 2\}\), are homogeneous polynomials of degree 2 in the variables \(z\), \(\bar{z}\), \(v_{1}\), \(\bar{v_{1}}\), \(v_{2}\), \(\bar{v_{2}}\). 
%Assume $\epsilon < \min\big(S, \xs-\xs_c\big)$ and
Note that from \eqref{eq:asoxs} it follows that \eqref{eq:YYX}, \eqref{eq:YXX}, and \eqref{eq:XXX} are satisfied with $\sigma_{*}$ replaced by $\xs$ and $S_j = S$, $j = \{1, 2, 3\}$. Indeed, the right hand sides in \eqref{eq:YYX}, \eqref{eq:YXX}, and \eqref{eq:XXX} contain $S+\sigma-1$, which is assumed in \eqref{eq:asoxs}. Additionally, the right hand sides might contain $\xs$ and such inequality is trivially satisfied. Finally, \eqref{eq:YYX} either contain $\xs$ or $\xs > \frac{\sigma}{4}$, with inequality assumed in \eqref{eq:asoxs}. 
Then \Cref{prop:trilinear-estimate} implies
\[
\begin{aligned}
\Big\|\Phi_{z}[v_{1}]-\Phi_{z}[v_{2}]\Big\|_{X^{*, \xs +\epsilon}(\R)}
& \leq 
\begin{aligned}[t]
& \Big\|G_{1}\big[z, \bar{z}, v_{1}, \overline{v_{1}}, v_{2}, \overline{v_{2}}\big] \big(v_{1}-v_{2}\big)\Big\|_{X^{*, \xs +\epsilon}(\R)}
\\
& \qquad +\Big\|G_{2}\big[z, \bar{z}, v_{1}, \overline{v_{1}}, v_{2}, \overline{v_{2}}\big] \big(\overline{v_{1}-v_{2}}\big)\Big\|_{X^{*, \xs +\epsilon}(\R)}
\end{aligned}
\\
& \lesssim\big(\|v_{1}\|_{X^{\xs}(I)}+\|v_{2}\|_{X^{\xs}(I)}+\|z\|_{Y^{S}(I)}\big)^{2}\|v_{1}-v_{2}\|_{X^{\xs}(I)}\,.
\end{aligned}
\]
Setting $v_2=0$, we obtain 
\[\eqnum\label{eq:F-bound}
\big\|\Phi_{z}[v_{1}]\big\|_{X^{*, \xs +\epsilon}(\R)}\lesssim\big(\|v_{1}\|_{X^{\xs}(\R)}+\|z\|_{Y^{S}(\R)}\big)^{2}\|v_{1}\|_{X^{\xs}(\R)}.
\]
and \eqref{eq:iteration-map-bound}, \eqref{eq:iteration-map-lipschitz} follow. Then, \(\mc{K}\) is a contraction map on \(\bar{\mathcal{B}}_{\delta_{0}}^{\R}\) and the remaining statements follow from the Banach fixed point theorem.

\end{proof}

\section{Proof of Theorem \ref{thm:main}}

Before stating main results of this section, let us introduce some
notation. For any $n \in 2\N+1$
we define the random variables
\(\rz_{n}: \R\times\R^{d}\to \mathbb{C}\) as
\[\eqnum\label{eq:def:zn}
\rz_{n}  \eqd\sum_{\tau \in \TT_{n}} R_{\tau} [\rf]
\]
and we set \(\rz_{2n}=0\) by convention. We also define
\[\eqnum\label{eq:def:zn-2}
\rz_{\leq M}\eqd\sum_{n\leq M}\rz_{n},
\qquad
[\rz, \rz, \rz]_{> M}\eqd\sum_{\substack{n_{1}+n_{2}+n_{3}> M, \\ n_{j}\leq M}} \rz_{n_{1}}\overline{\rz_{n_{2}}} \rz_{n_{3}}\,.
\]
Then for each integer $M \geq 1$, $u$ is a solution to \eqref{eq:intro} if 
\[
u= \rz_{\leq M}+v,
\]
and $v \in X^{\xs}(I)\cap C^0 (I; H^{\xs_c})$ solves
\[\eqnum\label{eq:iteration-map}
\mc{I}(v)\eqd e^{it\Lap}v_{0}\mp i\int_{0}^{t}e^{i(t-s)\Lap}(\Phi_{\rz_{\leq M}}[v]+[\rz, \rz, \rz]_{>M})\dd s,
\]
with $v_0=0$, where \(\Phi\) is defined in \eqref{eq:remainder-cubic-nonlinearity}.

\begin{proposition}\label{prop:v-solution}
Fix any \(M\in2\N+1\) and any \(0<S<\xs_{c}<\xs<\xs_{c}+\sfrac{(\sigma-1)}{2}\) that
satisfies $\xs \leq \mu(M+2, S)$. 
Choose any
\(0\leq\epsilon\lesssim_{M, S, \xs} 1\), and \(0<\epsilon_{0}\lesssim_{\epsilon, S, \xs, M} \!\!1\). For any initial datum
\(v_{0}\in H^{\xs+\epsilon}(\R^{d})\) and
\(f\in H^{S+\epsilon}(\R^{d})\) with corresponding randomization $\rf$ (as
defined in \eqref{eq:randomization}),  let \(\rz_{n}\) be as
defined in \eqref{eq:def:zn} and let \(v\mapsto\mc{I}(v)\) be as in
\eqref{eq:iteration-map}.

Then, there exist large constants \(C, c>0\) depending on 
$\epsilon, \epsilon_0,  S, \xs,$ and \(M\), such that the following hold\footnote{The $\epsilon_0$ dependence is implicit in the $X$ and $Y$ norms.}. 
\begin{description}
\item[Almost sure time-continuity of the correction term] Almost surely, if \(v\) is a fixed point of \(\mc{I}\) in \(X^{\xs}([0, T])\) for some \(T\in(0, \infty]\), then 
\[\eqnum\label{eq:v-classical-space}
v\in C_t^{0}\big([0, T], H_x^{\xs+\epsilon}(\R^{d})\big).
\]
\item[Almost sure uniqueness of the solution] Almost surely, any two fixed \\points \(v_{j}\in X^{\xs}([0, T_{j}])\), \(j\in\{1, 2\}\) of $\mc{I}$ coincide  on \(\big[0, \min(T_{1}, T_{2})\big]\).

\item[Almost sure local existence of solutions] 
Assume that a random variable $v_0$ satisfies almost surely 
$\|v_0\|_{H^{\xs + \epsilon}_x(\R^d)} < \infty$. Then,
there is a random variable $T\in(0, \infty]$ such that almost surely
\(\mc{I}\) has a unique fixed point on \(X^{\xs}([0, T])\). Furthermore,
\[\eqnum\label{eq:probability-of-time-of-existence}
\mathbb{P}\Big(  T <  \lambda  \Big)< C \exp \Big(-c \frac{\lambda^{-\sfrac{6}{cM}}}{\|f\|_{H_x^{S+\epsilon}(\R^{d})}^{2}} \Big)
+ \mathbb{P} \Big( \frac{\lambda^{-3/c}}{3} \leq   \|v_{0}\|_{H^{\xs+\epsilon}(\R^{d})}  \Big) \,.
\]

\end{description}
\end{proposition}
\begin{proof}[Proof of Proposition \ref{prop:v-solution}]
First, we show that if $v \in X^{\xs} (I)$, then we have that \(\mc{I}(v) \in C_t^{0}\big(I, H^{\xs+\epsilon}(\R^{d})\big)\). Indeed, by Corollary \ref{cor:linear-estimate-adjoint-and-continuity}
with $h = \Phi_{\rz_{\leq M}}[v]+[\rz, \rz, \rz]_{>M}$, 
we have
\[
\begin{aligned}
\|\mc{I}(v)\|_{C_t^0(I, H_x^{\xs +\varepsilon} (\R^d))} &
\lesssim \|v_0\|_{H^{\xs +\varepsilon} (\R^d)}
\\ & \qquad + \|\Phi_{\rz_{\leq M}}[v]\|_{X^{\ast , \xs +2\varepsilon} (I)}+\|[\rz, \rz, \rz]_{> M}\|_{X^{\ast, \xs +2\varepsilon}(I)} \,.
\end{aligned}
\]
Since $\xs \leq \mu(M+2, S)$, then  $\xs$ satisfies \eqref{eq:asoxs} and by
\eqref{eq:F-bound} we obtain
\[
 \big\|\Phi_{\rz_{\leq M}}[v]\big\|_{X^{\ast , \xs +2\varepsilon} (I)} \lesssim \big(\|v\|_{X^{\xs} (I)}+\|\rz_{\leq M}\|_{Y^S (\R)}\big)^2 \|v\|_{X^s (I)} .
\]
On the other hand, as a consequence of Proposition \ref{prop:tree-bound-probabilistic} and proceeding as in \cite[Lemma $6.3$]{CFU}, there is \(\Omega_{0} \subset \Omega\) with \(\mbb{P}(\Omega_{0})=1\) such that on $\Omega_{0}$
\[\eqnum\label{eq:almost-sure-multirandom-finite}
\big\|\rz_{\leq M}\big\|_{Y^{S}(\R)}+\big\|[\rz, \rz, \rz]_{>M}\big\|_{X^{*, \mu(M+2, S)}(\R)}< \infty.
\]
Also,  almost surely, it holds that 
\[
\big\|[\rz, \rz, \rz]_{> M}\big\|_{X^{\ast, \xs +2\varepsilon}(I)}\leq\|[\rz, \rz, \rz]_{> M}\|_{X^{\ast, \mu(M+2, S)}(I)}<\infty
\]
provided that \(\epsilon>0\) is sufficiently small such that \(\xs + 2\varepsilon<\mu(M+2, S)\). This proves that
 \(\mc{I}(v) \in C_t^{0}\big(I, H^{\xs+\epsilon}(\R^{d})\big)\).

\textbf{Almost sure uniqueness of the solution.}  For any two fixed points 
\(v_{j}\in X^{\xs}([0, T_{j}])\), \(j\in\{1, 2\}\)
of the map \(v\mapsto\mc{I}(v)\), set \(T_{\cap}=\min(T_{1}, T_{2})\) and \(I\eqd[0, T_{\cap}]\). 
Define a random variable
\[\eqnum\label{eq:uniqueness-contradiction-time}
t_{*}\eqd\sup\big(t\in[0, T_{\cap}] \st v_{1}(s)= v_{2}(s) \text{ for all } s\in[0, t] \big).
\]
Assume by contradiction that \(t_{*}<T_{\cap} \). By \eqref{eq:v-classical-space} we have that $v_{1}, v_{2}\in C_t^{0}\big(I,
H_x^{\xs+\epsilon}(\R^{d})\big)$.
Also, 
 \(v_{1}(s)=v_{2}(s)\) for \(s\in[0, t_{*}]\subset I\) and  we
define \(v_{0}\eqd v_{1}(t_{*})= v_{2}(t_{*})\in H^{\xs+\epsilon}(\R^{d})\). Define a random variable
\[\eqnum\label{eq:dfoko}
K_0 \eqd \max\big(\sup_{t \in I}\|v_1(t)\|_{H_x^{\xs+\epsilon}(\R^d)}, \sup_{t \in I}\|v_2(t)\|_{H_x^{\xs+\epsilon}(\R^d)}\big)<\infty\,.    
\]
From the group properties of the linear Schrödinger evolution \(t\mapsto
e^{it\Lap}\) for all \(t \in [t_{*}, T_{\cap}]\) we have that
\[
v_{j}(t)=e^{i(t-t_{*})\Lap}v_{0}\mp i\int_{t_{*}}^{t}e^{i(t-s)\Lap}(\Phi_{\rz_{\leq M}}[v_{j}] + [\rz, \rz, \rz]_{> M}) \, ds \qquad j = 1, 2\,.
\]
Using \(\mu(M+2, S)>\xs+\epsilon\), \eqref{eq:almost-sure-multirandom-finite}
 and the monotonicity of the norms \(Y^{S}(I)\) and \(X^{*, \mu(M+2, S)}(I)\) with respect to inclusion of time intervals, and  Lemma \ref{lem:iteration-map-bound}, we deduce that there exists a random variable \(T>0\) such that \(v_{1}(s)=v_{2}(s)\) for \(s\in\big[t_{*}, \min(t_{*}+T, T_{\cap})\big]\). This contradicts our assumption \(t_{*}<T_{\cap}\) and the desired assertion follows.

\textbf{Almost sure local existence of solutions.} Define the random variable $T = T_{\delta_0 = 1}$,
where $T_{\delta_0}$ is as in \eqref{eq:contraction-time-bound}. Note that since $[\rz, \rz, \rz]_{> M}$, and $\rz_{\leq M}$ are random variables and norms depend continuously on the argument, $T$ is also a random variable. Thanks to \eqref{eq:almost-sure-multirandom-finite}
and $\|v_0\|_{H^{\xs + \epsilon}_x(\R^d)} < \infty$ 
it holds that \(T>0\) and the existence of fixed point in \(X^{\xs}([0, T])\) follows from Lemma \ref{lem:iteration-map-bound}.

\end{proof}

\begin{proof}[Proof Of Theorem \ref{thm:main}]

The assertion follows from Proposition \ref{prop:v-solution} once we 
satisfy the assumption $\xs \leq \mu(\kappa+2, S)$, which, if we choose $\xs$ sufficiently close to $\xs_{c}$ follows from 
% \[
% \frac{d - \sigma}{2} \leq \min\Big\{ (\kappa + 2)S, S + \sigma - 1, 2S + \frac{3}{4}\sigma - 1\Big\} \,.
% \]

\[
\frac{d - \sigma}{2} < \min\Big\{ (\kappa + 2)S + \frac{(\kappa + 1)(\sigma-2)}{4}, S + \sigma - 1, 2S + \frac{3}{4}\sigma - 1\Big\} \,.
\]

Solving for $S$, we equivalently have 
% \[\eqnum\label{eq:lbons}
% S \geq \max \Big\{ \frac{d - \sigma}{2(\kappa + 2)}, \frac{d - 3\sigma + 2}{2}, \frac{2d - 5\sigma + 4}{8} \Big\} \,.
% \]

\[\eqnum\label{eq:lbons}
S > \max \Big\{ \frac{d - \sigma}{2(\kappa + 2)} - \frac{(\kappa + 1)(\sigma-2)}{4(\kappa + 2)}, \frac{d - 3\sigma + 2}{2}, \frac{2d - 5\sigma + 4}{8} \Big\} \,.
\]

Note that 
\[
\frac{d - 3\sigma + 2}{2} \geq \frac{2d - 5\sigma + 4}{8}
\qquad \textrm{if and only if} \quad d \geq \frac{7\sigma - 4}{2} \,.
\]
Also, as $\kappa \to \infty$, then the first term in the maximum converges to $\frac{2-\sigma}{4}$. It is standard to show that 
\[
\frac{d - 3\sigma + 2}{2} \geq \frac{2-\sigma}{4}
\qquad \textrm{if and only if} \quad d \geq \frac{5\sigma - 2}{2} 
\]
and 
\[
\frac{2d - 5\sigma + 4}{8} \geq \frac{2-\sigma}{4}
\qquad \textrm{if and only if} \quad d \geq \frac{3\sigma}{2} \,. 
\]
Since $\sigma \geq 2$ implies $\frac{7\sigma - 4}{2} \geq \frac{5\sigma - 2}{2} \geq \frac{3\sigma}{2}$, the assertion follows. 

% Thus if $d \leq \frac{5\sigma - 4}{2}$, then the last two terms in the maximum in \eqref{eq:lbons} are non-positive and \eqref{eq:lbons} is 
% equivalent to $S \geq \frac{d - \sigma}{2(\kappa + 2)}$. Otherwise, we obtain the desired conclusion. 

\end{proof}

\appendix\label{firstApp}

\section{The generalized Schrödinger maximal operator in 1 dimension} \label{sec:app-local}

Here we state and prove a global in space boundedness result for the
generalized 1D Schrödinger maximal operator.
\begin{proposition} \label{prop:local-shiraki} Let $\sigma\geq 2$ and
$R>1$, and suppose that for some \(C_0,C_{L}\geq 1\) the function
$m_{L} \in C^2 \big( (R,C_0 R) ; \R \big)$ satisfies the bounds
  \[\eqnum\label{eq:local-shiraki-multiplier-conditions} 
  \begin{aligned}[t]
       & C_{L}^{- 1} | \xi |^{\sigma-1} \leq \left|  \partial m_{L} (\xi)  \right| \leq C_{L} | \xi |^{\sigma-1},
       \\
       & C_{L}^{- 1} | \xi |^{\sigma-2} \leq \left| \partial^2 m_{L} (\xi)  \right| \leq C_{L} | \xi |^{\sigma-2},
     \end{aligned} 
     \]
  for all $\xi$ with $\xi \in (R, C_0 R)$.
  
  Then, for any $f\in L^{2}(\R)$ with $\spt(\FT{f}) \subset(R, C_0 R)$, and for any $s>\frac{\sigma}{4}$ it holds that
  \[\eqnum\label{eq:conclusion1-29may24}
      \Big\| \sup_{t \in [- 1, 1]} e^{i t L } f \Big\|_{L^2( \R )} \lesssim_{s} R^{s} \| f\|_{L^2( \R )},
  \]
  with an implicit constant depending only on  $s$, $\sigma$,
  $C_0$, and $C_L$, but not on $R$ or $f$.

  The same claim holds if \(R<-1\) and the interval \((R,C_{0}R)\) is
  replaced by \((C_{0}R,R)\).
\end{proposition}

The proof of this statement relies on a local-in-space version
obtained in \cite[Theorem
1]{shirakiPointwiseConvergenceRestricted2020}, which we state as
\Cref{thm:Shiraki1} below. Next, a local-to-global argument due to
\cite{rogersLocalSmoothingEstimate2008}, which we state as \Cref{thm:rogers1}, allows us to
obtain the global-in-space 1D Schrödinger maximal operator bounds
stated in \Cref{cor:shiraki-and-rogers-global}. To prove
\Cref{prop:local-shiraki} it remans to show that a symbol defined
locally on \((R,C_{0}R)\) can be extended to a symbol on \(\R\)
satisfying the conditions of \Cref{thm:Shiraki1}. This procedure is
illustrated in \Cref{lem:multiplier-extension-lemma-step1}.

Next we state the two theorems from \cite{shirakiPointwiseConvergenceRestricted2020} and \cite{rogersLocalSmoothingEstimate2008} discussed above, then we state and prove \Cref{lem:multiplier-extension-lemma-step1}, and we conclude by deducing \Cref{prop:local-shiraki}.
\todoI{There is another point I'd like to make: in the appendices, the time intervals are always $[-1,1]$, whereas in the text we use a generic $I$. This won't make a difference in the end, but we should address it sometime.}.

\begin{theorem}[{\cite[Theorem 1]{shirakiPointwiseConvergenceRestricted2020}}] \label{thm:Shiraki1} 
Let $d=1$ and suppose that $\mLap\in C^{2}(\R)$ satisfies 
satisfies
\[\eqnum\label{eq:Shiraki-symbol-conditions}
\begin{aligned}[t]
|\partial^{2}\!\mLap(\xi)| &\geq C_{L}|\xi|^{-1}, \quad\forall |\xi|\geq 1
\\
|\partial^{2}\!\mLap(\xi)| &\geq C_{L}|\xi|^{-1}|\partial\!\mLap(\xi)|, \quad\forall |\xi|\geq 1.
\\
\end{aligned}
\]
for some $C_{L}\geq 0$. Then, for any $q \in [1, 4]$ and $s>\frac{1}{4}$, there is a constant $C_{q, s}$ such that
\begin{equation*}
    \Big\| \sup_{t\in [-1, 1]}\big(e^{it\Lap}f\big) \Big\|_{L^{q}({B}_{1}(0))} 
    \leq C_{q, s} \Big\| f\Big\|_{H^{s}(\R)}.
\end{equation*}
\end{theorem}

\begin{theorem}[{\cite[Theorem 13]{rogersLocalSmoothingEstimate2008}}]
\label{thm:rogers1}
Let $q, r\geq 2$. Suppose that $|\partial^{\alpha}\mLap(\xi)|\leq C_{0}|\xi|^{\sigma-|\alpha|}$ and $|\nabla\mLap(\xi)|\geq C_{0}^{-1}|\xi|^{\sigma-1}$ for all $\xi\in\R^{d}\backslash\{0\}$, where $|\alpha|\leq 2$ and $\sigma>1$. Then
\begin{equation*}
    \Big\|\big(e^{it\Lap}f\big)(x)\Big\|_{L^{q}_{x}(\mathbb{B}^{d}, L^{r}_{t}[0, 1])}\leq C_{s}\|f\|_{H^{s}(\R^{d})}
\end{equation*}
holds for all $s>s_{0}$ if and only if
\begin{equation*}
    \Big\|\big(e^{it\Lap}f\big)(x)\Big\|_{L^{q}_{x}(\R^{d}, L^{r}_{t}[0, 1])}\leq C_{s}\|f\|_{H^{s}(\R^{d})}
\end{equation*}
for all $s>\sigma s_{0}-(\sigma-1)\Big(d(\frac{1}{2}-\frac{1}{q})-\frac{\sigma}{r}\Big)$.
\end{theorem}

In this section we prove a local version of \Cref{thm:Shiraki1}. 

\begin{corollary}\label{cor:shiraki-and-rogers-global} Let $d=1$ and suppose that $\mLap\in C^{2}(\R)$ satisfies
\[\eqnum\label{eq:shiraki-rogers-global-condition} 
   \begin{aligned}[t]
       & C_{L}^{- 1} | \xi |^{\sigma-1} \leq \left| \partial_{\xi} \mLap (\xi)
       \right| \leq C_{L} | \xi |^{\sigma-1},
       \\
       & C_{L}^{- 1} | \xi |^{\sigma} \leq \left| \xi^2 \partial_{\xi}^2 \mLap (\xi)
       \right| \leq C_{L} | \xi |^{\sigma},
     \end{aligned} 
\]
for all $\xi \in \R$. Then
\[\eqnum\label{ineq1-03june24}
    \Big\| \sup_{t\in [-1, 1]}\big(e^{it \Lap}f\big) \Big\|_{L^{2}(\R)}\leq C_{s}\|f\|_{H^{s}(\R^{d})}
\]
for all $s>\frac{\sigma}{4}$.
\end{corollary}

\begin{proof} By \Cref{thm:Shiraki1},
\begin{equation*}
    \Big\| \sup_{t\in [-1, 1]}\big(e^{it\Lap}f\big) \Big\|_{L^{q}({B}_{1}(0))} 
    \leq C_{q, s} \Big\| f\Big\|_{H^{s}(\R)}.
\end{equation*}
for any $q \in [1, 4]$ and $s>\frac{1}{4}$. The conditions \eqref{eq:shiraki-rogers-global-condition} verify the hypotheses of \Cref{thm:rogers1} with $q=2, r=\infty$, $s_0 = \frac{1}{4}$ and \eqref{ineq1-03june24} follows.
\end{proof}

The proof of \Cref{prop:local-shiraki} is essentially reduced to the following extension Lemma. 

\begin{lemma}\label{lem:multiplier-extension-lemma-step1}
Let $\sigma$, $R$, $C_0, C_L$ be as in \Cref{prop:local-shiraki}. Suppose
that $L \in C^2 \big( (R, C_0 R) ; \R \big)$ satisfies
\eqref{eq:local-shiraki-multiplier-conditions} for all
$\xi \in (R, C_0 R)$. Suppose, furthermore, that
$\xi \partial m_{L} (\xi)$ and $\xi^2 \partial^2 m_{L}(\xi)$ have the same sign for each
$\xi \in (R, C_0 R)$.
  
Then there exists $\mLap \in C^2( \R ; \R )$ such that
\begin{itemize}
\item $\mLap (\xi) = m_{L} (\xi)$ for $\xi \in (R, C_0R)$;
    
\item there exists $\tilde{C}_{L} > 0$, depending only on $\sigma > 0$, $C_0$, and 
$C_L$, but not on $R$ or $m_{L}$, such that
\[\eqnum\label{eq:eulbol}
\begin{aligned}[t]
\tilde{C}_{L}^{- 1} | \xi |^{\sigma-1} &\leq \left|  \partial \!\mLap (\xi) \right| \leq \tilde{C}_{L} | \xi |^{\sigma-1}
\\
\tilde{C}_{L}^{- 1} | \xi |^{\sigma-2} &\leq \left| \partial^2 \!\mLap (\xi) \right| \leq \tilde{C}_{L} | \xi |^{\sigma-2}
\end{aligned}
\] 
for all $\xi \in \R$.
\end{itemize}
\end{lemma}

We postpone the proof of this
\Cref{lem:multiplier-extension-lemma-step1} to the end of the
section. The idea of the proof is to extend $m_{L}$ to satisfy
$\mLap(\xi)=\rho_- |\xi|^\sigma$ for $\xi\ll R$ and
$\mLap(\xi)=\rho_+ |\xi|^\sigma$ for $\xi\gg C_0 R$ for some appropriate
$\rho_-, \rho_+$, using a variant of the standard smooth extension argument.

Next, assuming \Cref{lem:multiplier-extension-lemma-step1}, we prove
\Cref{prop:local-shiraki}.

\begin{proof}[Proof of \Cref{prop:local-shiraki}]\todoGU{This needs to be streamlined, given the revised statement of the lemma}
First, let us prove \Cref{prop:local-shiraki} under additional assumptions. 

\noindent
\textit{Claim (1).} The assertion of the proposition holds if  $\partial m_{L}(\xi)$ and $\partial^2 m_{L}(\xi)$ have the same constant sign on $(R, C_0 R)\subset (0,+\infty)$ and the support of $\FT{f}$ is contained in $(R, C_0 R)\subset(0,+\infty)$. 

\begin{proof}[Proof of Claim (1).]
Using \Cref{lem:multiplier-extension-lemma-step1} we extend $m_{L}$  to $\mLap\in C^2(\R;\R)$. Since $m_{L}(\xi)=\mLap(\xi)$ for $\xi\in\spt{\FT{f}}$, it holds that $e^{i t L } f(x)=e^{i t \Lap } f(x)$. Furthermore, $\xi\mapsto \mLap(\xi)$ satisfies \eqref{eq:shiraki-rogers-global-condition},
 then by \Cref{cor:shiraki-and-rogers-global} we have that 
\[
   \Big\| \sup_{t\in [-1, 1]}\big(e^{it L}f\big) \Big\|_{L^{2}(\R)}
    = 
    \Big\| \sup_{t\in [-1, 1]}\big(e^{it \Lap}f\big) \Big\|_{L^{2}(\R)}
    \lesssim_s  R^s \| f \|_{L^2(\R)} \,,
\]
where in the last inequality we used the definition of $H^s$ and 
$\spt (f) \subset [R, C_0R]$. 
\end{proof}

Next, we drop one hypothesis of Claim (1).

\textit{Claim (2).}
The assertion of the proposition holds if the support of $\FT{f}$ is contained in $(R, C_0 R)\subset(0,+\infty)$. 

\begin{proof}[Proof of Claim (2).]
Note that by \eqref{eq:local-shiraki-multiplier-conditions},  $\partial m_{L}(\xi)$ and $\partial^2 m_{L}(\xi)$ cannot vanish on $(0,+\infty)$ and thus it has constant sign, so let us assume that $\partial m_{L}(\xi)$ and $\partial^2 m_{L}(\xi)$ have opposite signs on $(R, C_0 R)$. We set $\tilde{m_{L}}(\xi)=-m_{L}(-\xi+ 2C_0 R)$ and obtain that for the symbol $\tilde{m_{L}}$ defined on $\big(C_0 R, (2C_0 - 1)R\big)\subset(0,+\infty)$, the derivatives $\partial\tilde{m_{L}}$ and $\partial^{2}\tilde{m_{L}}$ have the same sign, which is precisely the case studied above. Furthermore, $\tilde{m_{L}}$ satisfies the hypotheses of \Cref{prop:local-shiraki} with $C_0 R$ and $\tilde{C}_0 := (2C_0 - 1)R$ instead of respectively $R$ and $C_0 R$ and with appropriate $\tilde{C}_1>1$ and $\tilde{C}_2>1$ depending on $C_0, C_1, C_2$. 

It holds that 
\begin{equation*}
\begin{aligned}[t]
e^{it L}f(x)& =\int_{R}^{C_0 R}e^{ix\xi}e^{it L(\xi)}\FT{f}(\xi)\dd\xi
\\
&=\int_{C_0 R}^{(2C_0-1)R} e^{ix(-\eta+2C_0 R)}e^{it L(-\eta+2C_0 R)}\FT{f}(-\eta+2C_0R)\dd\eta
\\
&= e^{2ixC_0R}\bar{\int_{\R}e^{ix\eta}e^{it \tilde{L}(\eta)}\FT{g}(\eta)\dd\eta},
\end{aligned}
\end{equation*}
where $\FT{g}(\eta)=\bar{\FT{f}(-\eta+2C_0R)}$. Clearly 
\(\|g\|_{L^2(\R)}=\|f\|_{L^2(\R)}\) and $\FT{g}$ is supported on $\big(C_0R,(2-C_0^-1)C_0R\big)$ so, by Claim (1), we obtain that 
\begin{align*}
\Big\| \sup_{t\in [-1, 1]}\big(e^{it L}f\big) \Big\|_{L^{2}(\R)}
    &= 
    \Big\| \sup_{t\in [-1, 1]}\big(e^{it \tilde{L}}g\big) \Big\|_{L^{2}(\R)}
    \lesssim_s  (C_0R)^s \| g \|_{L^2(\R)} \lesssim_s  R^s \| f \|_{L^2(\R)},    
\end{align*}
as desired.
\end{proof}

Finally, let us show that compared to Claim (2) we can relax the assumption $\spt \FT{f}\subset(0,+\infty)$. Indeed, let 
\[
\FT{f}_+(\xi) \eqd \FT{f}(\xi) \1_{(R, C_0 R)}(\xi),
\qquad 
\FT{f}_-(\xi) \eqd \bar{\FT{f}(-\xi)} \1_{(R, C_0 R)}(\xi)
\]
so that 
\[
\begin{aligned}
& \FT{f}(\xi)=\FT{f}_+(\xi)+\bar{\FT{f}_-(-\xi)},
\\ & \spt{\FT{f}_+}, \spt{\FT{f}_-}\subset (R, C_{0}R),
\\ &
\| f_+ \|_{L^2( \R )}+\| f_- \|_{L^2( \R )}\leq 2 \|f\|_{L^2(\R)}.
\end{aligned}
\]
It follows that 
\[
e^{it L}f(x) \begin{aligned}[t]
& =\int_{\R}e^{ix\xi}e^{it L(\xi)}\FT{f}(\xi)\dd\xi
\\
& = \int_{\R}e^{ix\xi}e^{it L(\xi)}
\Big(\FT{f}_+(\xi)+\bar{\FT{f}_-(-\xi)}\Big) \dd\xi
\\
& = \int_{R}^{C_0R} e^{ix\xi} e^{it L(\xi)}
\FT{f}_+(\xi) \dd\xi
 - \bar{\int_{R}^{C_0R} e^{ix\xi} e^{it -L(-\xi)}
\FT{f}_-(\xi) \dd\xi}
\\
&
=e^{it L}f_+(x)-\bar{e^{it L}f_-(x)}.
\end{aligned}
\]
\begin{comment}
Since $\xi\mapsto -L(-\xi)$ satisfies the conditions \eqref{eq:local-shiraki-multiplier-conditions}, by Claim (2), we have that\todoI{Where does $\varepsilon_0$ come from?}
\[
\left\| \sup_{t \in [- 1, 1]} e^{i t L } f \right\|_{L^2( \R )} \lesssim_s (1 + \varepsilon_0)^{s} 
\Big(\| f_+ \|_{L^2( \R )}+\| f_- \|_{L^2( \R )}\Big) 
\lesssim_s R^{s}  \| f \|_{L^2( \R )},
\]
as required.
\end{comment}

Since $\xi\mapsto -L(-\xi)$ satisfies the conditions \eqref{eq:local-shiraki-multiplier-conditions}, by Claim (2), we have that
\[
\left\| \sup_{t \in [- 1, 1]} e^{i t L } f \right\|_{L^2( \R )} \lesssim_s \Big(\| f_+ \|_{L^2( \R )}+\| f_- \|_{L^2( \R )}\Big) 
\lesssim_s R^{s}  \| f \|_{L^2( \R )},
\]
as required. 
\end{proof}

We conclude this section by proving \Cref{lem:multiplier-extension-lemma-step1}. 

\begin{proof}[Proof of \Cref{lem:multiplier-extension-lemma-step1}] %\todoGU{Thereare 3 proofs of this lemma.This one is the shortest one (though maybe a bit opaque). If it is correct I would keep this one, mainly because of length. }
We prove the lemma when $\partial m_{L}(\xi)$ and $\partial^2 m_{L}(\xi)$ are both positive on \((R, C_{0} R)\).
If $\partial m_{L}(\xi)$ and $\partial^2 m_{L}(\xi)$ are both negative, it is sufficient to repeat the proof with $L$ replaced by $-L$. By scaling, we may assume that \(R=1\). Indeed, given \(m_{L}\) as in \Cref{lem:multiplier-extension-lemma-step1} the function \(m_{L_{R}}(\xi)\eqd|R|^{-\sigma}m_{L}(R\xi)\) satsifies the hypotheses of \Cref{lem:multiplier-extension-lemma-step1} with \(R=1\). Hence, assuming that the lemma holds for $R = 1$, we use it to obtain the extensions \(\mc{L}_{R}\) and and set \(\mc{L}(\xi)=|R|^{\sigma}\mc{L}_{R}(R^{-1}\xi)\) to obtain the full claim. Henceforth, we suppose \(R=1\).

Fix any \(m>(2C_{1}C_{2})^{-1}\) and set\todoI{Are there typos in the second and fourth lines of this definition? Specifically, should $h_{2}^{+}$ be $h_{2}^{-}$ in the second line? Should $h_{2}^{+}$ feature in the fourth line?}
\[
\mc{L} (\xi) \eqd \begin{dcases}
a_{-}+b_{-}|\xi|^{\sigma}
& \text{if } \xi<0
\\
a_{-}+b_{-}h_{1}^{-}(\xi)+c_{-}h_{2}^{+}(\xi)
& \text{if } 0\leq \xi <1
\\
L (\xi) & \text{if } 1\leq\xi< C_0
\\
a_{+}+b_{+}h_{1}^{+}(\xi/C_{0})+c_{+}h_{1}^{+}(\xi/C_{0})
& \text{if } \xi\geq C_{0} 
 \end{dcases}
 \]
 with
 \[
 \begin{aligned}[t]
 &
 h_{1}^{-}(\xi)\eqd |\xi|^{\sigma}-\frac{\sigma-1}{\sigma+1}|\xi|^{\sigma+1}
 \qquad
 h_{2}^{-}(\xi)\eqd (\sigma+m+1)^{-1}(\sigma+m)^{-1}|\xi|^{\sigma+1+m}
 \\
 &
 h_{1}^{+}(\xi)\eqd \frac{1}{\sigma} |\xi|^{\sigma}-\frac{\sigma-1}{\sigma(\sigma+1)}|\xi|^{-\sigma}
 \qquad
 h_{2}^{+}(\xi)\eqd \frac{1}{\sigma} |\xi|^{\sigma}+\frac{1}{\sigma}|\xi|^{-\sigma}.
 \end{aligned}
 \]
It holds that 
\[\begin{aligned}
&\begin{pmatrix}
\partial h_{1}^{-}(1) & \partial h_{2}^{-}(1)
\\
\partial ^2h_{1}^{-}(1) & \partial^2 h_{2}^{-}(1)
\end{pmatrix}
=
\begin{pmatrix}
1 & (m+\sigma)^{-1}
\\
0 & 1
\end{pmatrix}
\\
&
\begin{pmatrix}
\partial h_{1}^{+}(1) & \partial h_{2}^{+}(1)
\\
\partial^2 h_{1}^{+}(1) & \partial^2 h_{2}^{+}(1)
\end{pmatrix}
=
\begin{pmatrix}
\frac{2\sigma}{\sigma+1} & 0
\\
0 & 2\sigma
\end{pmatrix}\,.
\end{aligned}
\]
Thus, \(\mc{L}(\xi)\), \(\partial\mc{L}(\xi)\), and \(\partial^{2}\mc{L}(\xi)\) are continuous at \(1\) and \(C_{0}\) if and only if the following conditions hold:
\[
\begin{pmatrix}
b_{-}\\c_{-}
\end{pmatrix}
=\begin{pmatrix}
1 & -(m+\sigma)^{-1} \\
0 & 1
\end{pmatrix}
\begin{pmatrix}
\partial L(1)\\ \partial^{2} L(1)
\end{pmatrix},
\]
\[
\begin{pmatrix}
b_{+}\\c_{+}
\end{pmatrix}
=\begin{pmatrix}
\frac{\sigma+1}{2\sigma}  C_{0}& 0 \\
0 & \frac{1}{2\sigma }C_{0}^{2}
\end{pmatrix}
\begin{pmatrix}
\partial L(1)\\ \partial^{2} L(1)
\end{pmatrix},
\]
and \(a_{-}\) and \(a_{+}\) determined by the relations
\(a_{-}+b_{-}h_{1}^{-}(1)+c_{-}b_{-}h_{2}^{-}(1)=L(1)\) and \(
a_{+}+b_{+}h_{1}^{+}(1)+c_{+}h_{2}^{+}(1)=L(C_{0})\).

As long as one takes \(m>(2C_{1}C_{2})^{-1}\) from \eqref{eq:local-shiraki-multiplier-conditions} one obtains
\[
\begin{aligned}[t]
&
(2C_{1})^{-1}\leq b_{-}\leq C_{1}
\qquad
C_{2}^{-1}\leq c_{-}\leq C_{2}
\\
&
\frac{(\sigma+1)C_{0}}{2\sigma C_{1}}
\leq b_{+}
\leq \frac{(\sigma+1)C_{1}C_{0}}{2\sigma}
\qquad
\frac{C_{0}^{2}}{2\sigma C_{2}}\leq c_{+}\leq \frac{C_{2}C_{0}^{2}}{2\sigma} \,.
\end{aligned}
\]
In particular, all coefficients are bounded from above, non-negative, and bounded away from \(0\). This way, there exists \(C_{4}>1\) so that 

%Since\todoI{The rest of the sentence is missing here.}, exists \(C_{4}>1\) so that 
\[
\begin{aligned}[t]
0\leq \xi^{l} \partial^{l} h_{j}^{-}(\xi)
\leq C_{4}\xi^{\sigma}
\quad &\text{ when } \xi\in[0, 1], \;j, l\in\{1, 2\},
\\
C_4^{-1}\xi^{\sigma}
\leq \xi^{l} \partial^{l}\Big( h_{1}^{-}(\xi)+h_{2}^{-}(\xi)\Big)
\leq C_{4}\xi^{\sigma}
\quad &\text{ when } \xi\in[0, 1], \;l\in\{1, 2\},
\\
0\leq \xi^{l} \partial^{l} h_{j}^{+}(\xi)
\leq C_{4}\xi^{\sigma}
\quad &\text{ when } \xi\geq 1, \;j, l\in\{1, 2\},
\\
C_4^{-1}\xi^{\sigma}
\leq \xi^{l} \partial^{l}\Big( h_{1}^{+}(\xi)+h_{2}^{+}(\xi)\Big)
\leq C_{4}\xi^{\sigma}
\quad &\text{ when } \xi\geq 1, \;l\in\{1, 2\},
\end{aligned}
\]
the desired claim follows.

\end{proof}

\newpage
{\sloppy \printbibliography}

\end{document}